%% file: nilcox.tex
\documentclass [11pt]{amsart}
\usepackage {amsmath, amssymb, amscd, mathrsfs, graphicx, color, url,epsf, mathtools, type1cm, enumerate, caption}
\usepackage[all,knot,arc,dvips]{xy}
\usepackage[pdftex,lmargin=1in,rmargin=1in,tmargin=1in,bmargin=1in]{geometry}

\newtheorem {theorem}{Theorem}[section]
\newtheorem {lemma}[theorem]{Lemma}
\newtheorem {proposition}[theorem]{Proposition}
\newtheorem {corollary}[theorem]{Corollary}

\newtheorem {definition}[theorem]{Definition}

\theoremstyle{remark}
\newtheorem {remark}[theorem]{Remark}

\allowdisplaybreaks

\renewcommand\th{^{\text{th}}}
\newcommand\st{^{\text{st}}}
\newcommand\Xs{\mathbb{X}}
\newcommand\Os{\mathbb{O}}
\newcommand\nid{\noindent}
\def\k {\mathbf{k}}
\def\kk {\Bbbk}
\def\S {\mathbf{S}}
\def\Sym{\mathrm{Sym}}

\newcommand\Rect{\mathrm{Rect}^\circ}
\newcommand\Rectx{\mathrm{Rect}^\circ_{\X}}
\newcommand\Halfx{\mathrm{Half}^\circ_{\X}}
\newcommand\Quartx{\mathrm{Quart}^\circ_{\X}}
\newcommand\Tri{\mathrm{Tri}}

\def\zz {{\mathbb{Z}}}
\def\rr {{\mathbb{R}}}
\def\cc {{\mathbb{C}}}

\def\x {\mathbf{x}}
\def\y {\mathbf{y}}
\def\z {\mathbf{z}}
\def\X {\mathbb{X}}
\def\O {\mathbb{O}}

\def\nil {\mathfrak{N}}

\def\B {\mathcal{B}}
\def\H{\mathcal{H}}
\def\T {\mathcal{T}}
\def\L {\mathcal{L}}
\def\R {\mathcal{R}}

\def\del {{\partial}}

\def\I {\mathcal{I}}
\def\alg {\mathcal{A}}
\def\A {\mathfrak{A}}
\def\Filt {\mathcal{F}}
\def\zed {\mathcal{Z}}
\def\J {\mathcal{J}}
\def\a {\mathbf{a}}
\def\oast {\circledast}
\def\cpa{CPA^-}
\def\cpd{CPD^-}
\def\cpaa {CP\{AA\}^-}
\def\cpad {CP\{AD\}^-}
\def\cpda {CP\{DA\}^-}
\def\cpdd {CP\{DD\}^-}
\def\ha {\H^A}
\def\hd {\H^D}
\def\haa {\H^{\{AA\}}}
\def\had {\H^{\{AD\}}}
\def\hda {\H^{\{DA\}}}
\def\hdd {\H^{\{DD\}}}
\def\cpaah {\cpaa(\haa)}
\def\cpadh {\cpad(\had)}
\def\cpdah {\cpda(\hda)}
\def\cpddh {\cpdd(\hdd)}

\def\xaa {\x^{\{AA\}}}
\def\yaa {\y^{\{AA\}}}
\def\xad {\x^{\{AD\}}}
\def\yad {\y^{\{AD\}}}
\def\xdd {\x^{\{DD\}}}
\def\ydd {\y^{\{DD\}}}
\def\xda {\x^{\{DA\}}}
\def\yda {\y^{\{DA\}}}
\def\Oaa {\O^{\{AA\}}}
\def\Xaa {\X^{\{AA\}}}
\def\cro {\operatorname{cr}}
\def\eps {\varepsilon}

\def\ker {{\operatorname{Ker}}}
\def\im {{\operatorname{Im}}}
\def\inv {{\operatorname{inv}}}
\def\Inv {{\operatorname{Inv}}}

\def\id  {{\operatorname{id}}}
\def\gr  {{\operatorname{gr}}}
\def\gl {\mathfrak{gl}}
\def\ol {\overline}
\def\ax {\phantom{.}^A \x}
\newcommand\aalpha{\mbox{\boldmath$\alpha$}}
\newcommand\bbeta{\mbox{\boldmath$\beta$}}
\def\bs{\backslash}
\def\nin{\notin}
\def\ra{\rightarrow}

\def\cc{\mathbf{c}}
\def\cI{\mathcal{Z}}
\def\nid{\noindent}

\definecolor{CMcolor}{rgb}{0.0,0.5,0.75}	
\definecolor{CDcolor}{rgb}{0.8,0.0,0.2}

\begin{document}

\title{On the algebra of cornered Floer homology}

\author[Christopher L. Douglas]{Christopher L. Douglas}
\address {Mathematical Institute, University of Oxford, 24-29 St. Giles'\\ Oxford, OX1 3LB, UK}
\email{cdouglas@maths.ox.ac.uk}

\author[Ciprian Manolescu]{Ciprian Manolescu}
\address {Department of Mathematics, UCLA, 520 Portola Plaza\\ 
Los Angeles, CA 90095, USA}
\email {cm@math.ucla.edu}

\begin{abstract}Bordered Floer homology associates to a parametrized oriented surface a certain differential graded algebra. We study the properties of this algebra under splittings of the surface. To the circle we associate a differential graded 2-algebra, the nilCoxeter sequential 2-algebra, and to a surface with connected boundary an algebra-module over this 2-algebra, such that a natural gluing property is satisfied. Moreover, with a view toward the structure of a potential Floer homology theory of 3-manifolds with codimension-two corners, we present a decomposition theorem for the Floer complex of a planar grid diagram, with respect to vertical and horizontal slicing.
\end{abstract}

\maketitle

\tableofcontents

\section {Introduction}

Much progress in low-dimensional topology has come from the study of topological quantum field theories (TQFT's) and their related invariants. Roughly speaking, a $(d+1)$-dimensional TQFT is a functor from the bordism category of  $d$-dimensional smooth closed manifolds (and smooth cobordisms between them) to the category of vector spaces (and linear maps) \cite{AtiyahTQFT, WittenTQFT}. In dimension $3+1,$ a  theory of particular interest is Heegaard Floer homology, introduced by Ozsv\'ath and Szab\'o  \cite{HolDisk, HolDiskTwo, HolDiskFour}. Strictly speaking, Heegaard Floer homology is a TQFT only in a loose sense, because it is restricted to connected manifolds and cobordisms. Nevertheless, its TQFT properties have been exploited in \cite{HolDiskFour} to produce invariants for closed, smooth four-manifolds that are able to detect exotic smooth structures, and are conjecturally isomorphic to the Seiberg-Witten invariants of \cite{SW1, SW2, Witten}. It should be noted that Heegaard Floer homology comes in several variants, the simplest of which is called {\em hat}. While hat Heegaard Floer theory does not give rise to interesting invariants for closed 4-manifolds, it is sufficient for most three-dimensional applications, and easier to compute than the other versions.

Recently, there has been a surge of interest in extending TQFT's down in dimension, see for example \cite{BaezDolan, LurieTQFT, BDH, FHLT}.  Ideally, if one has a $(d+1)$-dimensional TQFT, one would like to construct a related {\em local field theory}. This would consist of (categorical) invariants of $i$-dimensional manifolds with codimension-$k$ corners, for all $k \leq i \leq d+1$, such that all natural gluing relations are satisfied. In the context of hat Heegaard Floer theory, an important step in this direction was taken by Lipshitz, Ozsv\'ath and Thurston in \cite{LOT, LOTbimodules, LOTmorphisms}, where they introduced {\em bordered Floer homology}. To a suitable handle decomposition of a closed, oriented, connected surface $F$ they associate a differential graded algebra $\alg(F),$ and to a closed, oriented $3$-manifold $Y$ with parametrized boundary $F$ they associate two $\alg_{\infty}$-modules over $\alg(F)$, denoted $\widehat{CFA}(Y)$ and $\widehat{CFD}(Y)$. They then prove a pairing theorem for the Heegaard Floer homology of a closed oriented $3$-manifold cut into two pieces along a surface. The algebra $\alg(F)$ depends on the chosen parametrization (that is, handle decomposition) of $F$. However, in \cite{LOTbimodules} it is proved that the derived category of dg-modules over $\alg(F)$ is an invariant of $F$, up to equivalence. We refer to $\alg(F)$ as the {\em matching algebra}.

The next step in the program of constructing a local Heegaard Floer theory would be to define suitable invariants (``cornered Floer homology'') for $3$-manifolds with codimension-two corners. Our aim is to present some partial results in this direction. Although constructing cornered Floer homology is beyond the scope of the paper, we hope that our constructions will shed some light on the algebraic form that cornered Floer homology should take. 

Our first goal is to propose a candidate for what local Heegaard Floer theory should associate to the circle. The answer takes the form of a differential graded {\em 2-algebra}, that is, a chain complex together with two multiplication operations, $*$ and $\cdot,$ satisfying certain compatibility relations. In particular, a weak form of commutation between $*$ and $\cdot$ is required. The 2-algebra we will consider is called the {\em nilCoxeter sequential 2-algebra}. (A signed version of this 2-algebra was  used by Khovanov in \cite{Khovanov12}, for the categorification of the positive half of quantum $\gl(1|1)$.) Concretely, the nilCoxeter sequential 2-algebra is given as the direct product of all nilCoxeter algebras $\nil_m$, where $\nil_m$ is the unital algebra generated (over a field of characteristic two) by elements $\sigma_i, \ i=1, \dots, m-1,$ subject to the relations:
\begin {eqnarray}
 \sigma_i^2 &=& 0, \label{eq:nc1} \\
 \sigma_i \sigma_j &=& \sigma_j \sigma_i  \ \ \text{for } \ |i - j | > 1,\label{eq:nc2} \\
 \sigma_i \sigma_{i+1} \sigma_i &=& \sigma_{i+1} \sigma_i \sigma_{i+1}. \label{eq:nc3}
\end {eqnarray}

\nid The product of the multiplications of the algebras $\nil_m$ induces the operation $\cdot$ on their direct product $\nil.$ The second multiplication $*$ comes from the natural concatenation maps $* : \nil_m \otimes \nil_n \to \nil_{m+n}.$ There is also a differential $\del$, which is induced by requiring that $\del \sigma_i = 1$ for all $i$.

With the 2-algebra $\nil$ being associated to the circle, we can construct invariants for surfaces with circle boundary. Indeed, to a compact, connected, oriented surface $F$ with oriented boundary $S^1$, equipped with a suitable handle decomposition, one can associate a differential graded algebra-module $\T(F)$ over $\nil$. Here, by an algebra-module we mean a left module over $\nil$ with respect to the multiplication $\cdot$, such that $\T(F)$ also admits an algebra structure with respect to an operation $*$. If the oriented boundary of $F$ is $-S^1,$ one can similarly associate to $F$ an algebra-module $\B(F)$, which is a right module over $\nil$ with respect to the $\cdot$ multiplication. Furthermore, an algebra-module of $\T$-type and an algebra-module of $\B$-type can be tensored together with respect to the multiplication $\cdot$ to obtain a differential graded algebra---the algebra structure is induced by the leftover $*$ multiplication. The symbol we use for this kind of tensoring is $\odot.$ 

We then have the following decomposition result for the matching algebra $\alg(F)$, that is for the bordered Floer homology of a surface:

\begin {theorem}
\label {thm:main1}
Suppose we have two compact, connected, oriented surfaces $F_1$ and $F_2$ (equipped with suitable handle decompositions), with $\del F_1 = -\del F_2 = S^1.$ We can glue them to form a closed surface $F=F_1 \cup F_2.$ The matching algebra of $F$ then be recovered as a tensor product over the nilCoxeter sequential 2-algebra:
$$ \alg(F) \cong \T(F_1) \odot_{\nil} \B(F_2).$$
\end {theorem}

\nid Theorem~\ref{thm:main1} indicates why $\nil$ should be associated to a circle in a local Heegaard Floer theory. It also indicates the form of the invariants of 2-dimensional manifolds with connected boundary.

With respect to $3$-manifolds with codimension-two corners, our expectation is that their invariants take the form of certain $\alg_{\infty}$-2-modules over the algebra-modules associated to their boundary components (which are surfaces with boundary). The goal of the second part of the paper is to present some evidence for this expectation, in a simplified model where many of the topological, analytical, and algebraic complications from bordered Floer homology disappear. 

The model we have in mind is that of planar grid diagrams. These have been previously used by Lipshitz, Ozsv\'ath, and Thurston in the study of bordered Floer homology \cite{LOTplanar}. Planar grid diagrams are a simplification of the toroidal grid diagrams for links in $S^3.$ The toroidal grid diagrams have played an important role in the development of combinatorial techniques in Heegaard Floer theory \cite{MOS, MOST, MO, MOT}. Whereas toroidal grid diagrams can be thought of as a particular kind of multi-pointed Heegaard diagrams  for link complements, the planar ones are not true Heegaard diagrams and, in fact, the Floer complexes built from them do not produce topological invariants. Nevertheless, these planar Floer complexes share many properties with the Floer complexes coming from Heegaard diagrams,  and are therefore a good model in which one can isolate some of the relevant algebra from the topology and analysis. 

A planar grid diagram is simply an $N$-by-$N$ grid in the plane, with one $O$ marking and one $X$ marking in each row and in each column. An example is shown in Figure~\ref{fig:dada}. Starting from such a diagram $\H$, one can construct a planar Floer complex $CP^-(\H)$. Let us cut the diagram $\H$ into four quadrants, using a vertical line $\ell$ (at distance $k+3/4$ from the left edge, so that $2k$ markings are to its left), and a horizontal line $\ell'$ (at distance $k'+3/4$ from the bottom edge). We denote the four quadrants by $\haa, \had, \hda$ and $\hdd,$ as in Figure~\ref{fig:dada}.

\begin{figure}
\begin{center}
\input{dada.pstex_t}
\end{center}
\caption {{\bf Slicing a planar grid diagram twice.} Here $N=7, k=4$ and $k' =3.$}
\label{fig:dada}
\end{figure}
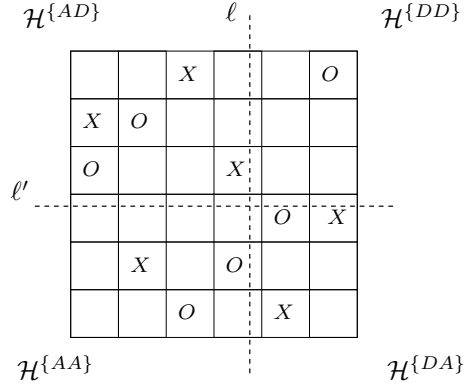

In \cite{LOTplanar}, only the vertical cut was considered. Set $\ha = \haa \cup \had$ and $\hd = \hda \cup \hdd.$ To the interface $\ell$, Lipshitz, Ozsv\'ath and Thurston associate a dg-algebra $\alg(N),$ called the {\em strands algebra}, which is a simpler analogue of the matching algebra $\alg(F).$ To the two sides $\ha$ and $\hd$ they  associate dg-modules $CPA^-(\ha)$ and $CPD^-(\hd)$, and then they prove a pairing theorem of the form:
\begin {equation}
\label {eq:pg1}
 CP^-(\H) \cong CPA^-(\ha) \otimes_{\alg(N)} CPD^-(\hd).
 \end {equation}

One should think of the planar grid as an analogue of a closed 3-manifold, of the interface $\ell$ as the analogue of a closed surface, and of $\ha$ and $\hd$ as analogues of 3-manifolds with boundary, intersecting along the surface. Then \eqref{eq:pg1} is the analogue of the pairing theorem for bordered Floer homology proved in \cite{LOT}.

Let us now consider the horizontal slicing too, and continue the analogy: the intersection point between $\ell$ and $\ell'$ corresponds to the circle, each of the two parts of $\ell$ (or $\ell'$) correspond to surfaces with boundary, and the four quadrants correspond to $3$-manifolds with a codimension-two corner. In the planar grid setting, we can then fully understand a decomposition of $CP^-(\H)$ into four parts, according to the quadrants. To the intersection point we still associate the nilCoxeter sequential 2-algebra $\nil$.  To the two parts of $\ell'$ we will associate algebra-modules over $\nil,$ denoted $\R(k)$ and $\L(N-k)$, respectively. Similarly, to the two parts of $\ell$ we associate algebra-modules $\T(k')$ and $\B(N-k').$ These last two bimodules satisfy a relation:
$$ \alg(N) \cong \T(k') \odot_{\nil} \B(N-k'),$$
similar to the one in Theorem~\ref{thm:main1}. To each quadrant we will associate a certain algebraic object called a sequential 2-module (over two of the four algebra-modules, depending on the quadrant). We denote these 2-modules by $\cpaah, \cpadh, \cpdah,$ and $\cpddh$. The 2-modules can be tensored together vertically (in which case we use the symbol $\odot$) or horizontally (in which case we use the symbol $\oast$). 

\begin {theorem}
\label {thm:main2}
Let $\H^A = \haa \cup \had \cup \hda \cup \hdd$ be a planar grid diagram of size $N$, doubly sliced into four quadrants, as above. There are isomorphisms of differential graded modules over $\alg(N) \cong \T(k') \odot_{\nil} \B(N-k')$: 
\begin {equation}
\label {eq:cpaha0}
CPA^-(\H^A) \cong \bigl(\cpaah\bigr) \odot_{\R(k)} \big(\cpadh \bigr),
\end {equation}
and
\begin {equation}
\label {eq:cpd_deco0}
CPD^-(\H^D) \cong \bigl( \cpdah \bigr) \odot_{\L(N-k)} \bigl( \cpddh \bigr).
\end {equation}
\end {theorem}

We can combine Equations~\eqref{eq:pg1}, \eqref{eq:cpaha0} and \eqref{eq:cpd_deco0} into a single equation, which expresses the whole planar Floer complex as a double tensor product (combining $\odot$ and $\oast$) of the four 2-modules:
$$
CP^-(\H) \cong \odot \oast \bigr(\cpaah, \cpadh, \cpdah, \cpddh \bigl).
$$

It is worth noting that in the planar grid setting, all the objects (algebras, algebra-modules, and 2-modules) come equipped with a differential, but there is no need to consider higher multiplication (that is, $\alg_\infty$) structures. In bordered Floer homology, $\alg_\infty$-structures are unavoidable, and tensor products have to be replaced by derived tensor products. Thus, in a local Heegaard Floer theory, we expect that  to a $3$-manifold with a codimension-two corner we would associate a 2-module with an $\alg_\infty$-structure, and that we would have pairing results analogous to Theorem~\ref{thm:main2}, but involving derived tensor products.

As noted in \cite[Section 8]{LOT}, in bordered Floer homology one case in which the higher $\alg_\infty$-operations vanish is when we work with diagrams that are nice in the sense of Sarkar and Wang \cite{SarkarWang}. Furthermore, in nice diagrams all holomorphic disks can be counted combinatorially, so the analytical difficulties are absent as well. One can define nice diagrams for $3$-manifolds with a codimension-two corner in a similar way. The constructions done for planar grids in this paper should extend to nice diagrams: to any nice cornered diagram one would associate a corresponding 2-module, such that a pairing result is satisfied.  However, the setting of nice diagrams is not entirely satisfactory for developing a theory of cornered Floer homology.  In particular, in the context of nice diagrams it may require elaborate measures to prove that the respective 2-modules are invariants of cornered 3-manifolds; compare \cite{OSSnice}.

\medskip \noindent \textbf {Outline of the paper.} Section~\ref{sec:seq} contains the definitions of the algebraic objects we work with: sequential 2-algebras, algebra-modules, and 2-modules. We also introduce the nilCoxeter sequential 2-algebra and a graphical interpretation of it. Section~\ref{sec:LOT} describes the strands algebra $\alg(N),$ and includes a proof of a decomposition result for it, as a warm-up for Theorem~\ref{thm:main1}. In Section~\ref{sec:bchd} we introduce various diagrammatic presentations (similar to Heegaard diagrams) for surfaces with boundary and for 3-manifolds with codimension two corners. Section~\ref{sec:match} is then devoted to proving Theorem~\ref{thm:main1}, the decomposition result for the matching algebra $\alg(F)$. The rest of the paper focuses on planar grid diagrams. In Section~\ref{sec:vslice} we review the results of \cite{LOTplanar} on the vertical slicing of the planar Floer complex. In Section~\ref{sec:cpa} we define the 2-modules $\cpaa$ and $\cpad$ and prove the first half of Theorem~\ref{thm:main2}, that is, Equation~\eqref{eq:cpaha0}. Lastly, in Section~\ref{sec:cpd} we define the 2-modules $\cpda$ and $\cpdd$ and prove the other half of Theorem~\ref{thm:main2}, that is, Equation~\eqref{eq:cpd_deco0}.

\medskip \noindent  \textbf {Acknowledgments.} We would like to thank Denis Auroux, Aaron Lauda, Robert Lipshitz, Peter Ozsv\'ath for useful conversations and encouragement, and Mohammed Abouzaid and Tim Perutz for a helpful discussion that led to Section~\ref{sec:nilcoxgrid}. The second author is grateful to David Nadler for asking him the question ``What is the Seiberg-Witten invariant of the circle?'', which inspired the research in this paper.

\nopagebreak The first author was partially supported by a Miller Research Fellowship. The second author was partially supported by NSF grant number DMS-0852439 and the Mathematical Sciences Research Institute.

\section {Sequential 2-algebras, algebra-modules, and 2-modules}
\label {sec:seq}

\subsection{Sequential 2-algebras}
\label {sec:sdga}

Fix a ground field $\k$ of characteristic two. Our algebras and modules will be defined over the field $\k.$

In bordered Heegaard Floer homology \cite{LOT, LOTplanar, LOTbimodules}, one associates to a closed surface $F$ a differential graded algebra $\alg(F).$ By the general philosophy of local TQFT's (see \cite{BDH} and \cite[Section 7]{FHLT}), we expect that to a 1-manifold we associate a kind of differential graded 2-algebra, that is, a vector space $\A$ with two multiplications ($\cdot$ and $*$) and one differential. It is natural to require the multiplications to have units and to commute in the following sense:
\begin {equation}
\label {eq:commute}
 (a*b) \cdot (c* d) = (a \cdot c) *(b \cdot d),
 \end {equation}
for all $a, b, c, d \in \A.$ However, this requirement is too strong, because it automatically implies that our structures are commutative:

\begin {lemma}[Eckmann-Hilton \cite{EH}]
Let $\A$ be a $\k$-vector space with two unital algebra structures $(\A, \cdot , e) $ and $(\A, *, u).$ Suppose that \eqref{eq:commute} is satisfied. Then the two algebra multiplications coincide and are commutative.  
\end {lemma}

\begin {proof}
We have
$$ e = e \cdot e = (e * u) \cdot (u * e) = (e \cdot u) * (u \cdot e) = u * u = u,$$
so the two units are the same. Therefore, 
$$ a \cdot b =  (a * e) \cdot (e * b) = (a \cdot e) * (e \cdot b) =  a * b = ( e \cdot a) * (b \cdot e) = (e * b) \cdot (a * e) = b \cdot a,$$ 
for any $a, b \in \A.$
\end {proof}

In the literature on local TQFT's, the hypotheses are sometimes weakened by allowing one of the multiplications to be non-unital, and even non-associative \cite{BDH, FHLT}. Here we take a different view, by weakening the commutativity requirement.

\begin {definition}
\label {def:dg2a}
A (differential graded) {\em sequential 2-algebra} $\A = \{\A_m| m \geq 0\}$ is a sequence of differential graded algebras $(\A_m, \cdot, \del)$, together with dg-algebra maps\footnote{Throughout this paper, the symbol $\otimes$ will denote  tensor product over the ground field $\k$ (or, from Section~\ref{sec:vslice} on, over a certain ground ring $\kk$). This is to be contrasted with the tensor products $\odot$ and $\oast$, to be defined later, which are usually taken over larger algebras.}   
$$ \mu_{m, n} : \A_m \otimes \A_n \longrightarrow \A_{m+n},  \ \ \mu_{m,n}(a \otimes b) =: a * b, $$
for all $m, n \geq 0,$ satisfying the associativity condition:
$$ (a * b) * c = a * (b * c),$$
for any $a \in \A_m, b \in \A_n, c \in \A_p,$ with $m, n, p \geq 0.$

The sequential 2-algebra $\A$ is called {\em unital} if each algebra $(\A_m, \cdot, \del)$ has a multiplicative unit $e_m$ such that
$$ e_m * e_n = e_{m+n}, \ \text{ for all } m, n \geq 0,$$
and
$$ e_0 * a = a * e_0 = a, \ \text{ for any } a \in \A_m, \ m \geq 0.$$
  
\end {definition}

If $\A$ is a sequential 2-algebra, the infinite product $\ol \A = \prod_m \A_m$ admits two multiplications $\cdot$ and $*$ defined as follows. If $a = (a_m)_{m \geq 0}$ and $b=(b_m)_{m \geq 0}$ are two sequences in $\ol \A,$ we let the components of their products be
$$ (a \cdot b)_m =a_m \cdot b_m$$
and
$$ (a * b)_m= \sum_{i+j =m} a_i * b_j.$$

Both of these algebra structures on $\ol \A$ are associative and unital, but with different units: $e = (e_0, e_1, e_2, \dots)$ for $\cdot$, and $u = (e_0, 0, 0, \dots)$ for $*.$ The two multiplications commute in a local sense: if $a, c \in \A_m$ and $b, d \in \A_n$ for some $m, n \geq 0,$ then \eqref{eq:commute} holds.

The Eckmann-Hilton Lemma does not apply here because the multiplications do not commute in general. Indeed, for $a = (a_m)_{m \geq 0} , b = (b_m)_{m \geq 0}, c=(c_m)_{m \geq 0}, d=(d_m)_{m \geq 0},$ we have
$$ \bigl( (a * b) \cdot (c * d) \bigr)_m = \sum_{\substack{i+j=m \\ k+l = m}} (a_i * b_j) \cdot (c_k * d_l) ,$$ 
whereas
$$  \bigl( (a \cdot c) * (b \cdot d) \bigr)_m = \sum_{i+j=m} (a_i * b_j) \cdot (c_i * d_j).$$

\begin {remark}
When talking about differential graded objects, we mean $\zz$-graded unless otherwise noted. (In Sections~\ref{sec:ncgrading} and \ref{sec:match}, we will consider gradings by noncommutative groups, and in Sections~\ref{sec:vslice}, \ref{sec:cpa}, and \ref{sec:cpd} the gradings will generally be by $\zz \times \zz$.) For now, in Definition~\ref{def:dg2a}, each $\A_m$ has a $\zz$-grading. {\em We emphasize that the subscript $m$ in $\A_m$ has no relation to the grading.} For simplicity, we typically suppress the grading from the notation. When we wish to explicitly mention an algebra summand in grading $k$, we use the subscript $\langle k \rangle,$ as in for example $\A_{m, \langle k \rangle}$. Our convention is that the differentials decrease the grading by one. 
\end {remark}

\begin {remark}
If we ignore the differential and the grading, a sequential 2-algebra can be viewed as a particular example of a (strict) 2-category. Precisely, it is a 2-category with a single object (or in other words, a monoidal category), with 1-morphisms from the object to itself indexed by the set $\zz_{\geq 0}$ of nonnegative integers, and with 2-morphisms from $m$ to $n$ existing only when $m=n$. If $m=n,$ the respective set of 2-morphisms is $\A_m.$
\end {remark}

\subsection {An example: the nilCoxeter sequential 2-algebra}
\label {sec:nilcox2a}

Let $m \geq 0.$ The {\em nilCoxeter algebra} $\nil_m$ is defined as the unital $\k$-algebra generated by elements $\sigma_i, \ i=1, \dots, m-1,$ subject to the relations \eqref{eq:nc1}-\eqref{eq:nc3}. The nilCoxeter algebra was first considered by Bernstein, Gelfand and Gelfand \cite{BGG}, and its name is due to Fomin and Stanley \cite{FominStanley}. Its role in categorification was pointed out by Khovanov \cite{KhNilCox, Khovanov12}; see also the remarks in \cite{LOT} and \cite{LOTplanar} about its relevance to bordered Floer homology. 

In the discussion below we will use a few standard facts about the symmetric group and the Bruhat order; for their proofs, see for example \cite{Bourbaki14} or \cite{BjornerBrenti}. See also \cite{FominStanley} for more details about the nilCoxeter algebra.

The nilCoxeter algebra is closely related to the symmetric group $S_m$. Indeed, the group algebra $\k[S_m]$ is generated by elements $s_i, \ i=1, \dots, m-1,$ subject to the same relations as the $\sigma_i$'s except we have to change the relation $\sigma_i^2 = 0$ to $s_i^2=1.$ As a $\k$-vector space, $\nil_m$ has $m!$ generators, namely
\begin{equation}
\label{eq:sigmas}
 (\sigma_{i_1}\sigma_{i_1 -1} \dots \sigma_{i_1 - k_1})(\sigma_{i_2} \sigma_{i_2 -1}\dots \sigma_{i_2-k_2}) \dots (\sigma_{i_p} \sigma_{i_p-1} \dots \sigma_{i_p-k_p}),
\end {equation}
for all possible $p \geq 0, \ 1 \leq i_1 < i_2 < \dots < i_p \leq m-1,$ and $0 \leq k_j < i_j \ (j=1, \dots, p).$ If we replace the $\sigma_i$'s by $s_i$'s, the same expressions can be used to uniquely describe the elements of $S_m$. This sets up a bijection between $S_m$ and a set of generators of $\nil_m.$ For $w \in S_m,$ we denote by $\sigma_w \in \nil_m$ the corresponding generator.  

In terms of this additive basis, the multiplication in $\nil_m$ is given by
$$ \sigma_w \sigma_{w'} = \begin{cases} \sigma_{ww'} & \text{if } \ell(ww') = \ell(w) + \ell(w'), \\ 0 & \text{otherwise,} \end {cases}$$
where $\ell(w)$ is the standard length (i.e., the number of inversions, or the minimal number of generators) of the permutation $w$, i.e. $\ell(w) = (k_1 + \dots + k_p) + p$ in terms of \eqref{eq:sigmas}.

We define a differential on $\nil_m$ by setting $\del \sigma_i =1,$ and then extending it by the Leibniz rule. Let us write $w' \prec w$ is $w'$ if an immediate predecessor to $w$ in the strong Bruhat order on $S_m;$ that is, if there exist $1\leq i < j \leq m$ such that $w'(i) = w(j) <w(i) =w'(j), \ w'(k) = w(k)$ for all $k \neq i,j$, and $\ell(w') = \ell(w) - 1.$ We can then write
$$ \del \sigma_w = \sum_{w' \prec w  } \sigma_{w'}. $$
\nid Note that the length $\ell$ gives a grading on $\nil_m,$ making $\nil_m$ into a differential graded algebra. 

\begin {remark}
As a chain complex, $\nil_m$ can be viewed as the cellular chain complex of the CW-space constructed by Bj\"orner  in \cite{Bjorner} (corresponding to the Weyl group $S_m$).
\end {remark} 

We define the {\em nilCoxeter sequential 2-algebra} $\nil$ to be composed of the pieces $\nil_m$, with the second product $* : \nil_m \otimes \nil_n \to \nil_{m+n}$ given by concatenation $\sigma_i \otimes \sigma_j \to \sigma_i \sigma_{m+j}.$ Observe that $\nil$ is unital.

We can represent the sequential algebra $\nil$ graphically as follows. A generator of $\nil_m$ is a box with $m$ strands going up, from the bottom edge to the top edge, see Figure~\ref{fig:nilcox1}. If we number the initial points on the bottom edge by $1, 2, \dots, m$ (from left to right), and we do the same with the final points on the top edge, then the strands specify a permutation in $S_m$ or, equivalently, an element of $\nil_m.$ In particular, $\sigma_i$ corresponds to an interchange between the $i\th$ and the $(i+1)\st$ strands. The algebra multiplication $\cdot$ is given by stacking pictures vertically (with the convention that $a \cdot b$ means $b$ is on top of $a$), and the second multiplication $a * b$ is stacking pictures horizontally (such that $a * b$ means $a$ is to the left of $b$). The defining relations of the nilCoxeter algebra are shown graphically by the local pictures from Figure~\ref{fig:nilcoxrel}. The grading on $\nil_m$ is given by counting the number of crossings in the diagram of a generator, so we denote it by 
\begin {equation}
\label {eq:cr}
\cro: \nil_m \to \zz, \ \ \ \cro(\sigma_w) = \ell(w).
\end {equation}

\begin{figure}
\begin{center}
\input{nilcox1.pstex_t}
\end{center}
\captionsetup{singlelinecheck=off}
\caption[.] {{\bf A generator of the nilCoxeter sequential 2-algebra.} We show here the element 
$\sigma_1 \sigma_3 \sigma_2 \sigma_1$ in $\nil_4 \subset \nil,$ corresponding to the permutation 
$w=\bigl ( \begin{smallmatrix}1& 2& 3 & 4 \\ 3 & 2 & 4 & 1\end{smallmatrix} \bigr)$. Note that the figure is read from bottom to top.
}
\label{fig:nilcox1}
\end{figure}
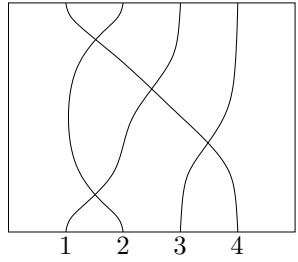 

\begin{figure}
\begin{center}
\input{nilcoxrel.pstex_t}
\end{center}
\caption {{\bf Relations in the nilCoxeter algebra.} 
}
\label{fig:nilcoxrel}
\end{figure}
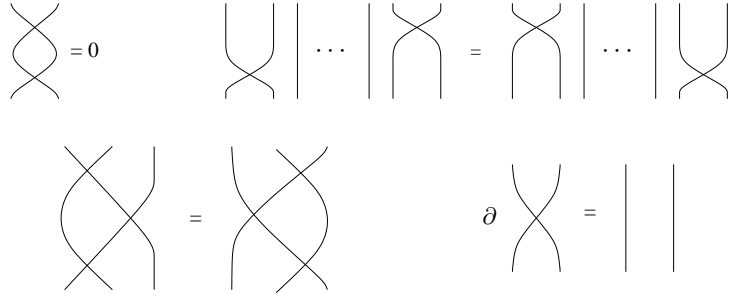

\subsection {Grid diagram interpretation}
\label {sec:nilcoxgrid}

Another pictorial representation of $\nil_m$ is by an empty planar grid diagram, reminiscent of \cite{MOS, LOTplanar}, and especially \cite{AurouxBordered}. Indeed, in \cite{AurouxBordered}, Auroux describes the matching algebra $\mathcal{A}(F)$ of a surface $F$ in terms of the partially wrapped Fukaya category of a symmetric product of the surface. Along those lines, we conjecture that the nilCoxeter algebra $\nil_m$ is equivalent to a variant of the partially wrapped Fukaya category of the symmetric product $\Sym^m(\rr^2)$.  Further, the product $*$ on the 2-algebra $\nil$ should be induced on Fukaya categories by a concatenation map $$* : \Sym^m(\rr^2) \times \Sym^n(\rr^2) \to \Sym^{m+n}(\rr^2),$$ which puts Lagrangians side by side. Alternately (and more simply), one could 
compare $\nil_m$ to the Fukaya category associated to a Lefschetz fibration on $\Sym^m(\rr^2)$ with a single critical point, arising from a Lefschetz fibration on $\rr^2$ with $m$ critical points, as in \cite[Section 2]{AurouxBordered}.

For simplicity, in this section we do not discuss Fukaya categories, but rather give a concrete description of $\nil_m$ in terms of grid diagrams. This description is the basis for the conjectural interpretation mentioned above.

Consider $m$ horizontal lines in the plane $\alpha_1, \dots, \alpha_m$ (ordered from top to bottom), $n$ vertical lines $\beta_1, \dots, \beta_m$ (ordered from left to right), forming an $(m-1)$-by-$(m-1)$ grid. To each permutation $w \in S_m$ we can associate an $m$-tuple of intersection points 
$$\x = \x_w = \{ x_i \in \alpha_i \cap \beta_{w(i)} \}_{i=1}^m.$$ Let $\S(\alpha, \beta)$ be the set of such $m$-tuples. 

Given sets $E, F \subset \rr^2,$ let 
\begin {equation}
\label{eq:I}
\I(E, F) = \# \{(e=(e_1, e_2), f=(f_1, f_2)) \in E \times F \mid e_1 < f_1 \text{ and } e_2 < f_2 \}.
\end {equation}
\nid For $\x = \x_w \in \S(\alpha, \beta)$, we denote
$$ \mu(\x) := \I(\x, \x) = \ell(w).$$
(Recall that $\ell(w)$ is the number of crossings in the diagram of $\sigma_w$.)

 We form a complex $C(m)$ over $\k$ freely generated by the elements of $\S(\alpha, \beta),$ with each $\x$ in grading $\mu(\x),$ and with the differential
$$ \del \x = \sum_{\x' \in \S(\alpha, \beta)} \sum_{r \in \Rect(\x, \x')} \x'.$$
\nid Here, $\Rect(\x, \x')$ is the set of empty rectangles between $\x$ and $\x',$ as in \cite[p.3]{MOST}. More precisely, $\Rect(\x, \x')$ has either zero or one element. It can be nonempty only when the $\x$ and $\x'$ differ in exactly two rows and two columns, in which case $r \in \Rect(\x, \x')$ is the rectangle bounded by these rows and columns, with the lower-left vertex being part of $\x$ (rather than $\x'$). Further, we require that $r$ is ``empty'', in the sense that the interior of $r$ does not contain any of the points in $\x \cap \x'.$ See Figure~\ref{fig:nilcox2} for an example.

\begin{figure}
\begin{center}
\input{nilcox2.pstex_t}
\end{center}
\captionsetup{singlelinecheck=off}
\caption[.]{{\bf The nilCoxeter chain complex on a planar grid diagram.} The black dots form the generator $\x$ associated to the permutation $w=\bigl(\begin{smallmatrix} 1 & 2 & 3 & 4 & 5 & 6 \\ 3&1&5&6&2&4 \end{smallmatrix})$, and the white dots form the generator $\x'$ associated to $w' = \bigl( \begin{smallmatrix} 1 & 2 & 3 & 4 & 5 & 6 \\ 3&1&2&6&5&4 \end{smallmatrix} \bigr).$ The shaded rectangle shows that $\x'$ is a term in the differential of $\x.$
}
\label{fig:nilcox2}
\end{figure}
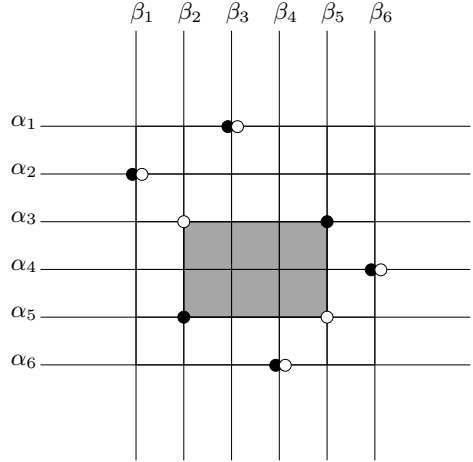

Now in addition to the horizontal $\alpha$ lines and the vertical $\beta$ lines, consider $m$ diagonal lines (of slope $-1/2$) $\gamma_1, \ldots, \gamma_m$ (ordered from left to right), such that no three lines intersect in one point. See Figure~\ref{fig:nilcox3}. We use the $\gamma$ curves to turn $C(m)$ into a differential graded algebra as follows. For $w \in S_m$, we let
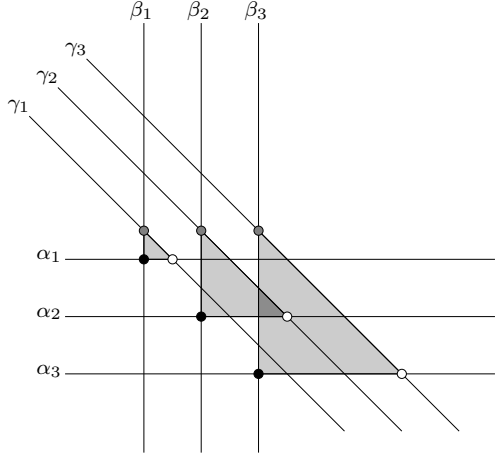
\begin{figure}
\begin{center}
\input{nilcox3.pstex_t}
\end{center}
\caption {{\bf The nilCoxeter multiplication in terms of grid diagrams.} The shaded objects form a triangle between $\x$, $\y$ and $\z$, where the generators $\x, \y$ and $\z$ all correspond to the identity permutation. The generators $\x, \y, \z$ are drawn as black, gray, and white dots, respectively. The darker shading represents a multiplicity two area (a head to tail overlap).
}
\label{fig:nilcox3}
\end{figure}
$$ \y_w = \{ y_i \in \beta_i \cap \gamma_{w(i)} \}_{i=1}^m, \ \ \ \  \z_w = \{ z_i \in \alpha_i \cap \gamma_{w(i)} \}_{i=1}^m.$$
\nid We let $\S(\beta, \gamma)$ be the set of $\y_w$'s, and $\S(\alpha, \gamma)$ the set of $\z_w$'s. Note that the complex $C(m)$ can be defined using $\S(\beta, \gamma)$ or $\S(\alpha, \gamma)$ instead of $\S(\alpha, \beta),$ and using empty parallelograms instead of empty rectangles.

For $\x \in \S(\alpha, \beta), \y \in \S(\beta, \gamma),$ and $\z \in \S(\alpha, \gamma),$ a {\em triangle} between $\x, \y $ and $\z$ is defined to be a collection of $m$ embedded triangles in $\rr^2$, each with boundaries on an $\alpha$, $\beta$ and a $\gamma$ curve (in this clockwise order), such that each of the components of $\x$, $\y$ and $\z$ appears once as a vertex of one of the $m$ triangles; further, we require that the $m$ triangles only intersect ``head to tail'' in the sense of \cite[Section 2.3]{LMW}-- compare also \cite{SarkarIndex}. See Figure~\ref{fig:nilcox3} for an example of an allowed intersection. We let $\Tri(\x, \y, \z)$ denote the set of triangles between $\x, \y$ and $\z.$ Note that $\Tri(\x, \y, \z)$ always has at most one element.

We define a multiplication $ \cdot : C(m) \otimes C(m) \to C(m)$ as follows. For $\x \in \S(\alpha, \beta)$ and $\y \in \S(\beta, \gamma),$ we let
$$ \x \cdot  \y = \sum_{\z \in S(\alpha, \gamma)} \sum_{r \in \Tri(\x, \y, \z)} \z.$$ 
\nid This product makes $C(m)$ into a differential graded algebra.

\begin {proposition}
The map taking a generator $\x_w$ to $\sigma_w$ (for all $w \in S_m$) induces an isomorphism between the differential graded algebras $C(m)$ and $\nil_m.$
\end {proposition}

\nid The proof is the same as the proofs of \cite[Propositions 3.4 and 3.5]{AurouxBordered}.

The reader may wonder about the homology of the complex underlying $\nil_m.$ For $m=0$ or $1,$ we have $\nil_m = \k,$ with trivial differential. In general, we have the following result, which was independently proved by Khovanov (\cite[Lemma 4]{Khovanov12}):

\begin{lemma}
\label {lemma:hom}
The nilCoxeter complex $\nil_m$ is acyclic for $m \geq 2.$ 
\end {lemma}

\begin {proof}
We use a geometric argument similar to those in \cite[Proposition 3.8]{MOST} and \cite[Section 5]{MOT}. Recall that the the $\alpha $ and $\beta$ curves form an $(m-1)$-by-$(m-1)$ grid in the plane, made of $(m-1)^2$ small rectangles. Mark an $X$ in each small rectangle except the ones in the first row (between $\alpha_1$ and $\alpha_2$). For $\x, \x' \in \S(\alpha, \beta)$ and $r \in \Rect(\x, \x'),$ we let $X(r)$ be the number of $X$ markings in the interior of the rectangle $r.$

Define a filtration $\Filt$ on $C(n)$ by the normalization $\Filt(\x_{\id}) = 0$ and the requirement that
$$ \Filt (\x) - \Filt(\x') = X(r),$$
whenever there exists a rectangle $r \in \Rect(\x, \x').$ This specifies a filtration $\Filt$ uniquely. In the associated graded $\gr_\Filt C(m),$ the differential is given just by counting rectangles supported in the first row. Such a rectangle exists between $\x_{w}$ and $\x_{w'}$ if and only if $w(1) < w(2)$ and $w' = \sigma_1 w.$ Thus, every $\x \in \S(\alpha, \beta)$ is either the initial or the final point of a unique such rectangle. This means that the differential on $\gr_\Filt C(m)$ makes the generators $\x_w$ cancel one another in pairs, yielding trivial homology. Since $\gr_\Filt C(m)$ is acyclic, so is $C(m) \cong \nil_m.$
\end {proof}

\begin {remark}
Let $\nil'$ be the sequential 2-algebra with $\nil'_m = \nil_m$ for $m=0,1$ and $\nil'_m = 0$ for $m \geq 2.$ The natural projection $\ol{\nil} \to \ol{\nil'}$ is a homomorphism with respect to both multiplications, and is a quasi-isomorphism according to Lemma~\ref{lemma:hom}. A quasi-isomorphism of differential graded algebras induces an equivalence of the corresponding derived categories of dg-modules over the algebras---see for example \cite[Theorem 10.12.5.1]{BernsteinLunts} and \cite[Proposition 2.4.10]{LOTbimodules}. We expect that a quasi-isomorphism of sequential 2-algebras induces an equivalence of derived categories of algebra-modules, and of derived categories of 2-modules.  (The notions of algebra-module and 2-module are described in the next two subsections.)  Supposing that such equivalences exist, it would still be inadvisable to replace the nilCoxeter 2-algebra $\nil$ with the truncation $\nil'$ in our constructions, because the planar-Floer 2-modules (presented in Sections~\ref{sec:cpa} and \ref{sec:cpd}) would become rather more complicated and less natural.
\end {remark}

\subsection{Sequential algebra-modules}
\label {sec:sam}

In the following, let $\A$ be an arbitrary unital sequential 2-algebra. 

\begin{definition}
\label {def:bmod}
A (differential graded, sequential) {\em top algebra-module} $\T$ over $\A$ consists of a sequence of differential graded right modules $\T_m$ over $\A_m$ for $m \geq 0,$ together with chain maps:
$$ \T_m \otimes \T_n \to \T_{m+n}, \ \ \phi \otimes \psi \to \phi * \psi,$$
satisfying 
$$  (\phi * \psi) * \zeta = \phi * (\psi * \zeta),$$
for any $\phi \in \T_m, \psi \in \T_n, \zeta \in \T_p,$ and
$$ (\phi * \psi) \cdot (a * b) = ( \phi \cdot a) * ( \psi \cdot b),$$
for any $a \in \A_m, b \in \A_n, \phi \in \T_m, \psi \in \T_n,$ where $\cdot$ denotes the algebra action on the modules.

The algebra-module $\T$ is called {\em unital} if there exists $1 \in \T_0$ which is a unit with respect to the $*$ multiplication. (Note also that when we say that $\T_m$ is a module over the unital algebra $\A_m,$ we implicitly require that $\phi \cdot e_m = \phi$ for any $\phi \in \T_m$.)    
\end {definition}

\nid The name algebra-module comes from the fact that $\ol \T = \prod_m \T_m$ is a differential graded algebra with respect to the $*$ multiplication, and a $\ol \A$-module with respect to the $\cdot$ multiplication.

We define a {\em bottom algebra-module} $\B = \{\B_p\}_{p \geq 0}$ over $\A$ in the same way, except that the $\B_p$ are left modules over $\A_p.$ We set $\ol \B = \prod_p \B_p.$ The names top and bottom (instead of right and left) come from viewing $\cdot$ as a vertical multiplication, in the spirit of Section~\ref{sec:nilcox2a}.

We introduce the notation $\odot$ to denote tensor product with respect to the $\cdot$ multiplication.

\begin {lemma}
\label {lem:deco}
Let $\A$ be a unital sequential 2-algebra, $\T$ a top algebra-module over $\A,$ and $\B$ a bottom algebra-module over $\A.$ Then: 
 \begin {equation}
 \label {eq:decompose}
 \ol \T \odot_{\ol \A} \ol \B =  \prod_m (\T_m \odot_{\A_m} \B_m).
  \end {equation}
\end {lemma}

\begin {proof}
For any $\phi \in \T_m,  \phi' \in \B_n$ with $m \neq n$, we use the fact that $\A$ is unital to get the following relation in $ \ol \T \odot_{\ol \A} \ol \B $:
$$ \phi \odot \phi' = \phi \odot (e_n \cdot \phi') = (\phi \cdot e_n) \odot \phi' = 0 \odot \phi' = 0.$$
Hence the left-hand side of \eqref{eq:decompose} decomposes as $\prod_m (\T_m \odot_{\ol \A} \B_m).$ This expression is, in turn, the same as the right-hand side of \eqref{eq:decompose}.
\end {proof}

\begin{definition}
\label{def:odot}
 Let $\A$ be a unital sequential 2-algebra, $\T$ a top algebra-module over $\A,$ and $\B$ a bottom algebra-module over $\A.$ The tensor product $\alg = \T \odot_{\A} \B$ is defined to be the expression~ \eqref{eq:decompose}. It has the induced structure of a differential graded algebra, with the multiplication given by applying the entrywise $*$ product,
$$ (\phi \odot \phi') * (\psi \odot \psi') := (\phi * \psi) \odot (\phi * \psi'), $$
for all $\phi \in \T_m, \psi \in \T_n, \phi' \in \B_m, \psi' \in \B_n.$
\end {definition}

\nid Note that if $\T$ and $\B$ are unital, so is the algebra $\T \odot_{\A} \B.$

The top and bottom algebra-modules are modules with respect to the vertical $\cdot$ multiplication and algebras with respect to the horizontal $*$ multiplication. We also have notions of right and left  algebra-modules, which are modules horizontally and algebras vertically. Their definitions are somewhat different from those of top and bottom algebra-modules, because of the lack of symmetry between the two multiplications on a sequential 2-algebra. In particular, left and right algebra-modules are collections of complexes indexed by two subscripts rather than just one:

\begin {definition}
A (differential graded, sequential) {\em right algebra-module} $\R$ over $\A$ is a collection of chain complexes $\{ \R_{m,p} | m, p \geq 0 \}$ together with chain maps:
$$ * : \R_{m,p} \otimes \A_{n} \to \R_{m+n, p+n},$$
$$ \cdot : \R_{m,n} \otimes \R_{n, p} \to \R_{m, p},$$
satisfying the following associativity and local commutativity relations:
\begin {eqnarray}
\label {eq:alc1}
(\phi \cdot \psi) \cdot \zeta &=& \phi \cdot (\psi \cdot \zeta), \hskip1cm \forall \phi \in \R_{m, n}, \psi \in \R_{n,p}, \zeta \in \R_{p,q}, \\
\label{eq:alc2}
\phi * (a * b)&=& (\phi * a) * b, \hskip1.2cm \forall  \phi \in \R_{m,p}, a \in \A_{n}, b \in \A_{q},  \\
\label{eq:alc3} 
(\phi * \psi) \cdot (a * b)&=& (\phi \cdot a) * (\psi \cdot b), \ \ \ \forall \phi \in \R_{m,n}, \psi \in \R_{n,p}, a, b \in \A_{q}.
\end {eqnarray}
$\R$ is called {\em unital} if there exists $1_0 \in \R_{0,0}$ such that $1_m := 1_0 * e_m$ is a unit for the vertical multiplication on $\R_{m,m}$, for any $m \geq 0$; and also $\phi * e_0 = \phi$ for any $\phi \in \R_{m,p}, m, p \geq 0$. 

A {\em left algebra-module} $\L$ over $\A$ is a collection of chain complexes $\{ \L_{m,p} | m, p \geq 0 \}$ together with chain maps:
$$ * : \A_{n} \otimes \L_{m, p} \to \L_{m+n, p+n},$$
$$ \cdot : \L_{m,n} \otimes \L_{n,p} \to \L_{m, p},$$
satisfying associativity and local commutation relations similar to \eqref{eq:alc1}-\eqref{eq:alc3}. $\L$ is called {\em unital} if there exists $1_0 \in \L_{0,0}$ such that $1_m :=  e_m * 1_0$ is a unit for the vertical multiplication on $\L_{m,m}$, for any $m \geq 0$; and $e_0 * \phi = \phi$ for any $\phi \in \L_{m,p}, m, p \geq 0$.
\end {definition}

To visualize the definitions of top, bottom, right and left algebra-modules, it helps to think of the generators of the algebra-modules as boxes, as in Figure~\ref{fig:mods}. The elements of the 2-algebra act on them in one direction (by horizontal or vertical concatenation), and the algebra-modules have their own multiplication (also concatenation) in the other direction.

\begin{figure}
\begin{center}
\input{modseq.pstex_t}
\end{center}
\caption {{\bf A sequential 2-algebra, algebra-modules, and 2-modules.} The notation refers to the relative position of the 2-algebra $\A$ with respect to the object. For example, the top-right 2-module $TR$ has $\A$ near its top right corner.
}
\label{fig:mods}
\end{figure}
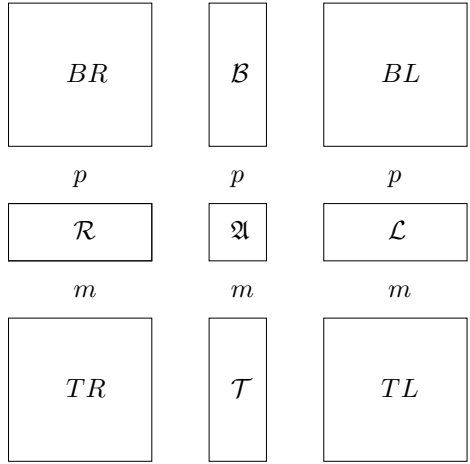

We define $\ol \R$ (resp. $\ol \L$) as the direct product of all $\R_{m,p}$ (resp. $\L_{m,p}$). These are right, resp. left modules over $\ol \A$, with respect to horizontal multiplication. With respect to vertical multiplication, $\ol \R$ and $\ol \L$ are not automatically algebras: Given $\phi=(\phi_{m,p})_{m,p \geq 0} \in \ol \R$ and $\phi' = (\phi'_{m,p})_{m, p \geq 0} \in \ol \R$, if we try to set
\begin{equation}
\label{eq:mf}
 \phi \cdot \phi' = \bigl (\sum_j \phi_{m,j} \cdot \phi'_{j, p} \bigr )_{m, p \geq 0},
 \end{equation}
we run into the problem that the sum over $j$ may be infinite. 

\begin{definition}
\label{def:mf}
We say that a right algebra-module $\R$ is {\em multiplicatively finite} if the sums appearing in \eqref{eq:mf} are finite, for any $\phi$ and $\phi'$. A corresponding definition applies to left algebra-modules.
\end{definition}

If $\R$ is multiplicatively finite, then $\ol \R$ is an algebra. A particular condition that guarantees multiplicative finiteness is if we have $\R_{m,p} = 0$ whenever $m < p$ (or alternately, if $\R_{m,p} = 0$ whenever $m > p$). The right and left algebra-modules that will appear later in this paper are multiplicatively finite, for this kind of reason.

By analogy with Definition~\ref{def:odot}, we can introduce the {\em horizontal tensor product} of right and left algebra-modules
$$ \R \oast_\A \L := \ol \R \oast_{\ol \A} \ol \L.$$
\nid For the horizontal tensor product we do not have an exact analogue of Lemma~\ref{lem:deco}, that is, a full decomposition of  $ \R \oast_\A \L$ as a direct product of local tensor products. Indeed, the local pieces $\R_{m,p}$ and $\L_{m,p}$ are not even modules over $\A_n$ (because the algebra action takes them to different local pieces). We only have a partial decomposition
$$\R \oast_\A \L =   \prod_{m, p} (\R \oast_\A \L)_{m,p}$$
where the local piece $(\R \oast_\A \L)_{m,p}$ is generated by elements of the form $\phi \odot \psi,$ with $\phi \in \R_{m',p'},  \psi \in \B_{m'', p''},$ such that $m'+m''=m$ and $p'+p'' =p$.

As a consequence,  it turns out that, in general (and even if $\R$ and $\L$ are multiplicatively finite), $\R \oast_\A \L$ does not have a well behaved algebra structure with respect to the vertical multiplication. Indeed, we would be tempted to set
$$ (\phi \oast \psi) \cdot (\phi' \oast \psi') = (\phi \cdot \phi') \oast (\psi \cdot \psi')$$
for $\phi \in \R_{m,n}, \phi' \in \R_{n,p}, \psi \in \R_{m', n'}, \psi'\in \R_{n', p'}$, and
$$  (\phi \oast \psi) \cdot (\phi' \oast \psi') = 0$$
for $\phi \in \R_{m,n}, \phi' \in \R_{q,p}, \psi \in \R_{m', n'}, \psi'\in \R_{q', p'}$ with  $n \neq q$ or $n' \neq q'$. When the algebra-modules under consideration are unital, this quickly leads to a contradiction. For example, on the one hand,
$$ (1_0 \oast 1_1) \cdot (1_0 \oast 1_1) = (1_0 \cdot 1_0) \oast (1_1 \cdot 1_1) = 1_0 \oast 1_1.$$
On the other hand,
$$ 1_0 \oast 1_1 =1_0 \oast (e_1 * 1_0) = (1_0 * e_1) \oast 1_0 = 1_1 \oast 1_0$$
and
$$ (1_0 \oast 1_1) \cdot (1_1 \oast 1_0) = 0.$$
This would imply that $1_0 \oast 1_1 =0$, thus either $0 = 1_0 \in \R_{0,0}$ (which contradicts the unitality of $\R$ unless $\R_{0,0} = 0$) or $0 = 1_1 \in \L_{1,1}$ (which contradicts the unitality of $\L$ unless $\L_{1,1} = 0$).

Because of this issue, in this paper we will not make use of the horizontal tensor product of algebra-modules. Note that we do use horizontal tensor products of usual dg-modules over a dg-algebra---for example when considering the expression \eqref{eq:bltr} below.

\begin {remark}
When working with dg-modules over an algebra, it is sometimes necessary to consider the derived tensor product. In all the situations considered in this paper, at least one of the two dg-modules in a tensor product will be projective over the algebra in question. This implies that the derived and usual tensor products coincide. Compare \cite[Remark 8.1]{LOTplanar}.
\end {remark}

\subsection {Sequential 2-modules}
\label {sec:s2m}
Let $\A$ be a unital sequential 2-algebra, and let $\T, \R, \B, \L$ be top, right, bottom, and left algebra-modules over $\A,$ respectively. Further, suppose that $\R$ and $\L$ are multiplicatively finite.

\begin {definition} \label {def:amods}
A (differential graded, sequential) {\em top-right 2-module} $TR$ over $\R$ and $\T$ is a collection of chain complexes $\{TR_{m} | m \geq 0 \}$ together with chain maps  
$$ * : TR_{m} \otimes \T_{m'} \to TR_{m+m'},$$
$$ \cdot : TR_{m} \otimes \R_{m, p} \to TR_{p},$$
satisfying associativity and local commutation relations as follows:
\begin {eqnarray}
\label {eq:alc1'}
(\x * \phi) * \psi &=& \x * (\phi * \psi), \hskip.9cm \forall \x \in TR_{m},\phi \in \T_{m'}, \psi \in \T_{m''}, \\
\label{eq:alc2'}
\x \cdot (\phi \cdot \psi)&=& (\x \cdot \phi) \cdot \psi, \hskip1.1cm \forall  \x \in TR_{m},\phi \in \R_{m, n}, \psi \in \R_{n,p},  \\
\label{eq:alc3'} 
(\x * \phi) \cdot (\psi * a)&=& (\x \cdot \psi) * (\phi \cdot a), \hskip .2cm \forall \x \in TR_{m}, \phi \in \T_{m'}, \psi \in \R_{m, p}, a \in \A_{m'}.
\end {eqnarray}

If $\R$ and $\T$ are unital, we also demand that 
\begin {eqnarray}
\label {eq:uni1}
\x * 1 &=& \x, \hskip.9cm \forall \x \in TR_{m}, \text{ where } 1 \in \T_{0}, \\
\label{eq:uni2}
\x \cdot 1_m&=& \x, \hskip.9cm \forall  \x \in TR_{m}, \text{ where } 1_m \in \R_{m,m}.
\end {eqnarray}

A {\em bottom-right 2-module} $BR$ over $\R$ and $\B$ is a collection of chain complexes $\{BR_{p} | p \geq 0 \}$ together with chain maps  
$$ * : BR_{p} \otimes \B_{p'} \to BR_{p+p'},$$
$$ \cdot : \R_{m,p} \otimes BR_{p} \to BR_{ m},$$
satisfying associativity and local commutation relations similar to \eqref{eq:alc1'}-\eqref{eq:alc3'}. Further, if $\R$ and $\B$ are unital, we impose conditions similar to \eqref{eq:uni1}-\eqref{eq:uni2}.

A {\em bottom-left 2-module} $BL$ over $\L$ and $\B$ is a collection of chain complexes $\{BL_{p} | p\geq 0 \}$ together with chain maps  
$$ * : \B_{p} \otimes BL_{p'} \to BL_{p+p'},$$
$$ \cdot : \L_{m,p} \otimes BL_{ p} \to BL_{m},$$
satisfying associativity and local commutation relations similar to \eqref{eq:alc1'}-\eqref{eq:alc3'}. Further, if $\L$ and $\B$ are unital, we impose conditions similar to \eqref{eq:uni1}-\eqref{eq:uni2}.

A {\em top-left 2-module} $TL$ over $\L$ and $\T$ is a collection of chain complexes $\{TL_{m} | m \geq 0 \}$ together with chain maps  
$$ * : \T_{m} \otimes TL_{m'} \to TL_{m+m'},$$
$$ \cdot : TL_{m} \otimes \L_{m, p} \to TL_{p},$$
satisfying associativity and local commutation relations similar to \eqref{eq:alc1'}-\eqref{eq:alc3'}. Further, if $\L$ and $\T$ are unital, we impose conditions similar to \eqref{eq:uni1}-\eqref{eq:uni2}.
\end {definition}

\nid Again, it helps to look at Figure~\ref{fig:mods} to visualize these definitions. 

If $TR$ is a top-right 2-module over $\R$ and $\T$, we set
$$ \overline{TR} = \prod_{m \geq 0} TR_m.$$
\nid We say that $TR$ is {\em of finite type} if $TR_m = 0$ for all $m \gg 0$, so that the direct product above is a finite direct sum. Note that $\ol{TR}$ is always a module over $\ol{T} = \prod_m \T_m$ (with respect to the horizontal multiplication). If $TR$ is of finite type, then $\overline{TR}$ is also a module over $\ol{R} = \prod R_{m,p}$ (with respect to the vertical multiplication). Similar definitions and remarks apply to the other kinds of 2-modules.

We can tensor together certain types of 2-modules and obtain ordinary (differential graded) modules. 
Precisely, if $TR$ is a top-right 2-module of finite type over $\R$ and $\T$, and $BR$ is a bottom-right 2-module of finite type over $\R$ and $\B$, we can form the tensor product
\begin {equation}
\label {eq:trbr}
 TR \odot_{\R} BR := \ol{TR} \odot_{\ol{\R}} \ol{BR}.
 \end {equation}

A typical generator for the tensor product $TR \odot_{\R} BR$ is of the form $\x \odot \y$, with $\x \in TR_m$ and $\y \in BR_p$. If $\R$ is unital, we automatically have $\x \odot \y = (\x \cdot 1_m) \odot \y = \x \odot (1_m \cdot \y) = 0$ unless $m=p$. Thus we only have generators $\x \odot \y$, with $\x \in TR_m$ and $\y \in BR_m$. In the tensor product, such a generator could be identified with one of the form  $\x' \odot \y'$, with $\x' \in TR_p$ and $\y' \in BR_p$, where $p \neq m$. Thus, in this case we do 
not have a decomposition similar to \eqref{eq:decompose}.

\begin {definition}
 Let $\A$ be a unital sequential 2-algebra, $\T, \R$ and $\B$ unital top, right and bottom algebra-modules over $\A,$ respectively, $TR$ a top-right 2-module of finite type over $\R$ and $\T$, and $BR$ is a bottom-right 2-module of finite type over $\R$ and $\B$. The tensor product $TR \odot_{\R} BR$ is defined as in \eqref{eq:trbr}. It has an induced structure of differential graded right module over the 
differential graded algebra $\alg = \T \odot_\A \B$, as follows:
$$ (\x \odot \y) * (\phi \odot \psi)  := (\x * \phi) \odot (\y * \psi), $$
for all $\x \in TR_m, \y \in BR_m, \phi \in \T_{m'}, \psi \in \B_{m'}.$
\end {definition}

\nid Similarly, if $TL$ and $BL$ are top-left and bottom-left 2-modules of finite type over unital algebra-modules $\T, \L, \B$, we can form the tensor product
$$  TL \odot_{\L} BL :=  \ol{TL} \odot_{\ol{\L}} \ol{BL},$$
which has an induced structure of (differential graded) left module over the algebra $\T \odot_\A \B.$ 

It is not useful to take the horizontal tensor products 
$$ TR \oast_{\T} TL, \ \  \ \ BR \oast_{\B} BL,$$
because, as explained in the previous subsection, $\R \oast_\A \L$ is not even an algebra.

\begin {definition}
\label {def:odotoast}
The chain complex 
\begin {equation}  
\label {eq:bltr}
(TR \odot_{\R} BR) \oast_{(\T \odot_\A \B)} (TL \odot_{\L} BL)\end {equation}
 is called the {\em double tensor product} of $TR, BR, TL, BL$ over $\T, \R, \B, \L$ and $\A$, and is denoted
$$ \odot \oast ( TR, BR, TL, BL \ | \ \T, \R, \B, \L \ | \ \A),$$
or simply $ \odot \oast (TR, BR, TL, BL).$
\end {definition}

\begin {remark}
One can consider more general notions of sequential 2-algebras, algebra-modules, and 2-modules, in which all objects have two subscripts (corresponding to the bottom and top edges of all the rectangles in Figure~\ref{fig:mods}). Precisely, we could define a generalized sequential 2-algebra $\A$ to consist of  chain complexes $\A_{m,p}$ for each $m, p \geq 0$, together with multiplications 
$$ * : \A_{m,p} \otimes \A_{m',p'} \to \A_{m+m', p+p'}, \ \ \ \cdot : \A_{m,p} \otimes \A_{p, q} \to  \A_{m,q}$$
satisfying associativity, local commutation, and Leibniz rules. Our notion of sequential 2-algebra can be recovered as the particular case in which $\A_{m,p} = 0$ unless $m=p$. Similarly, we can define a generalized top algebra-module $\T$ over $\A$ as a collection of chain complexes $\T_{l,m}$ for $l,m \geq 0$, together with operations
$$ * : \T_{l,m} \otimes \T_{l',m'} \to \T_{l+l', m+m'}, \ \ \ \cdot : \T_{l,m} \otimes \A_{m, p} \to  \T_{l,p}$$
satisfying associativity, local commutation, and Leibniz rules. We recover the previous top algebra-modules by imposing the condition $\T_{m,p} = 0$ unless $m=0$. Similarly, we can define generalized algebra-modules of the other kinds, as well as generalized 2-modules, all having two subscripts. These generalized notions are somewhat more natural from a categorical point of view. However, we have emphasized the more specialized definitions (in which some objects have two subscripts and some only one), as they are sufficient for this paper.
\end {remark}

\section {Slicing the strands algebra}
\label {sec:LOT}

\subsection {The strands algebra}
\label {sec:strands}

We review here the definition of the strands algebra $\alg(N, k)$ introduced by Lipshitz, Ozsv\'ath and Thurston in  \cite[Section 3.1]{LOT}. Given integers $N \geq k \geq 0,$ we let $\alg(N, k)$ be the $\k$-vector space generated by triples $(S, T, \phi),$ where $S$ and $T$ are $k$-element subsets of $[N] = \{1, \dots, N\}$, and $\phi: S \to T$ is a bijection satisfying $i \leq \phi(i)$ for all $i \in S.$ To such a triple we can associate a number $\inv(\phi)$, the number of inversions of $\phi$, that is, the number of pairs $i, j \in S$ with $i<j$ and $\phi(j) < \phi(i).$ We then define a multiplication on $\alg(N,k)$ by
$$ (S_1, T_1, \phi_1) * (S_2, T_2, \phi_2) = \begin{cases}
(S_1, T_2, \phi_2 \circ \phi_1) & \text{if } T_1 = S_2 \text{ and } \inv(\phi_2 \circ \phi_1 ) = \inv(\phi_1) + \inv(\phi_2)\\
0 & \text{otherwise.} 
\end {cases} $$

Let $\Inv(\phi)$ denote the set of inversions of $\phi,$ so that $\inv(\phi) = \# \Inv(\phi).$ For $\tau=(i,j) \in \Inv(\phi)$ we let $\phi_\tau$ be obtained from $\phi$ by interchanging the values at $i$ and $j.$ We write $\phi' \prec \phi$ if $\inv(\phi') = \inv(\phi) - 1$ and $\phi'=\phi_\tau$ for some $\tau \in \Inv(\phi)$. We define a differential $\del$ on $\alg(N,k)$ by
$$\del (S, T, \phi) = \sum_{\phi' \prec \phi} (S, T, \phi').$$

Graphically, the generators of $\alg(N,k)$ are drawn as strand diagrams with upward-veering strands. For example,
$$\includegraphics[scale=0.8]{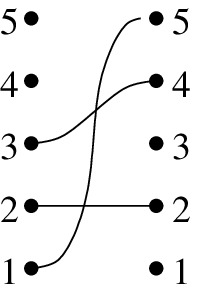}$$
represents the element $(\{1,2,3\}, \{2,4,5\}, \phi) \in \alg(5,3),$ where $\phi(1)=5$, $\phi(2)=2$, and $\phi(3)=4.$

The multiplication $*$ is horizontal concatenation (with double crossings set to zero), and the differential is the sum over all ways of smoothing a crossing---compare Figure~\ref{fig:nilcoxrel}. Since $\partial$ decreases $\inv$ by one, we can define a grading $\cro$ on $\alg(N, k)$ by setting $\cro(S, T, \phi) = \inv(\phi),$ that is, by counting the number of crossings in the corresponding strand diagram---compare Equation~\eqref{eq:cr}. Altogether then, $\alg(N, k)$ is a differential graded algebra.

It is convenient to consider the direct sum
$$ \alg(N) := \bigoplus_{k=0}^N \alg(N, k).$$
\nid In particular we can describe a set of multiplicative generators for $\alg(N).$ For each subset $S \subseteq [N],$ we let $ I_S := (S, S, \id_S)$.  For $1 \leq i < j \leq N,$ we let $ \rho_{i, j}$ be the sum of all triples $(S, T, \phi)$ such that $i \in S, j \in T, S \setminus \{i\} = T \setminus \{j \}, \phi(i) = j,$ and $\phi(r) = r$ for $r\neq i.$ Graphically, $\rho_{i, j}$ should be thought of as a {\em chord} from $i$ to $j.$ The elements $I_S$ and  $\rho_{i, j}$ generate $\alg(N)$ multiplicatively. These generators satisfy the following relations:
\begin {alignat}{2}
\sum_{S \subseteq [N]} I_S &= 1 & \qquad & \label{eq:rel1} \\
I_S * I_S &= I_S & \qquad & \label{eq:rel1.5} \\
I_S * I_T &= 0 & \qquad & \text{ for } S \neq T \label{eq:rel2}\\
I_S * \rho_{i, j} &= 0 &  \qquad & \text{ for } i \not \in S \text{ or } j \in S \label{eq:rel3}\\
\rho_{i, j} * I_S &= 0 &  \qquad & \text{ for } i \in S \text{ or } j \not \in S \label{eq:rel4}\\
\rho_{i, j} * I_{S \cup \{j\}} &= I_{S \cup \{i\}} * \rho_{i, j} &  \qquad & \text{ for } i, j \not\in S \label{eq:rel4.5}\\
I_S * \rho_{i, j} * \rho_{j, l} &= I_S * \rho_{i, l} &  \qquad & \text{ for } i < j < l \text{ and } j \not \in S \label{eq:rel5}\\
\rho_{i, j} * \rho_{l, m} &= \rho_{l, m} * \rho_{i, j}  & \qquad & \text{ for }  j < l \text{ or }  i < l < m < j \label{eq:rel6}\\
\rho_{i, j} * \rho_{l, m} &= 0  & \qquad &  \text{ for } i <  l < j < m. \label{eq:rel7}
\end {alignat}
\nid In fact, the above relations generate all the relations in $\alg(N).$

The differential acts on the generators by 
\begin {equation}
\label {eq:rel8}
\del I_S = 0, \ \ \ 
 \del \rho_{i,j} = \sum_{\{l | i<l<j\}} \rho_{l,j} * \rho_{i,l}.  \end {equation}

 We denote by $\I(N)$ the subalgebra of $\alg(N)$ generated by the idempotents $I_S$, and also let $\I(N, k) = \I(N) \cap \alg(N, k).$

 Given a collection of pairs $\{[i_1, j_1], \ldots, [i_s, j_s]\}$ with $i_r < j_r$ for all $r$ and such that all the $i_r$ are distinct and all the $j_r$ are distinct---such a collection is called a {\em consistent} set of Reeb chords---let $$\rho_{\{[i_1,j_1], \ldots, [i_s,j_s]\}} \in \alg(N)$$ denote the sum of all triples $(S,T,\phi)$ such that $\{i_1, \ldots, i_s\} \subset S$, $\{j_1,\ldots, j_s\} \subset T$, $S \setminus \{i_1, \ldots, i_s\} = T \setminus \{j_1, \ldots, j_s \}$, $\phi(i_r) = j_r$ for all $r$, and $\phi(l) = l$ for $l \nin \{i_1, \ldots, i_s\}$.  The element $\rho_{\{[i_1,j_1], \ldots, [i_s,j_s]\}}$ represents a collection of chords from $i_r$ to $j_r$. If $i_1 < i_2 < \dots < i_s$, we have
 $$ \rho_{\{[i_1,j_1], \ldots, [i_s,j_s]\}} = \rho_{i_s, j_s} * \dots * \rho_{i_1, j_1}.$$
 
\subsection{The strands algebra-modules}
\label {sec:stam}

Our goal here is to prove a decomposition theorem for the strands algebra $\alg(N+N')$, by slicing it into two algebra-modules $\T(N)$ and $\B(N')$ tensored over the nilCoxeter sequential 2-algebra $\nil.$ 

We define the top algebra-module $\T(N) = \{\T(N)_m \}_{m \geq 0}$ over $\nil$ as follows. (Recall Definition~\ref{def:bmod}.) The right $\nil_m$-module $\T(N)_m$ is zero for $m > N.$ For $m \leq N,$ we let 
$$ \T(N)_m = \bigoplus_{k=m}^N \T(N, k)_m. $$
Here the module $\T(N, k)_m$ is free, as a right $\nil_m$-module, with a basis given by triples $(S, T, \psi),$ where $S$ and $T$ are subsets of $[N]$ of cardinalities $k$ and $k-m,$ respectively, and $\psi: T \to S$ is an injective map satisfying $i \geq \psi(i)$ for all $i \in T.$ (In the case $m=0,$ one should think of $\psi$ as the inverse of $\phi$ from the previous definition of $\alg(N, k)=\T(N,k)_0$.)

A basis of $\T(N, k)_m$ over $\k$ is made of the elements $(S,T, \psi) \cdot \sigma,$ with $\sigma = \sigma_w \in \nil_m$ for some $w \in S_m.$ Graphically, we represent such a basis element by a strand diagram with $k$ upward-veering strands, $m$ of which go off at the top. The $m$ free strands at the top form the element $\sigma$ of the nilCoxeter algebra, drawn as in Figure~\ref{fig:nilcox1}, with the convention that $\sigma = 1$ corresponds to no intersections between the free strands. For example,
$$\includegraphics[scale=0.8]{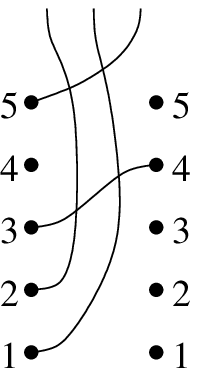}$$
represents an element $(S, T, \psi)\cdot \sigma \in \T(5,4)_3,$ with $S = \{1,2,3,5\}, T = \{4\},$ $\psi(4)=3$, and $\sigma=\sigma_1\sigma_2 \in \nil_3.$ 

We define a multiplication $$* : \T(N,k)_m \otimes \T(N, k-m)_n \to \T(N, k)_{m+n}$$ by horizontal concatenation, with double crossings set to zero. This produces a multiplication $* : \T(N)_m \otimes \T(N)_n \to \T(N)_{m+n}$ by setting products to zero when the number of strands at the interface does not match. The differential on $\T(N, k)_m$ is the sum over all smoothings, as usual; note that this is compatible with the differential on $\nil_m$, in the sense that it makes $\T(N, k)_m$ a differential module over $\nil_m$. The $\zz$-grading $\cro$ on $\T(N)$ is given by the number of crossings in the diagram, just as in $\alg(N).$ Altogether $\T(N)$ is a top algebra-module over $\nil.$ 

Forgetting for a moment about its vertical module structure, we can describe $$\ol \T(N) = \bigoplus_{m=0}^N \T(N)_m$$ as an algebra (with respect to the horizontal $*$ multiplication) in terms of generators and relations, as follows. In addition to the generators $I_S, \ S \subseteq [N]$ and $\rho_{i,j}, 1\leq i < j \leq N$ (similar to those for $\alg(N)$), we now also consider generators $\mu_i, 1 \leq i \leq N.$ Here, $\mu_i$ is the sum of all basis elements $(S, T, \psi) \cdot  1$ that have a unique upward veering strand, and such that their upward strand starts at height $i$ and goes off the top;  that is, $\mu_i$ is the sum of those $(S,T,\psi) \cdot 1$ with $i \notin T$, $S = T \cup \{i\}$, and $\psi(j) = j$, for $j \in T$.  We refer to $\mu_i$ as a {\em half-chord}. The generators $I_S$ and $\rho_{i,j}$ satisfy the same relations \eqref{eq:rel1}-\eqref{eq:rel7} as before, but now we also have the following relations involving $\mu_i$'s:
\begin {alignat}{2}
I_S * \mu_{i} &= 0 &  \qquad & \text{ for } i \not \in  S \label{eq:rel1'}\\
\mu_{i} * I_S &= 0 &  \qquad & \text{ for } i \in S \label{eq:rel2'}\\
I_S * \mu_i &= \mu_{i} * I_{S \setminus \{i \}} &  \qquad & \text{ for } i \in S \label{eq:rel2.5'}\\
I_S * \rho_{i, j} * \mu_{j} &= I_S * \mu_{i} &  \qquad & \text{ for } i < j \text{ and } j \not\in S  \label{eq:rel3'}\\
\mu_{i} * \rho_{j,l} &= \rho_{j,l} * \mu_{i}  & \qquad & \text{ for }  j < l <i \text{ or }  i < j < l \label{eq:rel4'}\\
\rho_{i, j} * \mu_{l} &= 0  & \qquad &  \text{ for } i <  l < j . \label{eq:rel5'}
\end {alignat}

This gives a presentation of $\ol \T(N)$ in terms of generators and relations. Note that we can indeed generate all elements of $\ol \T(N)$, including  the elements $(S, T, \psi) \cdot \sigma$ with $\sigma \in \nil_m$ nontrivial, from the multiplicative generators above. For example, for $i < j,$ we have:
$$ \mu_i * \mu_j = (\mu_j * \mu_i) \cdot \sigma_1.$$

Observe also that the differential acts on $\mu_i$ by
\begin {equation}
\label {eq:rel6'}
 \del \mu_i = \sum_{j > i} \mu_j * \rho_{i,j}.
 \end {equation}

We define a bottom algebra-module $\B(N) = \{\B(N)_m\}_{m \geq 0}$ over $\nil$ in a manner similar to $\T(N)$, except instead of diagrams with outgoing strands escaping at the top, we now use diagrams where strands are allowed to enter at the bottom. We let 
$$\B(N)_m = \bigoplus_{k=0}^{N-m} \B(N, k)_m,$$
where each $\B(N, k)_m$ is free as a left $\nil_m$-module, with a basis given by triples $(S, T, \phi),$ where $S$ and $T$ are subsets of $[N]$ of cardinalities $k$ and $k+m,$ respectively, and $\phi: S \to T$ is an injective map satisfying $i \leq \phi(i)$ for all $i \in S.$ A basis of $\B(N, k)_m$ over $\k$ is made of quadruples $\sigma \cdot (S, T, \phi)$, for $\sigma \in \nil_m.$ We define the horizontal multiplication and the differential as before. A multiplicative basis for $$\ol \B(N) = \oplus_{m=0}^N \B(N)_m$$ is given by elements of the form $I_S, \rho_{i,j}$ and $\nu_j$, where $\nu_j$ is represented by a half-chord starting from the bottom and ending at height $j.$ These generators satisfy relations analogous to \eqref{eq:rel1}-\eqref{eq:rel6'}.

\begin {theorem}
\label {thm:bnt}
For any $N, N' \geq 0$, there is an isomorphism of differential graded algebras:
 $$ \alg(N+N') \cong \T(N) \odot_{\nil} \B(N').$$
 \end {theorem}

\begin {proof}
The generators of 
$$ \T(N, k)_m \odot_{\nil_m} \B(N', k')_m,$$
as a $\k$-vector space, are of the form
$$ (S, T, \psi) \cdot \sigma \cdot (S', T', \phi)$$
with notations as before: $S, T \subseteq [N],\ S', T' \subseteq [N'],$ and so forth. From any such generator we can build a triple
\begin {equation}
\label {eq:cup}
 (S \cup (S'+N), T \cup (T'+N), \psi \cup_\sigma \phi),
 \end {equation}
as follows. The operation $+$ denotes adding the value $N$ to all the elements of a set. The element $\sigma=\sigma_w \in \nil_m$ corresponds to a permutation $w$ on $m$ elements. In turn, $w$ induces  a bijection between any totally ordered sets of $m$ elements each. In particular, it gives a bijection $\ol w$ between $S \setminus \im(\psi) $ and $(T' \setminus \im(\phi)) + N.$ 

We then let $$\psi \cup_\sigma \phi : \bigl( S \cup (S'+N) \bigr) \to \bigl( T \cup (T'+N) \bigr)$$
be defined by
$$\psi \cup_\sigma \phi (i) = \begin{cases}
\psi^{-1}(i) & \text{if } i \in \im(\psi) \subseteq S, \\
\ol w(i) & \text{if } i \in S \setminus \im(\psi), \\
\phi(i-N) + N & \text{if } i \in (S'+N). 
\end {cases}$$

The elements \eqref{eq:cup} form a subcomplex of $\alg(N+N', k+k')$. Indeed, graphically, they correspond to diagrams crossing the horizontal line at height $N+1/2$ exactly $m$ times. For example,
$$\includegraphics[scale=0.8]{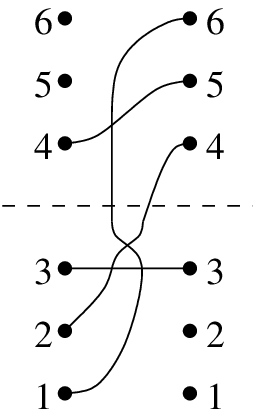}$$
represents an element in $\alg(6,4),$ decomposed as the tensor product of an element in $\T(3,3)_2$ and one in $\B(3,1)_2.$

Considering all possible values of $m, k, $ and $k',$ we obtain the desired isomorphism at the level of chain complexes. Compatibility with the horizontal multiplication, the grading, and the differential is clear. 
\end {proof}

Though Theorem~\ref{thm:bnt} gives a way of describing the whole algebra $\alg(N+N') = \oplus_k \alg(N+N', k)$, each piece $\alg(N+N', k)$ can also be recovered as
$$ \alg(N+N', k) = \bigoplus_{\substack{m, k', k'' \\ k' + k'' = k}} \T(N, k')_m \odot_{\nil_m} \B(N', k'')_m.$$

\subsection {Gradings by non-commutative groups}
\label {sec:ncgrading}

So far we have only discussed $\zz$-gradings on our algebras and algebra-modules, given by the number of crossings in the respective diagrams. However, as explained in \cite[Section 2.5]{LOT}, in bordered Floer homology one has to work with more complicated gradings.

Let $G$ be a group and $\lambda \in G$ a distinguished element in the center of $G$. We refer to $(G, \lambda)$ as a {\em group pair}.  In \cite[Section  2.4]{LOT}, a {\em $(G, \lambda)$-graded differential algebra}  is defined to be a differential algebra $(\alg, *, \del) $ together with a direct sum decomposition $$\alg=\bigoplus_{g \in G} \alg_{\langle g \rangle}$$ such that $$\del(\alg_{\langle g \rangle}) \subseteq \alg_{\langle \lambda^{-1}g \rangle}\ \text { and } \ \alg_{\langle g \rangle} * \alg_{\langle h \rangle} \subseteq \alg_{\langle gh \rangle}.$$ 
Observe that when $G=\zz$ and $\lambda=1$ we recover the usual notion of a $\zz$-grading.

There are two different ways of extending this definition to sequential algebra-modules, and both of them will be needed in this paper. The simplest extension is the following:
\begin {definition}
\label {def:pairgradedT}
Suppose $\A$ is a differential $\zz$-graded sequential 2-algebra, and $\T = \{\T_m\}$ a top algebra-module over it. Let $(G, \lambda)$ be a group pair. We say that $\T$ has a {\em $(G, \lambda)$-grading} if there are direct sum decompositions
$$ \T_m = \bigoplus_{g \in G  } \T_{m, \langle g \rangle}$$
for all $m \geq 0,$ such that 
$$\del(\T_{m, \langle g \rangle}) \subseteq \T_{m, \langle \lambda^{-1}g \rangle}, \ \ \T_{m,\langle g \rangle} * \T_{n, \langle h \rangle} \subseteq \T_{m+n,\langle gh \rangle }, \  \text{and } \  \T_{m, \langle g \rangle} \cdot \A_{m, \langle k \rangle} \subseteq \T_{m, \langle \lambda^{k} g \rangle}.$$
\end {definition}

\nid We can similarly define a grading of a bottom algebra-module by a group pair. If $(G_1, \lambda_1)$ and  $(G_2, \lambda_2)$ are group pairs, we can define their  amalgamated direct product as in \cite[Definition 2.5.9]{LOTbimodules}:
$$ G_1 \times_{\lambda} G_2 := G_1 \times G_2/ \langle \lambda_1 \lambda_2^{-1} \rangle.$$

The tensor of group-pair-graded top and bottom algebra-modules inherits a grading by the amalgamated direct product:
\begin {lemma}
\label {lem:fiberedpair}
Suppose $\A$ is a differential $\zz$-graded sequential 2-algebra.  If we have a $(G_1, \lambda_1)$-graded differential top algebra-module $\T$ over $\A$ and a  $(G_2, \lambda_2)$-graded bottom algebra-module $\B$ over $\A$, then the vertical tensor product $\T \odot_{\A} \B$ has an induced structure of $(G, \lambda)$-graded differential algebra, where $G = G_1 \times_{\lambda} G_2$ and $\lambda = [\lambda_1]=[\lambda_2] \in G.$ 
\end {lemma}
\noindent This lemma follows immediately by unpacking the definitions involved.

An alternative definition of gradings on sequential algebra-modules is the following. We first define a {\em group triple} to be data of the form $(G, \lambda, \tau),$ where $G$ is a group, $\lambda \in G$ is a central element, and $\tau:G \to \zz$ is a group homomorphism such that $\tau(\lambda) = 0$. 

\begin {definition}
\label {def:gradedT}
Suppose $\A$ is a differential $\zz$-graded sequential 2-algebra, and $\T = \{\T_m\}$ is a top algebra-module over it. We say that $\T$ has a {\em grading by the group triple $(G, \lambda, \tau)$} if there are direct sum decompositions
$$ \T_m = \bigoplus_{\{g \in G \mid \tau(g) = m\} } \T_{m, \langle g \rangle}$$
for all $m \geq 0,$ such that 
$$\del(\T_{m, \langle g \rangle}) \subseteq \T_{m, \langle \lambda^{-1}g \rangle}, \ \ \T_{m,\langle g \rangle} * \T_{n, \langle h \rangle} \subseteq \T_{m+n,\langle gh \rangle }, \  \text{and } \  \T_{m, \langle g \rangle} \cdot \A_{m, \langle k \rangle} \subseteq \T_{m, \langle \lambda^{k} g \rangle}.$$
\end {definition}

\nid We can similarly define a grading by a group triple on a bottom algebra-module over $\A$.

If $(G_1, \lambda_1, \tau_1)$ and  $(G_2, \lambda_2, \tau_2)$ are group triples, we can combine the  amalgamated direct product construction over the $\lambda$'s  with the fiber product over the $\tau$'s to define
$$ G_1 \times_{\lambda, \tau} G_2 := \{(g_1, g_2) \in G_1 \times G_2 \mid \tau_1(g_1) = \tau_2(g_2) \}/\langle \lambda_1 \lambda_2^{-1} \rangle.$$

The tensor of group-triple-graded algebra-modules has a grading by the amalgamated direct product group pair:
\begin {lemma}
\label {lem:fibered}
Suppose $\A$ is a differential $\zz$-graded sequential 2-algebra.  If we have $(G_1, \lambda_1, \tau_1)$-graded differential top algebra-module $\T$ over $\A$ and a  $(G_2, \lambda_2, \tau_2)$-graded bottom algebra-module over $\A$, then the vertical tensor product $\T \odot_{\A} \B$ has an induced structure of $(G, \lambda)$-graded differential algebra, where $G = G_1 \times_{\lambda, \tau} G_2$ and $\lambda = [\lambda_1]=[\lambda_2].$ 
\end {lemma}
\noindent Again, this follows directly from the definitions.

Observe that if a sequential algebra-module has a $(G, \lambda, \tau)$-grading, it also has a $(G, \lambda)$-grading. In particular, if we are in the setting of Lemma~\ref{lem:fibered}, then we already know from Lemma~\ref{lem:fiberedpair} that $\T \odot_{\A} \B$ has a grading by $(G_1 \times_{\lambda} G_2, \lambda)$. However, Lemma~\ref{lem:fibered} is stronger, in that it guarantees that the values of the $(G_1 \times_{\lambda} G_2, \lambda)$-grading all lie in the subgroup $G_1 \times_{\lambda, \tau} G_2 \subset G_1 \times_{\lambda} G_2.$

Let us now turn our attention to concrete examples of gradings. In \cite[Section 3.3]{LOT},  the strands algebra $\alg(N)$ is equipped with a grading by $(G'(N), \lambda)$, where the group $G'(N)$ and the element $\lambda$ are as follows. We view $\a =\{1, \dots, N \}$ as a subset of the open interval $Z' = (1/2,N+1/2)$. For $p \in \a$ and $\alpha \in H_1(Z', \a) \cong \zz^{N-1},$ we define the multiplicity $m(\alpha, p)$ of $\alpha$ at $p$ to be the average of the local multiplicity of $\alpha$ just above $p$ and just below $p$. We can extend $m$ to a bilinear map $m: H_1(Z', \a) \times H_0(\a) \to \frac{1}{2}\zz.$ Let us also denote by $\delta: H_1(Z', \a) \to H_0(\a)$ the boundary map. We then define 
\begin {equation}
\label {eq:g'N}
 G'(N) = \{(k, \alpha) \mid k \in \tfrac{1}{2}\zz, \ \alpha \in H_1(Z', \a) \}
 \end {equation}
with the group structure
$$ (k_1, \alpha_1) \cdot (k_2, \alpha_2) = (k_1 + k_2 + m(\alpha_2, \delta \alpha_1), \alpha_1 + \alpha_2).$$
Inside $G'(N)$ we consider the index two subgroup $G''(N)$ generated by the elements $\lambda = (1,0)$ and $(-\frac{1}{2}, [i, i+1])$, for all $i \in \{1, \dots, N-1\}$.\footnote{Note that in \cite[Section 3.1]{LOTbimodules}, the group $G''(N)$ is denoted $G'(\zed).$} Alternately, we can think of $G''(N)$ as the $\zz$-central extension of $H_1(Z', \a)$ with the commutation relation 
$$ \tilde \alpha \tilde \beta = \tilde \beta \tilde \alpha \lambda^{2m(\beta, \delta \alpha)},$$
where $\alpha, \beta \in H_1(Z', \a)$ and $\tilde \alpha, \tilde \beta$ are lifts of $\alpha$ and $\beta$ to $G''(N)$.

A basis element $a = (S, T, \phi) \in \alg(N)$ has an associated homology class given by summing up the intervals corresponding to its strands:
$$[a] := \sum_{s \in S} [s, \phi(s)] \in H_1(Z', \a).$$
\nid We then define 
\begin {equation}
\label {eq:gr'}
  \gr'(a) = (\cro(a) - m([a], S), [a]) \in G'(N).
  \end {equation}
Observe that $\gr'(a)$ always lands in the subgroup $G''(N) \subset G'(N)$. It is proved in \cite[Proposition 3.19]{LOT} that $\gr'$ defines a $(G''(N), \lambda)$-grading on $\alg(N)$, with respect to the central element $\lambda = (1,0)$. 

We can give similar gradings to the strands algebra-modules $\T(N)$ and $\B(N)$ from
Section~\ref{sec:stam}. The gradings will be by group triples, as in Definition~\ref{def:gradedT}. We start with $\T(N)$, and use the $\zz$-grading $\cro$ on the 2-algebra $\nil$. For $Z'=(1/2, N+1/2)$ and $\a$ as above, we define $H_1^\T(Z', \a)$ to be the homology of the relative chain groups where we allow locally finite chains with (possibly non-compact) closed support, but we still require the intersection of their support with the interval $(1/2, N]$ 
to be compact. In other words, $H_1^\T$ is a hybrid between Borel-Moore homology (near the $N+1/2$ end of the interval $Z'$) and ordinary homology (near the $1/2$ end of $Z'$). We have
$$ H_1^\T(Z', \a) \cong \zz^N.$$ 
Indeed, compared to $H_1(Z', \a),$ the group $H_1^\T(Z', \a)$ has the class of $[N, N+1/2)$ as an additional generator. We still have maps $m: H_1^\T(Z', \a) \times H_0(\a) \to \frac{1}{2}\zz$ and $\delta: H_1^\T(Z', \a) \to H_0(\a)$. We define
$$ G'_\T(N) = \{(k, \alpha) | k \in \tfrac{1}{2}\zz, \ \alpha \in H_1^\T(Z', \a) \}$$
with the same multiplication rule as in $G'(N)$. We then let $G''_\T(N) \subset G'_\T(N)$ be the index two subgroup generated by  $\lambda = (1,0)$, $(-\frac{1}{2}, [i, i+1])$ for all $i \in \{1, \dots, N-1\}$, and also $(-\frac{1}{2}, [N, N+1/2))$. 

There is a group homomorphism
$$ \tau_\T : G''_\T(N) \to \zz,$$
which maps $(k, \alpha)$ to the multiplicity of $\alpha$ near $N+1/2$. The kernel of $\tau_{\T}$ is exactly $G''(N)$.

We define the homology class of a basis element $a = (S, T, \psi) \cdot \sigma \in \T(N)_m$ to be
$$ [a] = \sum_{t \in T} [\psi(t), t] + \sum_{s \in S \setminus \im(\psi)} [s, N+1/2) \ \in H_1^\T(Z', \a).$$
Note that $[a]$ does not depend on the nilCoxeter element $\sigma$, but only on $S, T,$ and $\psi.$ Set
$$\gr'(a) = (\cro(a) - m([a], S), [a]) \in G''_\T(N).$$
Clearly $\tau_\T(\gr'(a)) = m$ for $a \in \T(N)_m$. The same argument as in \cite[Proposition 3.19]{LOT} then shows that $\gr'$ defines a $(G''_\T(N), \lambda, \tau_\T)$-grading on $\T(N)$, with $\lambda = (1,0)$. 

For the bottom algebra-module $\B(N)$ over $\nil$, instead of $H_1^\T(Z', \a)$ we use a homology group $H_1^\B(Z', \a),$ defined using chains for which we require that the intersection of their support with the interval $[1, N + 1/2)$ be compact. We define $G''_\B(N)$ and $\gr'$ based on this, and $\tau_\B: G''_\B(N) \to \zz$ as the multiplicity of the homology class near $1/2$. Then $\B(N)$ is $(G''_\B(N), \lambda, \tau_\B)$-graded, with $\lambda = (1,0)$.

For $N, N' \geq 0$, we have an isomorphism
$$ G''(N+N') \cong G''_\T(N) \times_{\lambda, \tau} G''_\B(N'),$$
given by summing up the components in $\frac{1}{2}\zz$, and adding up homology classes for the intervals $(1/2, N+1/2)$ and $(N+1/2, N+N'+1/2).$ Here, the latter interval is identified with $(1/2, N'+1/2)$ using translation. We use the fact that (due to taking the fibered product over the $\tau$'s) we only have to add up relative homology classes that have the same multiplicity near $N + 1/2$. This implies that we can extend their sum over the point $N+1/2$ to arrive at a well-defined class in the (relative) homology of the whole interval $(1/2, N+N'+1/2)$. 

According to Lemma~\ref{lem:fibered}, the vertical tensor product $\T(N) \odot_{\nil} \B(N')$ is $(G''(N+N'), \lambda)$-graded. This is to be expected, because by Theorem~\ref{thm:bnt} the tensor product is isomorphic to $\A(N+N'),$ and we know that the latter algebra is $(G''(N+N'), \lambda)$-graded. In fact,  we have the following result, whose proof is straightforward:

\begin {proposition}
\label {prop:gradedbnt}
The isomorphism in Theorem~\ref{thm:bnt} respects the $(G''(N+N'), \lambda)$-gradings on the two sides.
\end {proposition}

\section {Bordered and cornered Heegaard diagrams}
\label {sec:bchd}

Heegaard diagrams are a convenient way to encode the data determining a 3-manifold. Similar techniques are available for describing 3-manifolds with boundary or with corners, and also for describing 2-manifolds with boundary. In this section we present some of these descriptions. Our motivation for doing so is twofold. First, matching intervals (used to describe 2-manifolds with boundary) make an appearance in Section~\ref{sec:match}, where we prove a decomposition theorem for the algebra associated to a surface in bordered Floer homology. Second, the definition of split cornered Heegaard diagrams (used to describe 3-manifolds with a codimension-two corner) serves as motivation for the four quadrant splitting of a planar grid diagram in Sections~\ref{sec:cpa} and \ref{sec:cpd}.

\subsection{Heegaard diagrams and bordered Heegaard diagrams}
Let $Y$ be a closed, connected, oriented $3$-manifold.  A self-indexing Morse function $f$ on $Y$, having unique index $0$ and index $3$ critical points, determines a Heegaard diagram for $Y$ as follows.  The Heegaard surface is the level surface $\Sigma := f^{-1}(3/2)$, the collection of $\alpha$ curves is the intersection of $\Sigma$ with the unstable manifolds of the index 1 critical points of $f$, and the collection of $\beta$ curves is the intersection of $\Sigma$ with the stable manifolds of the index 2 critical points of $f$. (See \cite[Section 2.1]{HolDisk} for more details.) We denote by $\aalpha$ the collection of $\alpha$ curves and by $\bbeta$ the collection of $\beta$ curves. We then write $Y = Y[\Sigma, \aalpha, \bbeta].$ 

The analogous construction for a $3$-manifold with connected boundary $(Y, \partial Y)$ is the following---see \cite[Section 4.1]{LOT} for details.  Choose a Riemannian metric on $Y$ such that $\del Y$ is totally geodesic. Then choose a self-indexing Morse function $f$ on $Y$ whose gradient is tangent to $\partial Y$, with unique index 0 and index 3 critical points both on $\partial Y$ that are coincident with the unique index 0 and index 2 critical points of $f |_{\partial Y}$, such that the index 1 critical points of $f |_{\partial Y}$ are also index 1 critical points of $f$.  Given such a Morse function, consider the level surface $\Sigma := f^{-1}(3/2)$ together with 
\begin{enumerate}[\quad (1)]
\item the collection $\bbeta$ of curves given by the intersection of $\Sigma$ with the stable manifolds of the index 2 critical points of $f$; 
\item the collection $\aalpha^c$ of curves given by the intersection of $\Sigma$ with the unstable manifolds of the index 1 critical points of $f$ that are in the interior of $Y$; 
\item the collection $\aalpha^a$ of arcs given by the intersection of $\Sigma$ with the unstable manifolds of the index 1 critical points of $f$ that are on the boundary of $Y$.
\end{enumerate}
The triple $(\Sigma, \ol{\aalpha} = \aalpha^a \cup \aalpha^c, \bbeta)$ is exactly the data of a {\em bordered Heegaard diagram} (\cite{LOT}, Def. 4.2). We write $Y=Y[\Sigma, \ol{\aalpha}, \bbeta]$.

\subsection{Matched circles and matched intervals}
Notice that given a 3-manifold with connected boundary $(Y, \partial Y)$ and a self-indexing Morse function as above, the pair $(\partial Y \cap \Sigma, \partial Y \cap \aalpha^a)$ is a circle together with a collection of points arranged in pairs according to which $\alpha^a$ arcs they live on.  This is exactly the data of a matched circle:
\begin{definition}[\cite{LOT}, Def. 3.9]
A \emph{matched circle} is a triple $\zed=(Z, \a, M) $ consisting  of an oriented circle $Z$, a collection $\a= \{a_1, \ldots, a_{4k}\}$ of distinguished points on $Z$, and a  $2$-to-$1$ matching function $M: \a \to [2k]$, which describes that the points $a_i$ and $a_j$ are paired with each other, whenever $M(i)=M(j)$. We require that surgery along these $2k$ pairs of points yields a single circle.
\end{definition}
A matched circle is the surface analogue of a Heegaard diagram: surgery on the pairs of points provides a cobordism from a circle to a circle; capping this cobordism with two discs provides a closed surface $F: = F[Z, \a, M]$.  Conversely, given a surface $F$ with a self-indexing Morse function $f$ with unique index 0 and index 2 critical points, the level manifold $Z := f^{-1}(3/2)$ together with the intersection of $Z$ with the unstable manifolds of the index 1 critical points is a matched circle encoding $F$.

We now describe the analogous notion for surfaces with connected boundary.

\begin{definition}
A \emph{matched interval} is a triple $\cI=(Z, \a, M) $ of an oriented closed interval $Z$, $4k$ points $\a = \{a_1, \dots, a_{4k} \}$ in the interior of $Z$, and a {\em matching}, a $2$-to-$1$ function $M: \a \to [2k]$. We require that performing surgery along the $2k$ matched pairs of points yields a single interval. 
\end {definition}
\noindent See Figure~\ref{fig:matchint} for a picture of a matched interval.
\begin{figure}[ht]
\includegraphics{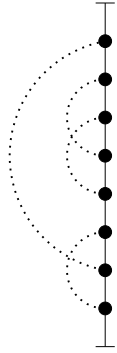}
\caption{{\bf A matched interval.} The matchings are indicated by dotted lines.} \label{fig:matchint}
\end{figure}

A matched interval might be called a ``bordered surface Heegaard diagram", as it encodes the data of a surface with (parametrized) circle boundary.  Specifically, taking a matched interval and performing surgery on all the pairs of points provides a cobordism from an interval to an interval. We can view that cobordism as a surface $F:=F[Z, \a, M]$ with circle boundary.  The Morse theory viewpoint is as follows.  Given a surface $(F, \partial F)$ with circle boundary, choose a Riemannian metric on $F$ such that $\del F$ is geodesic. Then choose a self-indexing Morse function $f$ on $F$, with gradient tangent to the boundary, such that $f$ has unique index 0 and index 2 critical points, which coincide with the unique index 0 and index 1 critical points of $f |_{\partial F}$. The level set $Z := f^{-1}(3/2)$, together with the intersection of $Z$ and the unstable manifolds of the index 1 critical points of $f$, is a matched interval encoding the surface $(F, \partial F)$.  See Figure~\ref{fig:bhdc} for an example.
\begin{figure}[ht]
\includegraphics[scale=1]{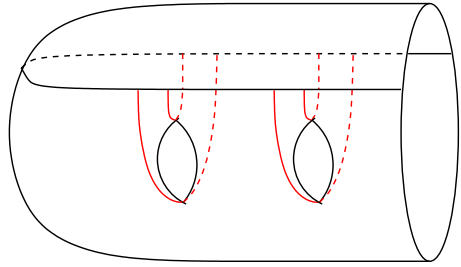}
\caption{{\bf A matched interval and its associated bordered surface.} We depict a surface of genus two with circle boundary, and the embedded matched interval resulting from the height Morse function. The matching intervals are drawn in red.} \label{fig:bhdc}
\end{figure}

\subsection{Cornered Heegaard diagrams}
A 3-manifold with connected boundary and a connected codimension-two corner can also be understood in terms of Heegaard diagrams.
\begin{definition} \label{def-chd}
A \emph{cornered Heegaard diagram} is a quadruple $(\Sigma, P, \ol\aalpha, \bbeta)$, as follows.
\begin{enumerate}[\quad (1)]
\item $\Sigma$ is a surface of genus $g$ with connected boundary $Z$;
\item $P$ is an embedded 0-sphere in $Z$;
\item $\bbeta = \{\beta_1, \ldots \beta_g\}$ is a collection of pairwise-disjoint circles in the interior of $\Sigma$;
\item $\ol\aalpha$ is a collection of pairwise disjoint curves in $\Sigma$, divided into two collections:
\begin{enumerate}[\quad (a)] 
\item $\aalpha^c = \{\alpha^c_1, \ldots, \alpha^c_{g-k}\}$, a collection of circles in the interior of $\Sigma$, and 
\item $\aalpha^a = \{\alpha^a_1, \ldots, \alpha^a_{2k}\}$, a collection of arcs with boundary in $Z \bs P$.
\end{enumerate}
\end{enumerate}
We require that the $\aalpha^a$ arcs intersect $Z$ transversely; that the boundary of every arc $\alpha^a_i$ is contained in a single component of $Z \bs P$; that the intersections of all the curves are transverse; and that $\Sigma \bs \ol\aalpha$ and $\Sigma \bs \bbeta$ are connected.
\end{definition}

\noindent Note that the definition implies that each component of $Z \bs P$ is a matched interval. See Figure~\ref{fig:hdc} for an illustration of a cornered Heegaard diagram.
 \begin{figure}[ht]
\begin{center}
\input{cornered.pstex_t}
\end{center}
\caption {{\bf A cornered Heegaard diagram.} Here $g=3$ and $k=2$. The alpha curves are drawn in red, and the beta curves in blue. The diagram represents the genus two handlebody, with a corner circle splitting the boundary into two genus one surfaces.}
\label{fig:hdc}
\end{figure}
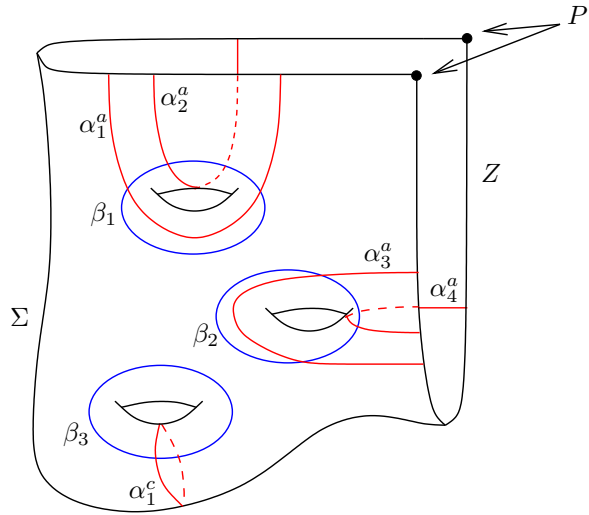

Given a 3-manifold $Y$ with connected boundary $\partial Y$ split along a corner circle $S \subset \partial Y$, choose a Riemannian metric on $Y$ such that the two components of $\del Y \setminus S$  and the circle $S$ are totally geodesic. Then choose a self-indexing Morse function $f$ on $Y$ whose gradient is tangent to $\partial Y$ and to $S$.  Choose $f$ such that it has unique index 0 and index 3 critical points coinciding with the unique index 0 and index 2 critical points of $f |_{\partial Y}$ and also coinciding with the unique index 0 and index 1 critical points of $f |_{S}$.  Moreover insist that the index 1 critical points of $f |_{\partial Y}$ are also index 1 critical points of $f$.  This can all be arranged by first defining $f$ on $S,$ then extending it to $\del Y,$ and finally to $Y$. In this situation, the following data is a cornered Heegaard diagram for $Y$: 
\begin{enumerate}[\quad (1)]
\item the level surface $\Sigma := f^{-1}(3/2)$;
\item $P = \Sigma \cap S$;
\item $\bbeta$ is the intersection of $\Sigma$ with the stable manifolds of the index 2 critical points;
\item[(4a)] $\aalpha^c$ is the intersection of $\Sigma$ with the unstable manifolds of the index 1 critical points in the interior of $Y$;
\item[(4b)] $\aalpha^a$ is the intersection of $\Sigma$ with the unstable manifolds of the index 1 critical points on the boundary of $Y$.
\end{enumerate}

Conversely, starting from a cornered Heegaard diagram $(\Sigma, P, \ol\aalpha, \bbeta)$, by doing surgery  along all the circles and arcs and then attaching 3-balls appropriately, one can construct a 3-manifold $Y :=  Y(\Sigma, P, \ol\aalpha, \bbeta)$ with connected boundary split along  a corner circle.

\begin{remark}
The data $(\Sigma, P, \ol\aalpha, \bbeta)$ of a cornered Heegaard diagram contains in particular the data $(\Sigma, \ol\aalpha, \bbeta)$ of a bordered Heegaard diagram.  Forgetting the embedded 0-sphere $P$ corresponds to smoothing the corner of the 3-manifold with corner $Y[\Sigma, P, \ol\aalpha, \bbeta]$ to obtain a 3-manifold with boundary $Y[\Sigma, \ol\aalpha, \bbeta]$.
\end{remark}

Note that in a cornered Heegaard diagram, as in Definition~\ref{def-chd}, there are no ``$\beta$ arcs"; that is, all the $\beta$ curves are closed, and only $\alpha$ curves intersect the boundary of the surface.  We restrict the notion that way so that it corresponds well with the notion of bordered Heegaard diagram from ~\cite{LOT}---in particular, so that ``smoothing the corner" of a cornered Heegaard diagram produces a bordered Heegaard diagram in the sense of \cite{LOT}.  However, there is a more general notion of bordered Heegaard diagram in which both $\alpha$ and $\beta$ curves are allowed to intersect the boundary.  Similarly, there is a notion of cornered Heegaard diagram that permits both $\alpha$ and $\beta$ curves intersecting the boundary. A particular case of this more general notion is the following:

\begin{definition} \label{def-schd}
A \emph{split cornered Heegaard diagram} is a quadruple $(\Sigma, P, \ol\aalpha, \ol\bbeta)$, as follows.
\begin{enumerate}[\quad (1)]
\item $\Sigma$ is a surface of genus $g$ with connected boundary $Z$;
\item $P$ is an embedded 0-sphere in $Z$, splitting $Z$ into two intervals $Z_{\alpha}$ and $Z_{\beta}$;
\item $\ol\aalpha$ is a collection of pairwise disjoint curves in $\Sigma$, divided into two collections:
\begin{enumerate}[\quad (a)] 
\item $\aalpha^c = \{\alpha^c_1, \ldots, \alpha^c_{g-k}\}$, a collection of circles in the interior of $\Sigma$, and 
\item $\aalpha^a = \{\alpha^a_1, \ldots, \alpha^a_{2k}\}$, a collection of arcs with boundary in the interior of $Z_{\alpha}$;
\end{enumerate}
\item $\ol\bbeta$ is a collection of pairwise-disjoint curves in $\Sigma$, divided into two collections:
\begin{enumerate}[\quad (a)] 
\item $\bbeta^c = \{\beta^c_1, \ldots, \beta^c_{g-l}\}$, a collection of circles in the interior of $\Sigma$, and 
\item $\bbeta^a = \{\beta^a_1, \ldots, \beta^a_{2l}\}$, a collection of arcs with boundary in the interior of $Z_{\beta}.$
\end{enumerate}
\end{enumerate}
We require that the $\aalpha^a$ and $\bbeta^a$ arcs intersect $Z$ transversely; that the intersections of all the curves are transverse; and that $\Sigma \bs \ol\aalpha$ and $\Sigma \bs \ol\bbeta$ are connected.
\end{definition}

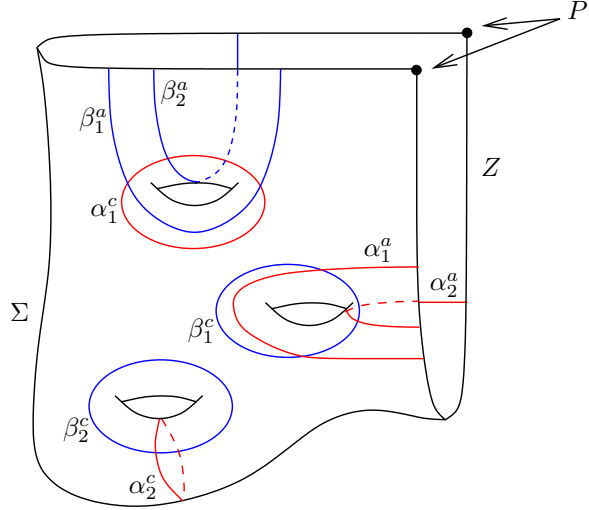
\begin{figure}[ht]
\begin{center}
\input{split.pstex_t}
\end{center}
\caption {{\bf A split cornered Heegaard diagram.} Here $g=3, k=1$ and $l=1.$ This diagram  represents again a genus two handlebody, with a corner circle splitting the boundary into two genus one surfaces.}
\label{fig:sdc}
\end{figure}

The Morse theory perspective on split cornered Heegaard diagrams is entirely analogous to that for cornered Heegaard diagrams, except that we ask for the index 1 critical points of the function $f |_{\partial Y}$ to be either index 1 or index 2 critical points of $f$. Recall that the circle $S$ splits the boundary into two pieces. We denote these components now by $\del_{\alpha}Y$ and $\del_{\beta}Y$, and ask for $\del_{\alpha}Y$ to have only index 1 critical points  of $f$, and for $\del_{\beta}Y$ to have only index 2 critical points of $f$. The unstable manifolds of the index 1 critical points of $f$ that lie on $\partial_{\alpha} Y$ intersect the level surface $\Sigma$ in the $\alpha^a$ arcs, as before, and now the stable manifolds of the index 2 critical points of $f$ that lie on $\partial_{\beta} Y$ intersect $\Sigma$ in the $\beta^a$ arcs. 

To summarize, we can represent a 3-manifold with connected boundary and a circle corner by a cornered Heegaard diagram, or alternatively by a split cornered Heegaard diagram. Both of these notions are candidates for a framework for a cornered Floer homology of 3-manifolds. A cornered Heegaard diagram has the nice feature that smoothing the corner produces a bordered Heegaard diagram. However, the double splitting of planar diagrams that occupies our attention in sections~\ref{sec:vslice},~\ref{sec:cpa}, and ~\ref{sec:cpd} is most closely analogous to the use of split cornered Heegaard diagrams.  

\begin {remark}
In the context of Heegaard Floer homology (say, for closed 3-manifolds), one also has to choose a basepoint $z$ on the Heegaard surface $\Sigma,$ in the complement of all the $\alpha$ and $\beta$ curves. The result is a {\em pointed Heegaard diagram} $(\Sigma, \aalpha, \bbeta, z)$, which represents the pointed 3-manifold $(Y, z)$. Similarly, in bordered Floer homology one chooses a basepoint $z$ on $\del Y \cap \Sigma$, in the complement of the alpha arcs, and the resulting bordered pointed Heegaard diagram represents a 3-manifold $Y$ with boundary and a basepoint on the boundary. We can define cornered (and split cornered) pointed Heegaard diagrams in a similar fashion, by requiring the basepoint $z$ to be one of the two points in $P \subset S$. There are also notions of pointed matched circle and pointed matched interval; for the latter, the basepoint is required to be one of the two endpoints of the interval.
\end {remark}

\subsection{Gluing matched intervals and cornered Heegaard diagrams}
Given two matched intervals $(Z, \a, M)$ and $(Z', \a', M')$, with boundaries $\partial Z = \{\partial Z^+, \partial Z^-\}$ and $\partial Z' = \{\partial Z'^+, \partial Z'^-\}$, we can glue them together by identifying $\partial Z^+$ with $\partial Z'^-$ and $\partial Z^-$ with $\partial Z'^+$, to form a matched circle.  This operation corresponds to gluing the associated surfaces $F[Z, \a, M]$ and $F[Z', \a', M']$ along their boundaries.  

Let $(\Sigma, P, \ol\aalpha, \bbeta)$ and $(\Sigma', P', \ol\aalpha', \bbeta')$ be cornered Heegaard diagrams, with $Z := \partial \Sigma$ and $Z' := \partial \Sigma'$.  Let $(Z \bs P)_1$ and $(Z \bs P)_2$ denote the two components of $Z \bs P$, and similarly for $(Z' \bs P')_1$ and $(Z' \bs P')_2$; all of the $(Z \bs P)_i$ and $(Z' \bs P')_i$ have the structure of matched intervals.  Suppose $\phi: (Z \bs P)_2 \ra (Z' \bs P')_1$ is an orientation-reversing isomorphism of matched intervals.  Then we can glue together $\Sigma$ and $\Sigma'$ along $\phi$ to form a new cornered Heegaard diagram $(\Sigma \cup_\phi \Sigma', P, \ol\aalpha'', \bbeta'')$:
\begin{itemize}
\item $\ol\aalpha'' = {\aalpha^c}'' \cup {\aalpha^a}''$;
\item ${\aalpha^c}''$ is the union of three collections, namely $\aalpha^c$, ${\aalpha^c}'$, and the gluing along $\phi$ of those $\alpha^a$ and ${\alpha^a}'$ curves with boundary on $(Z \bs P)_2$ and $(Z' \bs P')_1$ respectively; 
\item ${\aalpha^a}''$ is the union of those $\alpha^a$ and ${\alpha^a}'$ curves with boundary on $(Z \bs P)_1$ and $(Z' \bs P')_2$; 
\item $\bbeta''$ is the union of $\bbeta$ and $\bbeta'$.  
\end{itemize}
Of course, we can forget the embedded 0-sphere $P$ in $\partial (\Sigma \cup_\phi \Sigma')$ to obtain a bordered Heegaard diagram. At the level of $3$-manifolds, the above gluing of cornered Heegaard diagrams corresponds to the decomposition
\begin {equation}
\label {eq:dec3}
Y[\Sigma \cup_\phi \Sigma', P, \ol\aalpha'', \bbeta''] = Y[\Sigma, P, \ol\aalpha, \bbeta] \cup_{F} Y[\Sigma', P', \ol\aalpha', \bbeta'],
\end {equation}
where $F$ is the surface (with boundary) constructed from the matched interval $(Z\setminus P)_2.$

We can also glue together two split cornered Heegaard diagrams along their $Z_\beta$ boundary components to obtain a bordered Heegaard diagram. This corresponds to a decomposition of a 3-manifold with boundary of the same form as \eqref{eq:dec3}.

\section {Slicing the matching algebra}
\label {sec:match}

The main goal of this section is to prove Theorem~\ref{thm:main1} from the Introduction, giving a decomposition of the
matching algebra of a pointed matched circle as a tensor product of
algebra-modules over the nilCoxeter sequential 2-algebra.

\subsection {The algebra of a matched circle}
\label {sec:matchalg}

We review the definition of the algebra $\alg(\zed)$ associated to a pointed matched circle, from \cite[Section 3.2]{LOT}.  Recall that a pointed matched circle $\zed$ is a matched circle $(Z, \a = \{a_1, \ldots, a_{4k}\}, M: \a \ra [2k])$ together with a choice of basepoint $z \in Z \bs \a$.  Beginning at the basepoint and following the orientation of the circle provides an ordering on the points $\a$, say $a_1 < \cdots < a_{4k}$.  This ordering allows us, in particular, to relate the algebra associated to a matched circle to the strands algebra on the positions $a_1, \ldots, a_{4k}$.

The matching algebra $\alg(\zed)$ is defined as a subalgebra of the strands algebra $\alg(4k)$, more specifically as a subalgebra of $\bigoplus_{i=0}^{2k} \alg(4k,i) \subset \alg(4k)$.  For a subset $s \subset [2k]$ of the pairs of matched points, a set $S \subset \a$ of points is called a section of the matching $M: \a \ra [2k]$ over $s$ if $M$ maps $S$ bijectively onto $s$.  Let $I_S \in \alg(4k,i)$ be the primitive idempotent on $S$, that is the identity element on the subset $S \subset \a$ in the strands algebra; here $i = |S|$.  Set 
\begin{align}
\label {eq:is}
I(s) &:= \sum_{\text{sections } S \text{ over } s} I_S,  \\
\label {eq:is2} I &:= \sum_{s \subset [2k]} I(s). 
\end{align}
Let $I(\zed)$ denote the algebra generated by the elements $I(s)$, for $s \subset [2k]$.  The algebra $\alg(\zed)$ is generated as an algebra by $I(\zed)$ and the elements 
$$a(\{[i_1,j_1], \ldots, [i_s,j_s]\}) := I * \rho_{\{[i_1,j_1], \ldots, [i_s,j_s]\}} * I$$ 
for all consistent sets of Reeb chords $\{[i_1,j_1], \ldots, [i_s,j_s]\}$.  In fact, the nonzero elements in the set 
$$\{I(s) * a(\{[i_1,j_1], \ldots, [i_s,j_s]\}) \: | \: s \subset [2k], \{[i_1,j_1], \ldots, [i_s,j_s]\} \text{ consistent}\}$$
form a $\k$-basis for the algebra $\alg(\zed)$.

An example of such a basis element is 
$$\includegraphics[scale=0.7]{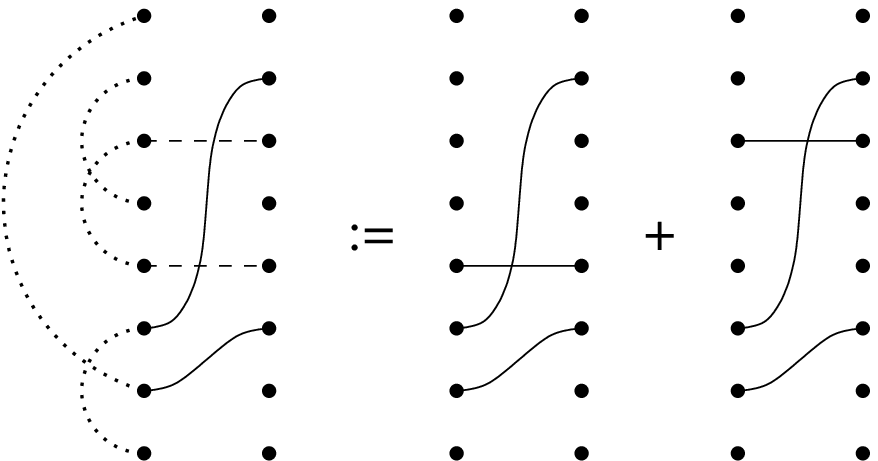}$$

We can describe the elements of the matching algebra more directly as follows.  Recall that a $\k$-basis for the strands algebra $\alg(4k)$ is given by the triples $(S, T, \phi),$ where $S$ and $T$ are subsets of $[4k] = \{1, \dots, 4k\}$, and $\phi: S \to T$ is a bijection satisfying $i \leq \phi(i)$ for all $i \in S$; these elements are conceived of as collections of upward-veering strands connecting position $i$ to position $\phi(i)$.  An element $$m := \sum_{r \in R} (S_r, T_r, \phi_r) \in \alg(4k)$$ of the strands algebra on the positions $a_1, \ldots, a_{4k}$ is an element of the matching algebra $\alg(\zed)$ if and only if  \begin{enumerate}[\quad(a)] \item for all $r \in R$, the set $S_r$ is a section of the matching over $M(S_r)$, \item for all $r \in R$, the set $T_r$ is a section of the matching over $M(T_r)$, and \item if $(S_r, T_r, \phi_r)$ is a summand of $m$ with $\phi_r(i) = i$ for some $i$, then for $i'$ the position matched to $i$, the element $m$ also contains the summand $(S_r', T_r', \phi_r')$ where $S_r' = i' \cup S_r \bs \{i\}$, $T_r' = i' \cup T_r \bs\{ i\}$, and $\phi_r'(j) = \phi_r(j)$ for $j \in S_r \bs \{i\}$ and $\phi_r'(i') = i'$.  \end{enumerate} The first two conditions are that no two chords start or end on matched positions, and the third condition is that the element is invariant under swapping horizontal chords between matched positions.  In these terms a $\k$-basis for the matching algebra $\alg(\zed)$ is given by the sums $\sum_{r \in R} (S_r, T_r, \phi_r) \in \alg(4k)$ satisfying the following conditions:  \begin{enumerate}[\quad(a)] \item as before $S_r$ is a section over $M(S_r)$ and $T_r$ is a section over $M(T_r)$; \item there is a subset $\cc \subset [4k]$ contained in $S_r$ for every $r$, such that  \begin{itemize} \item all the $\phi_r$ agree on $\cc$, and $\phi_r(i) > i$ for $i \in \cc$,  \item $\phi_r(i) = i$ for all $i$ and $r$ with $i \in S_r$ and $i \nin \cc$, \item the cardinality of the index set $R$ is exactly $2^{|S_1 \bs \cc|}$ and the $2^{|S_1 \bs \cc|}$ summands are exactly all the possible iterated horizontal chord swaps of the summand $(S_1, T_1, \phi_1)$---that is, there is one summand $(S \cup \cc, S \cup \phi_1(\cc), \id_S \cup \phi_1 |_{\cc})$ for each section $S$ of the matching over $M(S_1 \bs \cc)$.\end{itemize} \end{enumerate}

As noted in  \cite[Section 3.3]{LOT}, the differential algebra $\alg(\zed)$ does not admit a $\zz$-grading.  It does have a group grading by the group $G''(4k)$ from Section~\ref{sec:ncgrading}, and in fact $\alg(\zed)$ is also graded by a smaller group $G(\zed) \subset G''(4k)$, defined as follows. In the construction of $G''(4k)$ in Section~\ref{sec:ncgrading} we used the interval $Z' = (1/2, 4k+1/2)$. We identify $Z'$ with the complement of the basepoint $z$ in the circle $Z$. Let $H$
denote the kernel of $M_{\ast} \circ \delta : H_1(Z', \a) \ra
H_0([2k])$, and let $F =
F(Z,\a,M)$ be the surface associated to the matched circle.
We can identify $H$ with $H_1(F)$ as follows: given an
indecomposable class $h \in \ker (M_{\ast} \circ \delta)$, the corresponding
class in $H_1(F)$ is the union of $h$ with the core of the 1-handle of $F$ connecting the
two points of $\partial h$. Thus, we get an inclusion $i_*: H_1(F) \to H_1(Z', \a)$. We let $G(\zed)$ be the subgroup of $G''(4k)$ consisting of the pairs $(j, \alpha)$ with $j \in \frac{1}{2}\zz$ and $\alpha \in \im(i_*)$. 

We take $\lambda = (1,0)$ as the distinguished central element. To define a $(G(\zed), \lambda)$ grading $\gr_\psi$ on $\alg(\zed)$, we need to pick additional {\em grading refinement data}, as in \cite[Definition 3.8]{LOTbimodules}. Grading refinement data for $\alg(\zed)$ consists of a function
$$ \psi: \{s \subset [2k]\}  \to G''(4k)$$ 
such that $\psi(s) g' \psi(t)^{-1} \in G(\zed)$ whenever $g' =(k, \alpha) \in G''(4k)$ satisfies $M_*\delta(\alpha) = t - s.$ Given a function $\psi$ like this, we set
$$ \gr_\psi(I(s) * a * I(t)) = \psi(s) \cdot \gr'(a) \cdot \psi(t)^{-1},$$
for any generator $a=I_S * \rho_{\{[i_1,j_1], \ldots, [i_q,j_q]\}} * I_T \in \alg(4k)$ such that $I(s) * a * I(t) \neq 0$ in $\alg(\zed)$ (that is, such that $S$ and $T$ are sections over $s$ and $t$, respectively). Here $\gr'$ denotes the grading on generators defined in Equation~\eqref{eq:gr'}. We note that grading refinement data can be specified by choosing, for each $i=0, \dots, 2k$, a base subset $t_i \subset [2k]$ with $|t_i| = i$, and then defining $\psi(s)$ such that $M_*\delta(\psi(s)) = s-t_i$, for any $s \subset [2k]$ with $|s|=i$. We refer the reader to \cite[Section 3.3.2]{LOT} and  \cite[Section 3.1]{LOTbimodules} for more details.

\subsection {The algebra-modules of a matched interval}
\label {sec:matchalgmod}

Given a pointed matched circle split into two matched intervals, as in Figure~\ref{fig:splitcircle}, we can reconstruct the matching algebra of the circle as a tensor product, over the nilCoxeter sequential 2-algebra, of algebra-modules associated to the two matched intervals.
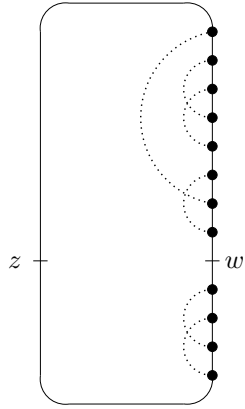
\begin{figure}[ht]
\begin{center}
\input{splitcircle.pstex_t}
\end{center}
\caption {{\bf A pointed matched circle split into two matched intervals.} }
\label{fig:splitcircle}
\end{figure}

As before, let $\zed = (Z, \a, M)$ be a matched circle with a basepoint $z \in Z \bs \a$ and a splitting point $w \in Z \bs \a$ such that both components of $Z \bs \{z,w\}$, with the corresponding points of $\a$, are matched intervals.  Denote those two matched intervals $\cI_1 = (Z_1, \a_1, M_1)$ and $\cI_2 = (Z_2, \a_2, M_2)$, and suppose $Z_1$ is oriented from $z$ to $w$ and $Z_2$ is oriented from $w$ to $z$.  In particular, $Z = Z_1 \cup_{\{z,w\}} Z_2$, $\a = \a_1 \cup \a_2$, and $M = M_1 \cup M_2$.  Let $N_1=4k_1$ be the order of $\a_1$ and let $N_2=4k_2$ be the order of $\a_2$. Let also $F_1 = F(\zed_1)$ and $F_2 = F(\zed_2)$ be the corresponding surfaces with circle boundary.

We associate to the splitting point $w$ the nilCoxeter sequential 2-algebra $\nil$.  To the matched interval $\zed_1$ (with distinguished basepoint $z$) we associate a top algebra-module $\T(\zed_1) = \{T(\zed_1)_m\}_{m \geq 0}$ over $\nil$.  For $m > 2k_1,$ the component $\T(\zed_1)_m$ is zero. For $m \leq 2k_1$, we let 
$$ \T(\zed_1)_m = \bigoplus_{l=0}^m \T(\zed_1, l)_m, $$
where each $\T(\zed_1, l)_m$ is free as a right $\nil_m$-module, with a basis given by formal sums $\sum_{r \in R} (S_r, T_r, \psi_r)$, where $S_r$ and $T_r$ are subsets of $\a_1$ of cardinalities $l$ and $l-m$, and $\psi_r : T_r \ra S_r$ is an injective map with $i \geq \psi(i)$ for all $i \in T_r$, satisfying the conditions: 
\begin{enumerate}[\quad (a)]
\item $S_r$ is a section of the matching over $M_1(S_r)$, and $T_r$ is a section of the matching over $M_1(T_r)$; 
\item there is a subset $\cc \subset \a_1$ contained in $T_r$ for every $r$, such that 
\begin{itemize}
\item all the $\psi_r$ agree on $\cc$, and $\psi_r(i) < i$ for $i \in \cc$, 
\item $\psi_r(i) = i$ for all $i$ and $r$ with $i \in T_r$ and $i \nin \cc$,
\item the cardinality of the index set $R$ is exactly $2^{|T_1 \bs \cc|}$ and 
there is one summand $(S \cup \psi(\cc) \cup (S_1 \bs \psi(T_1)), S \cup \cc, \id_S \cup \psi_1 |_{\cc})$ for each section $S$ of the matching over $M_1(T_1 \bs \cc)$.
\end{itemize}
\end{enumerate}

As in the case of the partial strands algebra-modules in section~\ref{sec:stam}, the triples $(S,T,\psi)$ are represented graphically by upward veering strand diagrams---the strands beginning at $S \bs \psi(T)$ go off at the top.  The multiplication $* : \T(\zed_1, l)_m \otimes T(\zed_1, l-m)_n \to \T(\zed_1, l)_{m+n}$ is given by horizontal concatenation of partial strands, with double crossings set to zero.  This multiplication induces a multiplication $* : \T_m (\zed_1) \otimes \T_n (\zed_1) \to \T_{m+n} (\zed_1)$ by setting the product to zero whenever the number of strands at the interface does not match.  The differential on $\T(\zed_1)$ is the sum over all smoothings of strand crossings. 

In fact, $\T(\zed_1, l)_m$ can be viewed as a subset of the piece $\T(N_1, l)_m$ of the top algebra-module $\T(N_1)$ defined in Section~\ref{sec:stam}. We can call $\T(\zed_1)$ a sub-algebra-module of $\T(N_1)$, since the algebra and the module structures are the restrictions of those on $\T(N_1)$. Furthermore, note that the algebra-module $\T(N_1)$ has idempotents $I_S$ just like in the strands algebra $\alg(N_1)$. Using the matching $M_1$, we can define  idempotents $I_s$ and  $I$ in $\overline \T(N_1) = \oplus_{m=0}^{2k_1} \T(N_1)_m$ by the same formulas as \eqref{eq:is} and \eqref{eq:is2}. With this in mind, the discussion of group gradings on $\alg(\zed)$ from Section~\ref{sec:matchalg} carries over to the top algebra-module $\T(\zed_1).$ We define the group $G(\zed_1)$ as the subgroup of $G''(N_1)$ consisting of pairs $(j, \alpha)$ with $j \in \frac{1}{2}\zz$ and $\alpha$ in the image of $H_1(F_1)$. (Alternately, we could use half-Borel-Moore homology for $Z' = (1/2, N_1 + 1/2)$ and define $G(\zed_1)$ as a subgroup of $G''_{\T}(N_1)$. However, the image of $H_1(F_1)$ would still lie in $H_1(Z'_1, \a_1) \subset H_1^\T(Z'_1, \a_1)$, so we would obtain the same $G(\zed_1) \subset G''(N_1) \subset G''_{\T}(N_1)$.) 

We then define grading refinement data $\psi_1 :  \{s \subset [2k_1]\}  \to G''(N_1)$ 
and a $(G(\zed_1), \lambda)$-grading on $\T(\zed_1)$ by complete analogy with Section~\ref{sec:matchalg}. Note that here we grade our top algebra-module by a group pair, as in Definition~\ref{def:pairgradedT}, rather than by a group triple.

To the matched interval $\zed_2$ (with distinguished basepoint $z$) we associate a bottom algebra-module $ \B(\zed_2)$, whose elements are conceived of as sums of partial strand diagrams with strands entering at the bottom.  Specifically, the algebra-module is $$\B(\zed_2)_m = \bigoplus_{l=0}^{2k_2 - m} \B(\zed_2, l)_m,$$ where each $\B(\zed_2, l)_m$ is free as a left $\nil_m$-module, with a basis given by formal sums $\sum_{r \in R} (S_r, T_r, \phi_r)$, where $S_r$ and $T_r$ are subsets of $\a_2$ of cardinalities $l$ and $l+m$, and $\phi_r : S_r \ra T_r$ is an injective map with $i \leq \phi(i)$ for all $i \in S_r$, satisfying the conditions: 
\begin{enumerate}[\quad (a)]
\item $S_r$ is a section of the matching over $M_2(S_r)$, and $T_r$ is a section of the matching over $M_2(T_r)$; 
\item there is a subset $\cc \subset \a_2$ contained in $S_r$ for every $r$, such that 
\begin{itemize}
\item all the $\phi_r$ agree on $\cc$, and $\phi_r(i) > i$ for $i \in \cc$, 
\item $\phi_r(i) = i$ for all $i$ and $r$ with $i \in S_r$ and $i \nin \cc$, 
\item the cardinality of the index set $R$ is exactly $2^{|S_1 \setminus \cc|}$ and 
there is one summand $(S \cup \cc, S \cup \phi(\cc) \cup (T_1 \setminus \phi(S_1)), \id_S \cup \phi_1 |_{\cc})$ for each section $S$ of the matching over $M_2(S_1 \setminus \cc)$.  
\end{itemize}
\end{enumerate}
Analogously with the top algebra-module, the multiplication is given by horizontal concatenation and the differential is given by smoothing crossings. We can view $\B(\zed_2)$ as a sub-algebra-module of $\B(N_2).$ We then define a group $G(\zed_2)$ as a subgroup of $G''(N_2)$ as in Section~\ref{sec:matchalg}. The bottom algebra-module $\B(\zed_2)$ has a grading by the pair $(G(\zed_2), \lambda)$,  depending on grading refinement data $\psi_2 :  \{s \mid s \subset [2k_2]\}  \to G''(N_2)$. 

Observe that  the inclusions $H_1(F_1) \to H_1(F)$ and $H_1(F_2) \to H_1(F)$ induce an isomorphism
$$ H_1(F_1) \oplus H_1(F_2) \xrightarrow{\cong} H_1(F).$$
From this we get an isomorphism
\begin {equation}
 \label {eq:gzz}
  G(\zed) \cong G(\zed_1) \times_{\lambda} G(\zed_2).
  \end {equation}

We can now state Theorem~\ref{thm:main1} in a more precise form:

\begin {theorem}
Suppose $\zed_1$ and $\zed_2$ are two matched intervals with union the matched circle $\zed = \zed_1 \cup \zed_2$.  The matching algebra of this circle can be recovered as a differential graded algebra by a tensor product over the nilCoxeter sequential 2-algebra:
\begin {equation}
\label {eq:azed}
 \alg(\zed) \cong \T(\zed_1) \odot_{\nil} \B(\zed_2).
 \end {equation}
 
 Moreover, there exist choices of grading
refinement data for $\zed$, $\zed_1$, and $\zed_2$ such that the isomorphism \eqref{eq:azed}
intertwines the $(G(\zed),\lambda)$-grading on $\alg(\zed)$ with the $(G(\zed_1)
\times_{\lambda} G(\zed_2), \lambda) \cong (G(\zed), \lambda)$ grading on $\T(\zed_1) \odot_{\nil} \B(\zed_2)$.
\end {theorem}

\begin {proof} Because the conditions specifying the
subalgebra-module $\mathcal{A}(\zed) \subset \mathcal{A}(N)$ refer exclusively to
properties of the matching that are detected either by the conditions specifying
the subalgebra-module $\mathcal{T}(\zed_1) \subset \mathcal{T}(N_1)$ or
$\mathcal{B}(\zed_2) \subset \mathcal{B}(N_2)$, the isomorphism ~\eqref{eq:azed} of
differential graded algebras is a direct consequence of the proof of the
isomorphism of differential graded algebras $\mathcal{A}(N+N') \cong
\mathcal{T}(N) \odot_{\mathfrak{N}} \mathcal{B}(N')$ from Theorem~\ref{thm:bnt}.

For the statement about the gradings, observe that if we have grading refinement data $\psi_1$ for $\T(\zed_1)$ and $\psi_2$ for $\B(\zed_2)$, we can combine them to get grading refinement data $\psi = \psi_1 \cup \psi_2$ for $\alg(\zed)$, by the formula
$$ \psi(s_1 \cup (s_2+2k_1)) =[( \psi_1(s_1), \psi_2(s_2))] \in G''(N_1) \times_{\lambda} G''(N_2),$$
for any $s_1 \subset [2k_1], s_2 \subset [2k_2].$ Let us grade the algebra-modules $\T(\zed_1)$ and $\B(\zed_2)$ by $\gr_{\psi_1}$ and $\gr_{\psi_2}$, respectively, and let us grade the algebra $\alg(\zed)$ by $\gr_{\psi_1 \cup \psi_2}$. The isomorphism \eqref{eq:azed} then preserves the gradings, as desired.
\end {proof}

We briefly reiterate the field-theoretic viewpoint underlying these constructions.  To the standard circle $S^1$ we are associating the nilCoxeter sequential 2-algebra $\nil$.  To a surface $F_1$ with $S^1$ as outgoing boundary (and with a parametrization encoded as an appropriate Morse function or as a presentation of $F_1$ as the surface associated to a matched interval), we associate a top algebra-module over $\nil$.  Similarly to a surface $F_2$ with $S^1$ as incoming boundary, and an appropriate parametrization, we associate a bottom algebra-module over $\nil$.  This is done in such a way that the matching algebra of the union $F := F_1 \cup_{S^1} F_2$, associated to the parametrized surface $F$ by bordered Heegaard Floer homology, is recovered as a tensor product over the 2-algebra $\nil$. 

Of course, instead of cutting a surface into only two pieces, we could have cut it into more pieces and used top-bottom $\nil$--$\nil$ bimodule-algebras for reconstruction.

\section {Vertical slicing of the planar Floer complex} 
\label{sec:vslice}

For the reader's convenience, we include here a summary of the results of \cite{LOTplanar}, taken almost verbatim from the original reference. The content of \cite{LOTplanar} has served as an informative analogue of bordered Floer homology \cite{LOT}, and it will also form a reference point for our constructions in Sections~\ref{sec:cpa}-\ref{sec:cpd}.

\subsection {The planar Floer complex}
We recall some definitions from \cite[Section 2.1]{LOTplanar}, with a few minor notational changes.

A {\em planar grid diagram} $\H$ of size $N$ is an $(N-1)$-by-$(N-1)$ grid in the plane, composed of $N$ horizontal lines $\alpha_1, \dots, \alpha_N$ and $N$ vertical lines $\beta_1, \dots, \beta_N,$ just as in the geometric interpretation of the nilCoxeter algebra $\nil_N$ in Section~\ref{sec:nilcoxgrid}. However, in order to be consistent with the conventions from Section~\ref{sec:strands}, the alpha curves are ordered from bottom to top, starting with $\alpha_1$ and ending with $\alpha_N.$ In addition, as part of the data for $\H$ we include two sets of markings, $\X = \{X_i\}_{i=1}^{N-1}$ and $\O = \{O_i\}_{i=1}^{N-1},$ with each row and each column containing exactly one $X$ marking and one $O$ marking. 

We let $\S(\H)=\S(\alpha, \beta)$ be the set of $n$-tuples of intersection points $\x = \x_w$ as in Section~\ref{sec:nilcoxgrid}, in one-to-one correspondence with permutations $w \in S_N.$ We call such an $\x$ a {\em planar generator}. For $\x, \y \in \S(\H),$ we let $\Rect(\x, \y)$ be the set of empty rectangles connecting $\x$ to $\y,$ also as in Section~\ref{sec:nilcoxgrid}. (Note that ``empty'' rectangles are allowed to
contain $X$ markings and $O$ markings.)  For $R \in \Rect(\x, \y),$ we set $\X_i(R) = \# (R \cap \{X_i\}) \in \{0,1\}, \ \O_i(R) = \# (R \cap \{O_i\}) \in \{0,1\}, \X(R) = \sum_{i=1}^{N-1} \X_i(R), \O(R) = \sum_{i=1}^{N-1} \O_i(R),$ and
\begin {equation}
\label {eq:ur}
 U(R) = \prod_{i=1}^{N-1} U_i^{\O_i(R)} ,
 \end {equation}
where $U_1, \dots, U_{N-1}$ are formal variables. Further, we consider the set of empty rectangles with no $X$ markings inside, $$\Rectx(\x, \y) = \{R \in \Rect(\x, \y) | \X(R) = 0 \}.$$ 

For $\x \in \S(\H),$ we define the Alexander and Maslov gradings by
\begin {eqnarray*}
 A(\x) &=& \I(\X, \x) - \I(\O, \x) \\
  \mu(\x) &=& \I(\x, \x) - 2\I(\O, \x),
  \end {eqnarray*}
where $\I$ is defined in Equation~\eqref{eq:I}.

Consider the ring
$$\kk = \k[U_1,\dots, U_{N-1}].$$
The planar Floer complex $CP^-(\H)$ associated to $\H$ is the free chain complex over $\kk $ generated by $\S(\H),$ with the differential
$$ \del \x =  \sum_{\y \in \S(\H)} \sum_{R \in \Rectx(\x, \y)} U(R)  \y.$$

\nid We extend the gradings $A$ and $\mu$ to $CP^-(\H)$ by requiring multiplication by a variable $U_i$ to decrease $A$ by one, and $\mu$ by two. The differential $\del$ then decreases $\mu$ by one and keeps $A$ constant.

\subsection {Planar bordered modules and the pairing theorem}
\label {sec:cpad}

Next, we review the main result of \cite{LOTplanar}, which is a reconstruction theorem for the planar Floer complex, with respect to vertical slicing.

Let $\H$ be a planar grid diagram, with horizontal lines $\alpha_i = \{y =i\}$ and vertical lines $\beta_i = \{x=i\}$, for $1 \leq i \leq N$. Let $\ell$ be the vertical line $\{x = k+3/4\}, $ such that the two markings in the $k\th$ column are kept to the left of $\ell$. We denote by $\H^A$ the {\em partial planar grid diagram} consisting of the part of $\H$ to the left of $\ell,$ including the data given by the $X$ and $O$ markings in that half-plane. Similarly, we let $\H^D$ be the partial planar grid diagram consisting of the part of $\H$ to the right of $\ell.$ We refer to $k$ and $N-k$ as the {\em widths} of $\H^A$ and $\H^D$, respectively, and to $N$ as their {\em height}. See Figure~\ref{fig:hahd}.

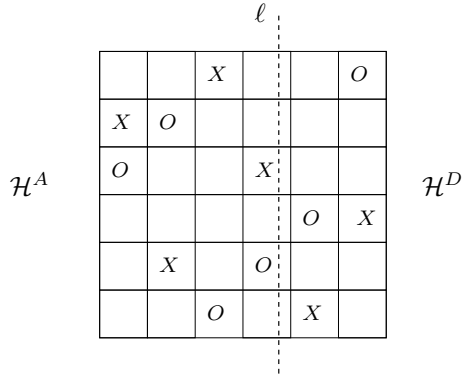
\begin{figure}[ht]
\begin{center}
\input{hahd.pstex_t}
\end{center}
\caption {{\bf Vertical slicing of a planar grid diagram.} Here $N=7$ and $k=4.$}
\label{fig:hahd}
\end{figure}

Define a planar generator of $\H^A$ to be a $k$-tuple $\x^A = \{p_j\in \alpha_{\iota^{\! A}(j)} \cap \beta_j\}_{j=1}^k,$ where $\iota^{\! A}: \{1, \dots, k\} \to \{1, \dots, N\}$ could be any injective function. A planar generator of $\H^D$ is an $(N-k)$-tuple $\x^D = \{p_j \in \alpha_{\iota^{\! D}(j)} \cap \beta_j\}_{j=k+1}^N,$ where $\iota^D: \{k+1, \dots, N\} \to \{1, \dots, N\}$ is again any injection. Observe that if the images of $\iota^{\! A}$ and $\iota^{\! D}$ are disjoint, then $\x^A \cup \x^D$ is a planar generator for the whole grid $\H.$

Let $\S(\H^A)$ be the set of planar generators of $\H^A.$ Lipshitz, Ozsv\'ath, and Thurston define a differential graded right module $CPA^-(\H^A)$ over (the extension to $\kk$ of) the strands algebra $\alg(N, k)$ from Section~\ref{sec:strands}.  The grading on $CPA^-(\H^A)$  will be commutative, in fact by the group $\zz \times \zz$. (By abuse of notation, we will still use $\alg(N,k)$ to denote the extension of the strands algebra from the field $\k$ to the ring $\kk = \k[U_1, \dots, U_{N-1}]$; each variable $U_i$ will be given the grading $(-1,-2)$.) As a chain complex, $CPA^-(\H^A)$ is freely generated over $\kk$ by $\S(\H^A),$ and has the differential
$$ \del \x^A = \sum_{\y^A \in \S(\H^A)} \sum_{R \in \Rectx(\x^A, \y^A)} U(R) \ \y^A,$$
where $\Rectx(\x^A, \y^A)$ is the set of empty rectangles from $\x^A$ to $\y^A$ with no $X$ markings,  just as for the planar generators of the whole grid diagram.

The action of the strands algebra on $CPA^-(\H^A)$ is as follows. Let $\x^A$ be a planar generator of $\H^A, $ with an associated injective map $\iota_{\x^A}: [k] \to [N].$ Recall from  Section~\ref{sec:strands} that $\alg(N)$ is multiplicatively generated by the elements $I_S$ and $\rho_{i, j}$. Similarly, the subalgebra $\alg(N,k)$ is multiplicatively generated by the elements $I_S$, for $|S| = k$, and the elements $\sum_{\{S \, \mid \, |S| = k\}} I_S \rho_{i,j}$.  By abuse of notation, we let $\rho_{i,j}$ refer to both the element $\rho_{i,j} \in \alg(N)$ and to its projection $\sum_{\{S \, \mid \, |S| = k\}} I_S \rho_{i,j} \in \alg(N,k)$.  For $|S| = k$, the action of the idempotent $I_S$ is given by
$$ \x^A*  I_S = \begin{cases}
\x^A & \text{if } S = \im(\iota_{\x^A})\\
0 & \text{otherwise.}
\end {cases} $$ 

Suppose $\x^A = \{p_l \}$ and $ \y^A = \{q_l \}$ are in $\S(\H^A)$ and are such that $p_l = q_l$ for all $l \in \{1, \dots, k\} \backslash l_0$, and $\iota_{\x^A}(l_0) = i < j =\iota_{\y^A}(l_0).$ Let $H$ be the rectangle (called a ``half-strip'') whose left edge is the segment from $p_l$ to $q_l,$ and whose right edge is on the vertical line $\ell.$  If $H$ contains no points $p_l = q_l$ or $X$ markings, we set $\Halfx(\x^A, \y^A; \rho_{i,j}) = \{H\}.$ If $\x^A, \y^A$ do not satisfy the properties above, we set  $\Halfx(\x^A, \y^A; \rho_{i, j}) = \emptyset.$ We define 
$U(H)$ to be the product of the $U_i$ variables corresponding to the $O$ markings inside the half-strip $H.$ We then define the action of a generator $\rho_{i, j} \in \alg(N,k)$ on $CPA^-(\H^A)$ by the formula
$$ \x^A * \rho_{i, j} = \sum_{\y^A \in \S(\H^A)} \sum_{H \in \Halfx(\x^A, \y^A; \rho_{i,j})} U(H) \ \y^A.$$

\noindent See Figure~\ref{fig:cpa}. It is proved in \cite[Proposition 7.1]{LOTplanar} that the differential and the multiplication defined above do indeed turn $CPA^-(\H^A)$ into a differential module over $\alg(N, k).$

\begin{figure}[ht]
\begin{center}
\input{cpa.pstex_t}
\end{center}
\caption {{\bf The basic relation $\x^A * \rho_{i,j} = \y^A $ in the module $CPA^-.$} The black dots represent the components of $\x^A$ and $\y^A,$ while the shaded area is a half-strip $H$ producing the respective term.}
\label{fig:cpa}
\end{figure}
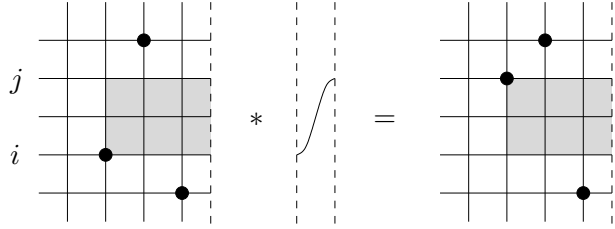 

To make $CPA^-(\H^A)$ into a differential graded module, it remains to equip it with a grading. In fact, we can give it both an Alexander and a Maslov grading (denoted $A$ and $\mu$, respectively, just as for $CP^-(\H)$), such that $\del$ keeps $A$ constant and decreases $\mu$ by one. This is done by setting
\begin {eqnarray*}
 A(\x^A) &=& \I(\X^A, \x^A) - \I(\O^A, \x^A) \\
  \mu(\x^A) &=& \I(\x^A, \x^A) - 2\I(\O^A, \x^A),
  \end {eqnarray*}
where $\O^A$ (resp. $\X^A$) is the sets of $O$ (resp. $X$) markings in the partial diagram $\H^A$.

The differential algebra $\alg(N,k)$ can also be equipped with Alexander and Maslov gradings, as follows. Recall that $\alg(N,k)$ already has a grading $\cro$, which counts the crossings in a strands diagram. Let $L_X, L_O \subset \{3/2, \dots, N-1/2\}$ be the sets of $y$-coordinates of the $X$'s and $O$'s in $\H^A$. Given an algebra element $a \in \alg(N,k)$, viewed as a strand diagram, we let $L_X(a)$ be the total number of intersections between $a$ and the lines $y=l$ for $l \in L_X$. We define $L_O(a)$ similarly, and set
\begin {eqnarray}
\label {eq:alex}
 A(a) &=& L_X(a) - L_O(a) \\
\label{eq:maslov}  \mu(a) &=& \cro(a) - 2L_O(a).
  \end {eqnarray}

\nid The multiplication on $\alg(N,k)$ preserves both $A$ and $\mu$, while the differential preserves $A$ and drops $\mu$ by one. Furthermore, with these choices of gradings, it is proved in \cite[Proposition 7.2]{LOTplanar} that the algebra action of $\alg(N,k)$ on $CPA^-(\H^A)$ preserves both $A$ and $\mu$. 

Let us now turn attention to the diagram $\H^D,$ on the right-hand side of the line $\ell.$ We let $\S(\H^D)$ be the corresponding set of planar generators. Lipshitz, Ozsv\'ath and Thurston define a differential graded left module $CPD^-(\H^D)$ over $\alg(N, k)$ as follows. Let $\kk\langle \S(\H^D) \rangle$ be the free $\kk$-module generated by the elements of $\S(\H^D)$; recall that $\kk=\k[U_1, \dots, U_{N_1}]$. Define an action of the algebra of idempotents $\I(N, k) \subset \alg(N, k)$ on $\kk \langle \S(\H^D) \rangle$ by
$$ I_S * \x^D =  \begin{cases}
\x^D & \text{if } S \cap \im(\iota_{\x^D}) = \emptyset\\
0 & \text{otherwise,}
\end {cases} $$ 
where $\iota_{\x^D} : \{k+1, \dots, N\} \to \{1, \dots, N\}$ is the injection associated to the planar generator $\x^D.$ Then let
$$ CPD^-(\H^D) = \alg(N, k) \otimes_{\I(N, k)} \kk \langle \S(\H^D) \rangle,$$
as a left $\alg(N, k)$-module. Thus, a $\kk$-basis of $CPD^-(\H^D)$ is given by elements of the form $(S, T, \phi) * \x^D$, with $(S, T, \phi) \in \alg(N, k), \x^D \in \S(\H^D),$ and $T \cap  \im(\iota_{\x^D}) = \emptyset.$

To define the differential on $ CPD^-(\H^D),$ for $\x^D, \y^D \in \S(\H^D)$, let us denote by $\Rectx(\x^D, \y^D)$ and $\Halfx(\rho_{i, j}; \x^D, \y^D)$ the sets of empty rectangles, resp. half-strips connecting $\x^D$ to $\y^D$ and containing no $X$ markings. The definitions are analogous to the ones for $CPA^-,$ except we now consider half-strips with the chord $\rho_{i,j}$ on the left.  Then set
$$ \del \x^D = \sum_{\y^D \in \S(\H^D)} \sum_{R \in \Rectx(\x^D, \y^D)} U(R)\ \y^D +  \sum_{\y^D \in \S(\H^D)} \sum_{i, j}  \sum_{H \in \Halfx(\rho_{i, j}; \x^D, \y^D)} U(H) \ \rho_{i,j} *  \y^D.$$

Thus, unlike in $CPA^-,$ where half-strips appear in the algebra multiplication, in $CPD^-$ they play a role in the differential. See Figure~\ref{fig:cpd}. Proposition 6.1 in \cite{LOTplanar} shows that $CPD^-(\H^D)$ is indeed a differential module over $\alg(N, k).$

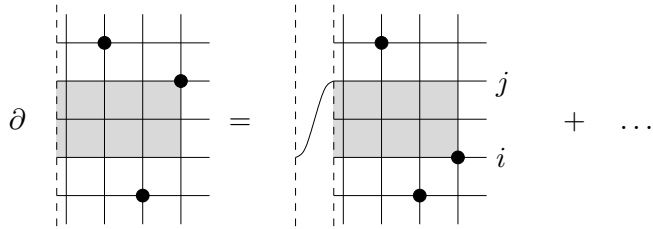
\begin{figure}[ht]
\begin{center}
\input{cpd.pstex_t}
\end{center}
\caption {{\bf The differential $\del \x^D = \rho_{i,j} * \y^D + \dots$ in the module $CPD^-.$} The black dots represent the components of $\x^D$ and $\y^D,$ while the shaded area is a half-strip $H$ producing the respective term in the differential.
}
\label{fig:cpd}
\end{figure}

Define Alexander and Maslov gradings on $CPD^-$ by first setting
\begin {eqnarray*}
 A(\x^D) &=& \I(\X, \x^D) - \I(\O, \x^D) \\
  \mu(\x^D) &=& \I(\x, \x^D) - 2\I(\O, \x^D),
  \end {eqnarray*}
  where $\x \in \S(\H)$ is any generator extending $\x^D \in \S(\H^D)$. Note that $A(\x^D)$ and $\mu(\x^D)$ do not depend
on the extension $\x$ of $\x^D$, nor on the arrangements $\Xs$ and $\Os$ of $X$ and $O$  markings outside of $\H^D$.We then let $A(a \x^D) = A(a) + A(\x^D)$ and $\mu (a \x^D) = \mu(a) + \mu(\x^D)$, for any $a \in \alg(N,k)$ and $\x^D \in \S(\H^D)$. It is proved in \cite[Proposition 6.3]{LOTplanar} that these gradings turn $CPD^-$ into a differential graded module over $\alg(N,k)$.

The following is Theorem 2 in \cite{LOTplanar}:

\begin {theorem}[Lipshitz-Ozsv\'ath-Thurston]
\label {thm:lot2} Let $\H = \H^A \cup_{\ell} \H^D$ be a planar grid diagram of size $N$, vertically sliced into $\H^A$ and $\H^D,$ with the width of $\H^A$ being $k.$ Then, there is an isomorphism of bigraded chain complexes over $\kk$:
$$CP^-(\H) \cong CPA^-(\H^A) \oast_{\alg(N, k)} CPD^-(\H^D). $$
\end {theorem}

The idea behind Theorem~\ref{thm:lot2} is that every planar generator $\x \in \S(\H)$ is of the form $\x^A \cup \x^D,$ for some $\x^A \in \S(\H^A)$ and $\x^D \in \S(\H^D)$ such that the images of their corresponding injections into $\{1, \dots, N \}$ are disjoint; this generator can then be represented in $CPA^-(\H^A) \oast_{\alg(N, k)} CPD^-(\H^D)$ by the tensor $\x^A \oast \x^D$. The differential on $CP^-$ is given by counting empty rectangles (with no $X$ markings). If such a rectangle is supported in $\H^A,$ it appears in the differential for $CPA^-$, whereas if it is supported in $\H^D,$ it appears in the differential for $CPD^-.$ The interaction with the algebra $\alg(N, k)$ comes into play for rectangles that cross the interface $\ell.$ Indeed, if we have such a rectangle from $\x = \x^A \oast \x^D$ to $\y = \y^A \oast \y^D,$ the relations from Figures~\ref{fig:cpa} and \ref{fig:cpd} imply that
$$ \del(\x^A \oast \x^D) = \x^A \oast \bigl(\del(\x^D)\bigr) + \dots = \x^A \oast (\rho_{i,j} * \y^D) + \dots = (\x^A* \rho_{i,j}) \oast \y^D + \dots = \y^A \oast \y^D + \dots, $$
so the relation in $CP^-$ is recovered.

Recall that we defined $\alg(N) = \oplus_k \alg(N, k).$ We can make $CPA^-(\H^A)$ and $CPD^-(\H^D)$ into $\alg(N)$-modules, by letting the elements of $\alg(N, k')$ act by zero whenever $k'\neq k.$ The result of Theorem~\ref{thm:lot2} can be rephrased as:
\begin {equation}
CP^-(\H) \cong CPA^-(\H^A) \oast_{\alg(N)} CPD^-(\H^D).  
\end {equation}

\nid This rephrasing will be convenient for giving a uniform presentation of the cornered decompositions in the next two sections.

\section{Reconstructing $CPA^-$ from cornered 2-modules} \label{sec:cpa}

Recall that in Sections~\ref{sec:sam} and \ref{sec:s2m} we defined the notions of sequential algebra-modules and 2-modules over a sequential 2-algebra. In particular, given a sequential 2-algebra $\A = \{A_m \},$ top, right and bottom algebra-modules $\T = \{\T_m \}, \R = \{R_{m, p}\}, \B = \{B_p \}$ over $\A,$ and top-right and bottom-right 2-modules $TR=\{TR_m\}, BR = \{BR_p\}$ over these, we can form the vertical tensor product 
\begin {equation}
\label {eq:taraboi}
TR \odot_{\R} BR,
\end {equation}
which is a right module over the algebra $\T \odot_{\A} \B.$ 

Let $\H^A$ be a partial grid diagram of height $N$ and width $k$, bounded on the right by a vertical line $\ell,$ as in Section~\ref{sec:cpad}. Our aim is to present a tensor product decomposition of the form \eqref{eq:taraboi} for the differential graded module $CPA^- = CPA^-(\H^A)$ over $\alg(N).$ 

For this purpose, let us consider a horizontal slicing of $\H^A$ by a horizontal line $\ell'$ at height $k'+3/4$, such that all the markings in the $(k')\th$ row of $\H^A$ are below $\ell'.$ We denote the quadrant below $\ell'$ (and left of $\ell$) by $\haa,$ and the quadrant above $\ell'$ by $\had.$ This is illustrated on the left-hand side of Figure~\ref{fig:dada}.

We already know, by Theorem~\ref{thm:bnt}, that the strands algebra $\alg(N)$ can be decomposed as a vertical tensor product over the nilCoxeter sequential 2-algebra:
$$ \alg(N) \cong \T(k') \odot_{\nil} \B(N-k').$$
(We implicitly use here the extensions of the respective objects to the coefficient ring $\kk.$) We will define a right algebra-module $\R(k)$ over $\nil$, a top-right 2-module $\cpaa = \cpaah$ over $\R(k)$ and $\T(k')$, and a bottom-right 2-module $\cpad = \cpadh$ over $\B(N-k')$ and $\R(k)$, such that the vertical tensor product of $\cpaa$ and $\cpad$ produces $\cpa.$ We will refer to $\cpaa$ and $\cpad$ as {\em planar cornered 2-modules}.  (We place braces around $AA$ in order to distinguish our notation from $CPAA^-.$ The latter is already in use in the literature as notation for a different object, a bimodule appearing in \cite[Appendix A]{LOT}.)

\subsection {The right algebra-module}
\label {sec:ram}
We start by defining the algebra-module 
$$\R(k) = \{\R(k)_{m,p} | m, p \geq 0\}$$ 
over the nilCoxeter 2-algebra $\nil.$ The definition is reminiscent of, but slightly more complicated than, those of $\T(N)$ and $\B(N)$ from Section~\ref{sec:stam}. Roughly, a generator of $\R(k)$ will be a picture containing the following: rightward veering strands on up to $k$ spots, accounting for differential rectangles in $\haa$ across the horizontal split; a nilCoxeter element to account for horizontally splitting the strands algebra;  and an interaction between the strands and the nilCoxeter element (in the form of caps) to account for differential rectangles that cover the double  splitting  point, i.e. the intersection of the vertical and horizontal cuts.

Precisely, each chain complex $\R(k)_{m,p}$ is nonzero only for $0 \leq m-p \leq k$, in which case it has an additional decomposition
$$ \R(k)_{m,p} = \bigoplus_{s=m-p}^k \R(k, s)_{m,p}.$$

\nid As a $\kk$-module, the summand $\R(k,s)_{m,p}$ is freely generated by data of the form $(S, T, S', P, \phi, \psi, \psi'),$ where:
\begin {itemize}
\item $S, T \subseteq [k] = \{1, \dots, k\}$ are subsets with $s-(m-p)$ elements each;
\item $S'$ is a subset of $[k]$ with $m-p$ elements such that $S \cap S' = \emptyset$; 
\item $\phi: S \to T$ is a bijection satisfying $\phi(i) \geq i$ for all $i\in S;$ 
\item $P$ is a subset of $[m]$ of cardinality $p;$ 
\item $\psi: P \to [p]$ and $\psi' : ([m] \setminus P) \to S'$ are bijections.
\end {itemize}

Graphically, we represent such a generator by a rectangular diagram split into two halves by a vertical dotted line. To the left of this line we have $k$ bullets at the bottom and $k$ at the top, and $s-(m-p)$ rightward-veering strands joining some of the bottom bullets (corresponding to the subset $S$) to some of the top ones (corresponding to the subset $T$); the strands link bullets according to the bijection $\phi$. To the right of the dotted line we have $m$ bullets at the bottom and $p$ at the top, with $p$ of the bottom ones (corresponding to $P$) joined with the top ones according to $\psi.$ Finally, $m-p$ of the bottom bullets to the left (corresponding to $S'$) are joined with the remaining bottom bullets to the right according to $\psi'$. For example,
$$\includegraphics[scale=0.8]{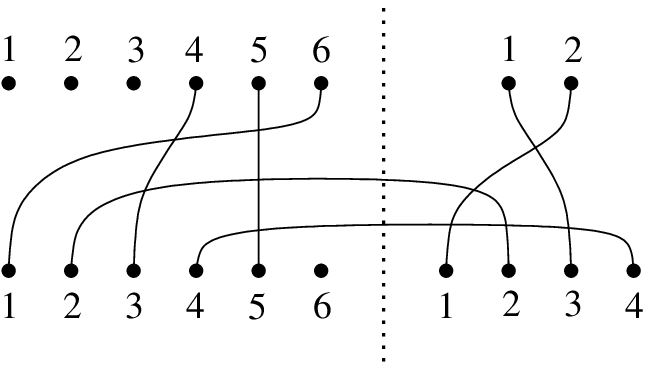}$$
represents a generator $(S, T, S', P, \phi, \psi, \psi')$, with $k=6, s=5, m=4, p=2, S=\{1,3,5\}, T=\{4,5,6\}, S' =\{2,4\},$ and $P = \{1,3\}.$

We define a differential on $\R(k,s)_{m,p}$ as the sum over all ways of smoothing a crossing, analogously to the differentials on $\nil_m, \alg(N, k),$ etc.; note that for each crossing there is only one smoothing resulting in another element of $\R(k,s)_{m,p}$. In the same vein, we define a vertical multiplication:
$$\cdot: \R(k,s)_{m,m'} \otimes \R(k, s-m+m')_{m',p} \to \R(k,s)_{m,p}$$ 
by vertical concatenation of diagrams, with double crossings being set to zero. This induces a vertical multiplication
$$ \cdot: \R(k)_{m,m'} \otimes \R(k)_{m', p} \to \R(k)_{m,p}, $$
as required in the definition of a sequential right algebra-module, by setting products to zero when not previously defined (that is, when the $s$ values do not correspond appropriately).

Finally, we define a horizontal multiplication
$$  * : \R(k,s)_{m,p} \otimes \nil_n \to \R(k,s)_{m+n, p+n}, $$
by horizontal concatenation. Note that the $n$ strands from the algebra element in $\nil_n$ do not directly interact with the ones in $\R(k,s)_{m,p}$; however, they will interact later when we consider cornered 2-modules. Altogether, these operations turn $\R(k)$ into a unital sequential right algebra-module over $\nil.$ Note that since $\R(k)_{m,p} = 0$ for $m < p$, the algebra-module $\R$ is multiplicatively finite in the sense of Definition~\ref{def:mf}.

Recall that in Section~\ref{sec:stam} we gave a generator-and-relations presentation for the top algebra-module $\T(N),$ with respect to the  horizontal multiplication. We now give a similar presentation for $\R(k),$ with respect to the vertical multiplication.

Note that, unlike $\ol \T(N) = \oplus_{m=0}^N \oplus_{k=m}^N \T(N,k)_m$, the total algebra $\ol \R(k) = \prod_{m,p,s} \R(k,s)_{m,p}$ is an infinite direct product, so by generators we will mean ``generators in the sense of direct products'', that is, we only ask that an arbitrary element can be expressed as an infinite sum of generators. With that in mind, we consider the following generators: 
\begin {itemize}
\item idempotents $J_{S, m}=(S, S, \emptyset, [m], \id_S,\id_{[m]}, \emptyset),$ for any $S \subseteq [k]$ and $m \geq 0$;
\item chords $\lambda_{i,j}$ going from the $i\th$ bottom to the $j\th$ top bullet on the left-hand side of the dotted line, for $i < j$; these are the analogues of $\rho_{i,j}$ from Section~\ref{sec:strands}, except now each $\lambda_{i,j}$ is an {\em infinite} sum of all septuples differing from a $J_{S, m}$ only in positions $i$ and $j$ on the left;
\item nilCoxeter elements $\sigma_i$ for each $i \geq 1;$ each $\sigma_i$ is an infinite sum of all elements $\sigma_{S, m, i}$ for $m > i$, where $\sigma_{S, m, i}$ is obtained from the idempotent $J_{S, m}$ by introducing a twist between the $i\th$ and the $(i+1)\th$ strands on the right-hand side of the dotted line;  
\item {\em caps} of the form $\xi_{i}$, drawn as chords joining the $i\th$ bottom bullet on the left to the first bottom bullet on the right; that is, $\xi_{i}$ is the infinite sum of all $(S, S,\{i\}, [m] \setminus \{1 \}, \id_S, \psi, \psi')$, where $\psi'$ is the unique bijection and $\psi$ is the unique order-preserving bijection.    
\end {itemize}

These generators are subject to a number of relations as follows. First, we have analogues of \eqref{eq:rel1}-\eqref{eq:rel2}:
\begin {equation}
\prod_m \sum_{S \subseteq [k]}  J_{S,m} = 1,
\end {equation}
and
\begin {equation}
\label {eq:rel1''}
J_{S, m} \cdot J_{T, n} = \begin{cases} J_{S,m} & \text{if } S=T \text{ and } m=n; \\
0 & \text{otherwise.} \end {cases}
\end {equation}

\nid We also have exact analogues of \eqref{eq:rel3}-\eqref{eq:rel7}, with the $\cdot$ multiplication replacing the $*$ multiplication, $J_{S,m}$ replacing $I_S$, and the horizontal chords $\lambda_{i,j}$ replacing the vertical chords $\rho_{i,j}.$ 

Next, the nilCoxeter elements $\sigma_i$ satisfy the usual nilCoxeter relations \eqref{eq:nc1}-\eqref{eq:nc3}, and interact with the chords and the idempotents by:
\begin {alignat}{2}
J_{S,m} \cdot \sigma_i &= \sigma_i \cdot J_{S, m} & \qquad & \text{ for all } S \text{ and } i \label{eq:rel2''}\\
J_{S,m} \cdot \sigma_i &= 0 &  \qquad & \text{ for } m \leq i\label{eq:rel3''}\\
\lambda_{i, j} \cdot \sigma_l &= \sigma_l \cdot \lambda_{i,j} &  \qquad & \text{ for all } i, j, l. \label{eq:rel4''}\end {alignat}

Further, we have new relations involving the caps $\xi_{i}$:
\begin {alignat}{2}
J_{S, m} \cdot \xi_{i} &= 0 & \qquad & \text{ for } i \not\in S \text{ or } m=0 \label{eq:rel5''} \\
\xi_{i} \cdot J_{S, m} &= 0 & \qquad & \text{ for } i \in S   \label{eq:rel6''} \\
\xi_{i} \cdot J_{S, m} &= J_{S\cup \{i\}, m+1} \cdot \xi_{i}  & \qquad & \text{ for }  i \not \in S \label{eq:rel7''} \\
J_{S,m} \cdot \lambda_{i, j} \cdot \xi_{j} &= J_{S, m} \cdot \xi_{i} &  \qquad & \text{ for } i < j \text{ and } j \not \in S \label{eq:rel8''}\\
\lambda_{i,j} \cdot \xi_{l}  &= \xi_{l} \cdot \lambda_{i,j}  & \qquad & \text{ for }  i < j <l \text{ or } l< i < j \label{eq:rel9''}\\
\lambda_{i, j} \cdot \xi_{l} &= 0  & \qquad &  \text{ for } i <  l < j  \label{eq:rel10''} \\
\xi_i \cdot \xi_j &= \sigma_1 \cdot \xi_j \cdot \xi_i & \qquad &  \text{ for } i < j \label{eq:rel11''} \\
\xi_{i} \cdot \sigma_v &= \sigma_{v+1} \cdot \xi_{i} & \qquad &  \text{ for } v \geq 1.  \label{eq:rel12''}  
\end {alignat}

These generators and relations give a presentation for the algebra $\ol \R(k)$ (in the sense of infinite direct products). We should also mention the effect of the differential on generators: 
\begin {alignat}{2}
\del J_{S,m} &=0 & \qquad & \label{eq:rel15''}\\
\del \lambda_{i,j} &= \sum_{\{l | i<l<j\}} \lambda_{l,j} \cdot \lambda_{i,l} & \qquad & \label{eq:rel16''}\\
\del \sigma_i &= \sum_{S \subseteq [k]} \sum_{m > i} J_{S,m} & \qquad & \label{eq:rel17''}\\
\del \xi_{i} &= \sum_{j > i} \xi_{j} \cdot \lambda_{i,j}. & \qquad & \label{eq:rel18''}
\end {alignat}

\subsection {The cornered 2-module $\cpaa$} \label {sec:cpaa}
Next, we define a top-right 2-module $$\cpaa = \{\cpaa_m\}_{m \geq 0}$$ over $\R(k)$ and $\T(k')$.   Consider the quadrant $\haa$ of width $k$ and height $k'$. A partial planar generator $\xaa$ of $\haa$ is defined to be a collection of points $\{p_i \in \alpha_i \cap \beta_{\eps(i)} | i \in S \}$, where $S= S_{\xaa} \subseteq [k']$ is a subset, and $\eps=\eps_{\xaa} : S \to [k]$ is an injection. Note well that S is indeed a subset of
the $\alpha$ lines, not a subset of the $\beta$ lines as had been our convention
until this moment. Whichever choice is made conflicts with our established
notation for either the top algebra-module or the right algebra-module; we pick
notation that aligns well with the top algebra-module, at the expense of the
more recently described right algebra-module. Observe also that there may be both rows and columns in $\haa$ unoccupied by components of $\xaa.$ We will denote the set of partial planar generators in $\haa$ by $\S(\haa)$.

As a $\kk$-module, we let $\cpaa_m$ be freely generated by triples of the form $(\xaa, M, \sigma),$ where:
\begin {itemize}
\item $\xaa$ is a partial planar generator with a corresponding injection $\eps_{\xaa}: S_{\xaa} \to [k],$ \item $M$ is a subset of $[k]$ of cardinality $m$ that is disjoint from the image of $\eps_{\xaa};$
\item $\sigma = \sigma_w \ (w \in S_m)$ is a standard basis element of the nilCoxeter algebra $\nil_m$.   
\end {itemize}

Graphically, we represent a triple $(\xaa, M, \sigma)$ by the collection $\xaa$ of points in the plane, together with $m$ arrows on the top edge of $\haa,$ marking the elements of $M,$ and a small rectangular box at the top right of the diagram $\haa,$ containing strands that encode the element $\sigma$. An example is shown in Figure~\ref{fig:cpaa}. The nilCoxeter element $\sigma$ is associated to a permutation $w$ of $m$ objects. In the graphical representation, each arrow in $M$ is paired with one of the strands in $\sigma$, with the leftmost arrow paired with the strand starting at the leftmost point of the bottom of the dotted square, and so forth. 

\begin{figure}[ht]
\begin{center}
\input{cpaa.pstex_t}
\end{center}
\caption {{\bf A generator of $\cpaa.$} In this example, the quadrant $\haa$ has width $k=6$ and height $k'=5.$ The partial planar generator $\xaa$ is indicated by the two black dots, and has a corresponding injection $\eps: \{2,4\} \to \{1, \dots, 6\}$ with $\eps(2)=6, \eps(4)=3.$ The set $M=\{1,4,5\},$ is indicated by the arrows, while the element $\sigma =\sigma_2\sigma_1 \in \nil_3$ is drawn in the top-right dotted square. We view each arrow as being paired with a corresponding strand in $\sigma.$}
\label{fig:cpaa}
\end{figure}

\begin {remark}
The intuition behind the definition of $\cpaa$ is the following. One should think of the arrows in $M$ as extra black dots (i.e., contributions to $\xaa$) that occupy the respective columns but have been moved ``to infinity'' in the vertical direction. This is justified by the action of half-chords on $\cpaa,$ as in Figure~\ref{fig:cpaa2} below. An arrow in $M$ is therefore the trace of a black dot in that column, after the action of a half-chord which moved it to infinity. The nilCoxeter element $\sigma$ is a permutation of these arrows, encoding the order in which they were moved to infinity;  more specifically, the left to right order of the strands
at the bottom of $\sigma$ is the left to right order of the arrows in $M$, while
the left to right order of the strands at the top of $\sigma$ is the order in
which the dots were moved to infinity.\end {remark}

We define a differential on $\cpaa$ by combining the differential on  the nilCoxeter algebra with the usual rectangle-counting differential from $CPA^-$, and a new term counting certain half-strips, as follows.
Suppose $\yaa \in \S(\haa)$ is obtained from $\xaa$ by shifting one component $p_i$ from the column $\beta_j$ to the column $\beta_{l}$ (in the same row $\alpha_i$), with $j < l.$ We then let $H$ be the half-strip bounded by $\beta_j, \alpha_i, \beta_l$ and $\ell'.$ If $H$ has no components of $\xaa$ or $X$ markings inside, we set $\Halfx(\xaa, \yaa; \lambda_{j,l}) = \{H\},$ and define $U(H)$ as in Section~\ref{sec:cpad}. In all other cases, we set $\Halfx(\xaa, \yaa; \lambda_{j,l}) = \emptyset.$ The differential is defined as the sum
\begin {equation}
\label {eq:delcpaa}
\del = \del_1 + \del_2 + \del_3,
 \end {equation}
where, for a triple $(\xaa, M, \sigma)$, we set
\begin {eqnarray*}
\del_1(\xaa, M, \sigma) &=& (\xaa, M, \del \sigma) \\
\del_2(\xaa, M, \sigma) &=& \sum_{\yaa \in \S(\haa)} \sum_{R \in \Rectx(\xaa, \yaa)} U(R) \ (\yaa, M , \sigma) \\
 \del_3(\xaa, M, \sigma) &= &  \sum_{\yaa \in \S(\haa)} \sum_{j \not \in M} \sum_{ l\in M} \sum_{H \in \Halfx(\xaa, \yaa; \lambda_{j,l})} U(H) \ (\yaa, M \cup \{j\} \setminus \{l\}, \sigma^{M, j, l}). 
\end {eqnarray*}

In the last equation, $\sigma^{M, j, l}$ denotes the element of $\nil_m$ given by
\begin {equation}
\label {eq:smjl}
 \sigma^{M, j, l} =  (\sigma_{j'+1} \sigma_{j'+2} \dots  \sigma_{l'-1}) \cdot \sigma,
 \end {equation}
where $j' = \# \{u \in M | u < j \}$ and $l' =  \# \{u \in M | u \leq l \}.$ In other words, in the graphical representation, turning $\sigma$ into $\sigma^{M, j,l}$ means that the strand of $\sigma$ paired with the arrow at $l \in M$ becomes paired with the one at $j$, so its starting point gets moved a few positions to the left using pre-multiplication with $ \sigma_{j'+1} \sigma_{j'+2}  \dots \sigma_{l'-1}.$ See Figure~\ref{fig:Dcpaa} for an example.

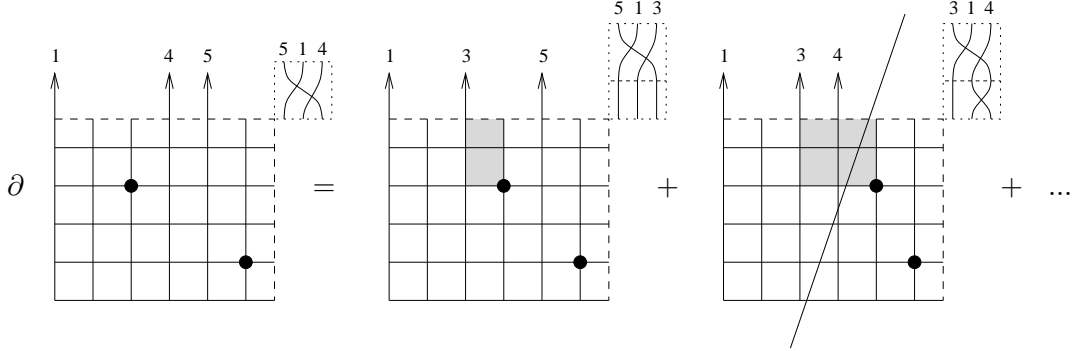
\begin{figure}[ht]
\begin{center}
\input{Dcpaa.pstex_t}
\end{center}
\caption {{\bf The differential of the generator from Figure~\ref{fig:cpaa}.} We show here two contributions to the differential coming from the third term, $\del_3,$ in Equation~\eqref{eq:delcpaa}. The shaded areas represent the respective half-strips $H$. Note that the second contribution is actually zero, because of the double crossing in the nilCoxeter element. The differential also has two other contributions from $\del_1$, not shown here, which correspond to smoothing crossings in the original nilCoxeter element.}
\label{fig:Dcpaa}
\end{figure}

\begin {remark}
At this point the reader may wonder why we included the term $\del_3$ (counting half-strips) in the differential. Since $\cpaa$ is meant to be a module of type $A$ in both the horizontal and vertical directions, one expects that half-strips would appear in the action, not in the differential. In fact, half-strips without arrows on their top edge are included in the action and not in the differential. (See Equation~\eqref{eq:aar2} below.) The half-strips counted in $\del_3$ have an arrow at their top-right corner; conceiving of the arrow
as a dot at infinity, these half strips may be thought of as lying fully inside
the quadrant $\haa$.  They are more analogous to the rectangles that appear in
$\partial_2$ than ordinary half strips, so it is natural to count them in the differential of $\cpaa$. The necessity of the term $\del_3$ can also be justified by analyzing the Leibniz rule for the action of a half-chord: see Figure~\ref{fig:cpaa_leibniz}, establishing Equation~\eqref{eq:cpaa_leibniz} below.
\end {remark}

\begin{lemma}
\label {lemma:dcpaa}
The operation $\del = \del_1 + \del_2 + \del_3$  defines a differential on $\cpaa,$ that is, $\del^2 = 0.$ 
\end {lemma}

\begin {proof}
We already know that $\del_1^2 = 0.$ Further, we have $\del_2^2 = 0$ by the same proof as for closed grids, see \cite[Proposition 2.8]{MOST}: contributions from two noninteracting rectangles in the two possible orders cancel out, as do contributions from the two ways of cutting an L-shaped figure made of two rectangles.  We also have $\del_1 \del_2 + \del_2  \del_1 = 0,$ because contributions from rectangles and nilCoxeter elements do not interact.

The last step is to show that 
\begin {equation}
\label {del13}
\del_1 \del_3 + \del_3 \del_1 + \del_2 \del_3 + \del_3 \del_2 + \del_3^2 = 0.
\end {equation}

\nid We claim that when we evaluate the left-hand side of \eqref{del13} at a generator $(\xaa, M, \sigma),$ all terms cancel in pairs. 

Indeed, any contributions to $\del_3 \del_1 (\xaa, M, \sigma)$ are of the form $U(H) (\yaa, M \cup \{j\} \setminus \{l\}, (\del \sigma)^{M, j, l})$, coming from half-strips $H \in \Halfx(\xaa, \yaa; \lambda_{j,l})$. Since
\begin {equation}
 \label {eq:delsi}
 \del(\sigma^{M, j, l}) = (\del \sigma)^{M, j, l} + \del(\sigma_{j'+1}\sigma_{j'+2} \dots \sigma_{l'-1}) \cdot \sigma,
 \end {equation}
these contributions also appear either in 
$$U(H) (\yaa, M \cup \{j\} \setminus \{l\},\partial(\sigma^{M,j,l}))$$ or in $$U(H) (\yaa, M \cup \{j\} \setminus \{l\}, \del(\sigma_{j'+1}\sigma_{j'+2} \dots \sigma_{l'-1}) \cdot \sigma).$$
In the first case, the contribution appears in $\partial_1 \partial_3 (\xaa, M, \sigma)$ as shown in Figure~\ref{fig:deldel1}(a).  In the second case, the contribution appears in $\partial_3^2(\xaa, M, \sigma)$ as in Figure~\ref{fig:deldel1}(b).  Similarly, any contribution to $\partial_1 \partial_3 (\xaa, M, \sigma)$ is of the form $U(H) (\yaa, M \cup \{j\} \setminus \{l\}, \partial(\sigma^{M,j,l}))$.  By Equation~\eqref{eq:delsi}, this contribution occurs either in $\partial_3 \partial_1 (\xaa, M, \sigma)$ as in Figure~\ref{fig:deldel1}(a) or in $\partial_3^2(\xaa, M, \sigma)$ as in Figure~\ref{fig:deldel1}(c).

\begin{figure}[p]
\begin{center}
\input{deldel1.pstex_t}
\end{center}
\caption {{\bf Cancellations between $\del_1 \del_3$, $\del_3 \del_1$, and $\del_3^2$.} The shaded area represent the respective half-strips $H$ (from all previous steps), which produce the factors $U(H)$. Note that in case (b), there is an additional contribution to $\del_3 \del_1$ (not drawn here), which cancels a contribution to $\del_1 \del_3$ just like in case (a).}
\label{fig:deldel1}
\end{figure}
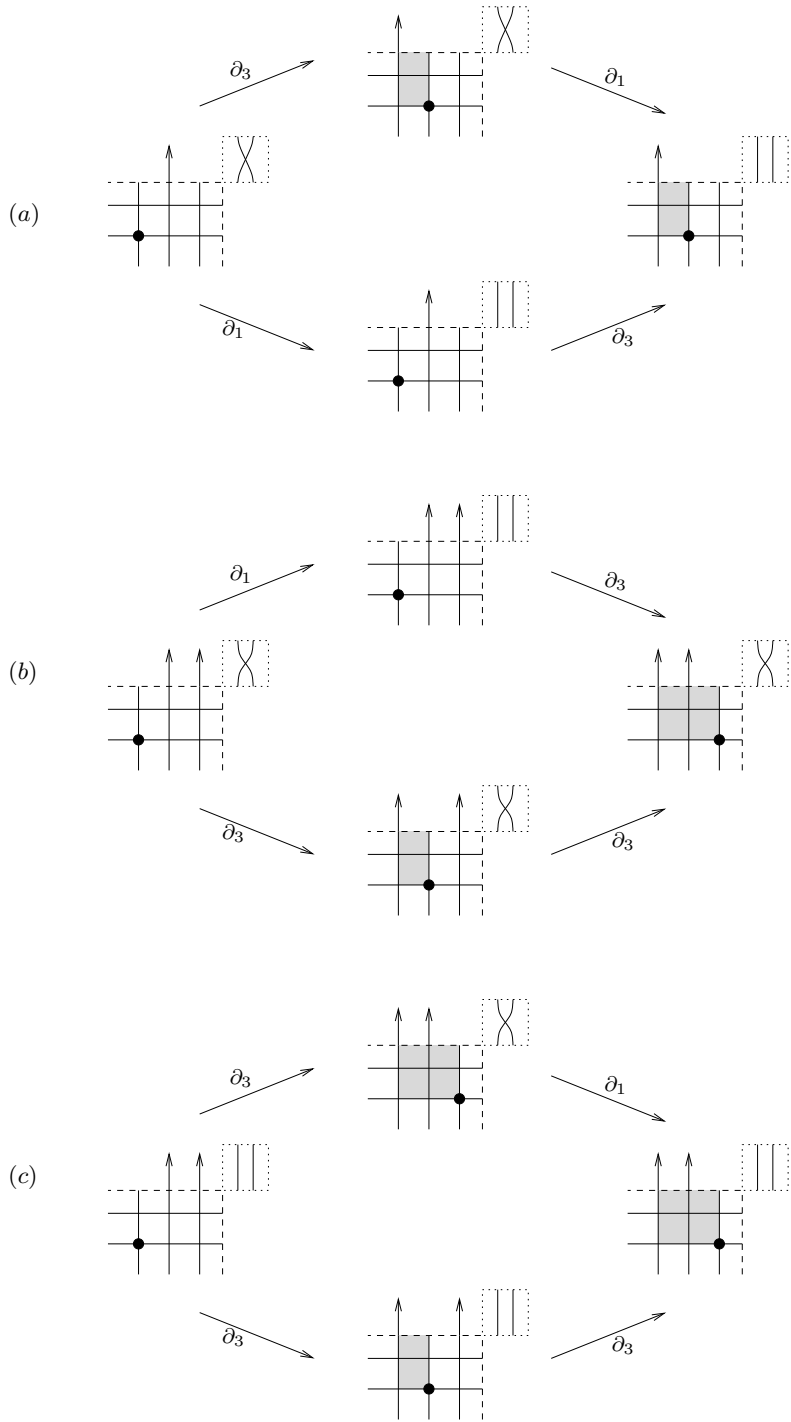

Next, all contributions to $\del_2 \del_3 (\xaa, M, \sigma)$ are of one of the types in Figure~\ref{fig:deldel2} (a)-(c), which also appear in $\del_3 \del_2 (\xaa, M, \sigma)$; or of the type in Figure~\ref{fig:deldel2} (d), which also appears in $\del_3^2 (\xaa, M, \sigma)$. In addition, $\del_3 \del_2 (\xaa, M, \sigma)$ contains terms as in Figure~\ref{fig:deldel2} (e), which cancel out certain terms from $\del_3^2(\xaa, M, \sigma)$.

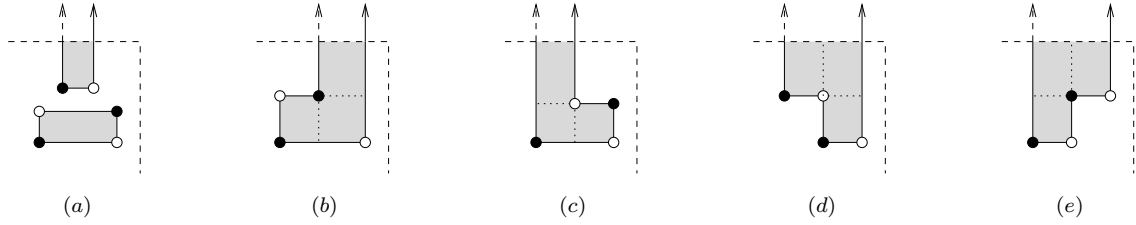
\begin{figure}[ht]
\begin{center}
\input{deldel2.pstex_t}
\end{center}
\caption {{\bf Cancellations between $\del_3 \del_2$, $\del_2 \del_3$, and $\del_3^2$.} The initial point $(\xaa, M, \sigma)$ is shown by black dots and the solid arrow. The final point, shown by white dots and the dashed arrow, appears as a term in $\del^2(\xaa, M, \sigma)$ in two different ways.}
\label{fig:deldel2}
\end{figure}

We are left with some other terms in $\del_3^2 (\xaa, M, \sigma)$, which cancel each other out in pairs. The possible cases are shown in Figure~\ref{fig:deldel3}. In cases (a) and (b), the final point appears as a term in $\del_3^2 (\xaa, M, \sigma)$ in two different ways, corresponding to two different orders of using the two rectangles in $\del_3$. In cases (c) and (d), applying $\del_3$ twice yields zero, regardless of the order in which we use the rectangles: either we encounter a rectangle that is not empty (i.e., it has a generator point in its interior), or else we get a double crossing in the nilCoxeter element, which gets set to zero.

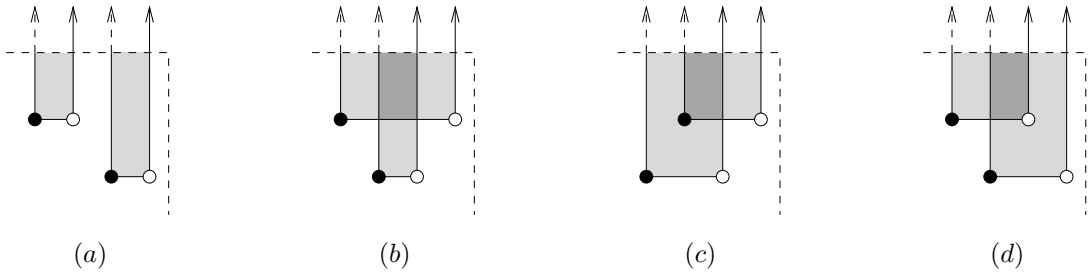
\begin{figure}[ht]
\begin{center}
\input{deldel3.pstex_t}
\end{center}
\caption {{\bf Cancellations between terms in $\del_3^2$.} The conventions are the same as in Figure~\ref{fig:deldel2}, and the darker shading signifies multiplicity two.  }
\label{fig:deldel3}
\end{figure}
\end {proof}

Let us now define the action of the top algebra-module $\T(k')$ on $\cpaa.$ Recall that $\overline \T(k')= \oplus_{m=0}^{k'} \T(k')_m$ is multiplicatively generated by the elements $I_S, \rho_{i, j}$ and $\mu_i,$ as in Section~\ref{sec:stam}.  The action of an idempotent $I_S$ is by
\begin {equation}
\label {eq:aat1}
 (\xaa, M , \sigma) *  I_S = \begin{cases}
(\xaa, M, \sigma) & \text{if } S = S_{\xaa}\\
0 & \text{otherwise.}
\end {cases} 
\end {equation} 

\nid The action of a chord $\rho_{i,j}$ is just as in $CPA^-$:
\begin {equation}
\label {eq:aat2}
(\xaa, M, \sigma) * \rho_{i, j} = \sum_{\yaa \in \S(\haa)} \sum_{H \in \Halfx(\xaa, \yaa; \rho_{i,j})} U(H) \ (\yaa, M, \sigma).
\end {equation}

\noindent Here $\Halfx(\xaa, \yaa; \rho_{i,j})$ is the space of empty half-strips with the bottom-left corner at a point of $\xaa,$ the top-left corner at a point of $\yaa$, and the right edge along $\rho_{i,j}$---the same kind of half-strips as those used in the definition of $CPA^-$ in Section~\ref{sec:cpad}. Observe also that the right-hand side of \eqref{eq:aat2} contains at most one term.

More interesting is the action of the half-chords $\mu_i.$ This is basically a count of ``quarter-strips'', remnants of rectangles in $\H$ that are cut by both interfaces $\ell$ and $\ell'.$ More precisely, a quarter-strip is a rectangle such that its bottom-left corner is a point $p_i \in \xaa$, its right edge is on $\ell$ starting at height $i,$ and its top edge is on $\ell'$ starting at the coordinate $\eps_{\xaa}(i) =j.$ If such a quarter-strip $Q$ contains no $X$ markings nor other points of $\xaa$ in its interior, we set $\Quartx(\xaa; \mu_i, \xi_j) = \{Q\},$ and let $U(Q)$ be the product of the $U$ variables corresponding to the $O$ markings inside $Q.$ For all other $\xaa, i$, and $j$, we let $\Quartx(\xaa; \mu_i, \xi_j) = \emptyset.$ The action of $\mu_i$ on $\cpaa$ is then given by
\begin {equation}
\label {eq:aat3}
(\xaa, M, \sigma) * \mu_i =  \sum_{Q \in \Quartx(\xaa; \mu_{i}, \xi_j)} U(Q) \ (\xaa \setminus \{p_i \}, M \cup \{ j \}, \sigma \cup_{M} j).
\end {equation}

Here, if we let $n$ be the number of elements of $M$ smaller than $j$, then for $\sigma\in \nil_m$ corresponding to a bijection $w: [m] \to [m],$ we let $\sigma \cup_M j \in \nil_{m+1}$ be the element corresponding to the bijection $\tilde w: [m+1] \to [m+1]$ given by
$$ \tilde w (l) = \begin {cases} 
w(l) & \text{if } l \leq n, \\
m+1 & \text{if } l=n+1,\\
w(l-1) & \text{if } l > n+1.
\end {cases}   $$

\nid This is easy to visualize in the graphical interpretation of $\nil$: to get $\sigma \cup_M j$ from $\sigma,$ we insert a new strand starting at the relative position of $j$ in $M$, and ending at the rightmost point on the top edge of the dotted square. See Figure~\ref{fig:cpaa2} for an example.

\begin{figure}[ht]
\begin{center}
\input{cpaa2.pstex_t}
\end{center}
\caption {{\bf The action of a half-chord on $\cpaa.$} Here, the half-chord $\mu_4$ acts on the element from Figure~\ref{fig:cpaa}. The shaded area is the corresponding quarter-strip.}
\label{fig:cpaa2}
\end{figure}
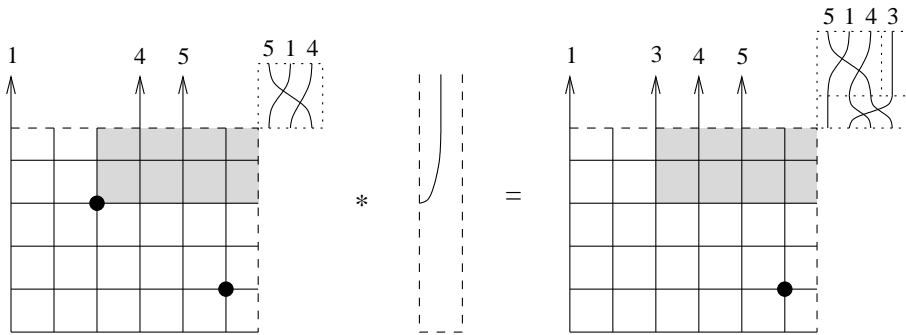

\begin {lemma}
\label {lemma:cpaaT}
The equations \eqref{eq:aat1}, \eqref{eq:aat2}, \eqref{eq:aat3} turn $\cpaa$ into a differential $ \T(k')$-module with respect to the horizontal $*$ multiplication; that is, they induce well-defined chain maps
$$ \cpaa_m \otimes \T(k')_{m'}  \longrightarrow \cpaa_{m+m'}$$
satisfying associativity relations of the form \eqref{eq:alc1'}.
\end {lemma}

\begin {proof}
The proof involves checking that the actions of the generators $I_S, \rho_{i,j}$ and $\mu_i$ are compatible with the relations \eqref{eq:rel1}-\eqref{eq:rel7} and \eqref{eq:rel1'}-\eqref{eq:rel5'}, and that the differential on $\cpaa$ is compatible with \eqref{eq:rel8} and \eqref{eq:rel6'}, that is, it satisfies the Leibniz rule.

Compatibility with the relations \eqref{eq:rel1}-\eqref{eq:rel7} involving only $I_S$ and $\rho_{i,j}$ is entirely similar to the compatibility of the action of the generators of $\alg(N, k)$ on $CPA^-,$ as in \cite[Proposition 7.1]{LOTplanar}.  In particular the action
of $\ol{\T}(k')$ is unital in the sense that $1 = \sum I_S $ does indeed function
as a unit; that is, for all $(\xaa, M, \sigma) \in \cpaa_m$, we have $(\xaa, M, \sigma) * 1 = (\xaa, M, \sigma)$.

Compatibility with the relations \eqref{eq:rel1'}-\eqref{eq:rel2.5'} is straightforward. 

Compatibility with \eqref{eq:rel3'}-\eqref{eq:rel5'} is illustrated in Figure~\ref{fig:cpaa_t}. In part (a), for example, we check compatibility with \eqref{eq:rel3'}. For $t \not \in M,$ we have
$$ ((\{x\}, M, \sigma) * \rho_{i,j}) * \mu_j = (\{y\}, M, \sigma) * \mu_j  = (\emptyset, M \cup \{t\}, \sigma \cup_M t )= (\{x\}, M, \sigma) * \mu_i. $$

 \begin{figure}[ht]
\begin{center}
\input{cpaa_t.pstex_t}
\end{center}
\caption {{\bf The action of the top algebra-module on $\cpaa$ is well-defined.} Part (a) corresponds to relation \eqref{eq:rel3'}, parts (b), (c), and (d) to \eqref{eq:rel4'}, and parts (e) and (f) to \eqref{eq:rel5'}.}
\label{fig:cpaa_t}
\end{figure}
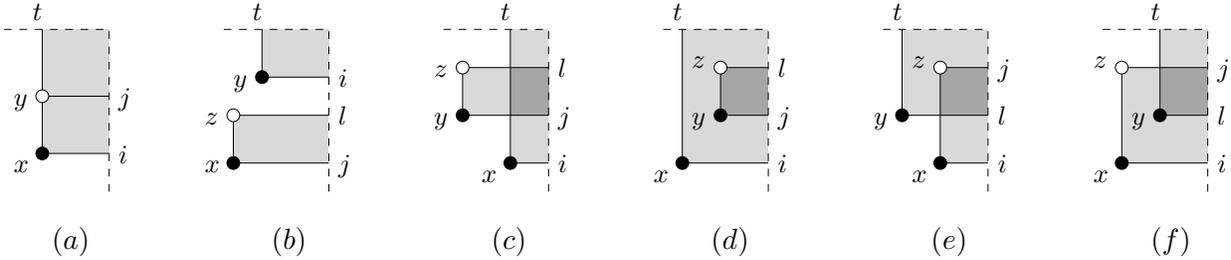

\nid To simplify notation, we have assumed that our generator $\xaa = \{x\}$ has only one element, but a similar string of equalities holds in general. Also, we have suppressed the $U$ powers; since the (shaded) domains involved are the same, these coincide.
 
Similarly, in parts (b) and (c) of Figure~\ref{fig:cpaa_t}, we have
$$ ((\{x,y\}, M, \sigma) *  \mu_i)  * \rho_{j,l}= (\{z\}, M \cup \{t\}, \sigma \cup_M t)= ((\{x,y \}, M, \sigma) * \rho_{j,l}) * \mu_i.$$

\nid In part (d), the actions of $\mu_i * \rho_{j,l}$
and $\rho_{j, l} * \mu_i$ are both zero.

In part (e),
$$ ((\{x,y \}, M, \sigma) * \rho_{i,j}) * \mu_l = (\{y,z\}, M, \sigma) * \mu_l = 0,$$
because the respective quarter-strip is not empty.  In part (f), already the action of $\rho_{i,j}$ is zero because the half strip in question is not empty.  
  
We now turn to verifying the Leibniz rule. When we look at multiplication by $I_S$ or $\rho_{i,j},$ the verification is entirely similar to that in \cite[Proposition 7.1]{LOTplanar}. It remains to check that the differential is compatible with multiplication by a half-chord, that is,
\begin {equation}
\label{eq:cpaa_leibniz}
\del((\xaa, M, \sigma) * \mu_i) + (\del(\xaa, M, \sigma)) * \mu_i + (\xaa, M, \sigma) * \del \mu_i = 0.
\end {equation} 

Recall that $\del \mu_i = \sum_{j > i} \mu_j * \rho_{i,j}$, and that the differential on
$\cpaa$ is $\partial = \partial_1 + \partial_2 + \partial_3$, where $\partial_1$ operates on the
nilCoxeter element $\sigma$, while $\partial_2$ and $\partial_3$ count rectangles and certain half-strips in $\haa$, respectively. Thus, the first two terms in \eqref{eq:cpaa_leibniz} decompose as triple sums. Each of the seven resulting terms in \eqref{eq:cpaa_leibniz} is a sum in itself; when we expand them out, the contributions cancel in pairs. We describe this cancellation, referring to
the cases illustrated in Figure~\ref{fig:cpaa_leibniz}.  We discuss particular cases where $\xaa$ is
rather small, but the general cases are entirely similar.

 \begin{figure}[ht]
\begin{center}
\input{cpaa_leibniz.pstex_t}
\end{center}
\caption {{\bf The Leibniz rule for the action of a half-chord.} }
\label{fig:cpaa_leibniz}
\end{figure}
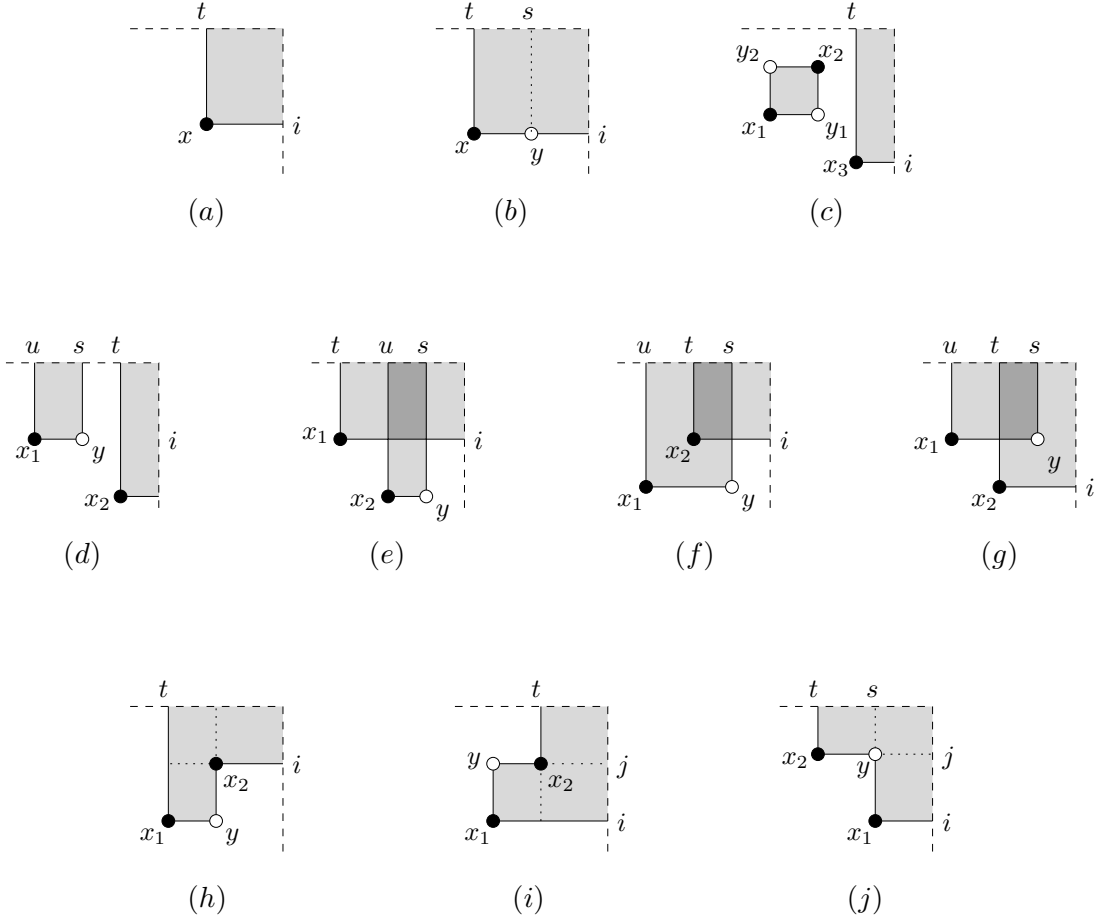

In part (a), supposing that $ \xaa=\{x\}$, the term
$$ (\emptyset, M \cup \{t\}, \del \sigma \cup_M t)  $$
appears in both $\del_1((\{x\}, M, \sigma) * \mu_i)$ and $(\del_1(\{x\}, M, \sigma)) * \mu_i.$ Note that, for general $\sigma, M, t,$ we have
$$ \del(\sigma \cup_M t) = (\del \sigma) \cup_M t +  \sum_{\{s \in M \cup \{t\} | s > t\}} \sigma^{M,t,s} \cup_{(M \cup \{t\} \setminus \{s\} )}  s.$$

\nid The last summation makes an appearance in part (b) of Figure~\ref{fig:cpaa_leibniz}, where, for $s \in M,$ the term
$$  (\emptyset, M \cup \{t\}, \sigma^{M,t,s} \cup_{(M \cup \{t\} \setminus \{s\} )}  s)$$
contributes both to $\del_1((\{x\}, M, \sigma) * \mu_i)$ and to 
$$(\del_3(\{x\}, M, \sigma)) * \mu_i =(\{y\}, M \cup \{t\} \setminus \{s\}, \sigma^{M, t, s}) * \mu_i.$$

In part (c), the term
$$ (\{y_1, y_2\}, M \cup \{t\}, \sigma \cup_M t)$$
appears in both $\del_2((\{x_1, x_2, x_3\}, M, \sigma) * \mu_i)$ and $(\del_2(\{x_1, x_2, x_3\}, M, \sigma)) * \mu_i.$

In part (d) and in part (e), supposing that $s \in M$, the term 
$$ (\{y\},  M \cup \{t,u\} \setminus \{s\}, (\sigma \cup_M t)^{M \cup \{t\}, u, s}) $$
appears in both $\del_3((\{x_1, x_2\}, M, \sigma) * \mu_i)$ and $(\del_3(\{x_1, x_2\}, M, \sigma)) * \mu_i.$

In part (f) and in part (g), for $s \in M$, the term 
$$ (\{y\},  M \cup \{t,u\} \setminus \{s\}, (\sigma \cup_M t)^{M \cup \{t\}, u, s}) $$
appears in $\del_3((\{x_1, x_2\}, M, \sigma) * \mu_i)$, but is actually zero because of double crossings in the nilCoxeter element. There is also no term of the form
$$ (\{y\},  M \cup \{t,u\} \setminus \{s\}, \sigma^{M, u, s} \cup_{M \cup \{u\} \setminus \{s\} } t) $$
in $(\del_3(\{x_1, x_2\}, M, \sigma)) * \mu_i$, because one of the respective rectangles is not empty. 

In part (h), the term
$$ (\{y\},  M \cup \{t\}, \sigma \cup_M t) $$
appears in both $(\del_2(\{x_1, x_2\}, M, \sigma)) * \mu_i$ and
$ \del_3( (\{x_1, x_2\}, M, \sigma) * \mu_i)$.

In part (i), the term
$$ (\{y\},  M \cup \{t\}, \sigma \cup_M t) $$
appears both in $(\del_2(\{x_1, x_2\}, M, \sigma)) * \mu_i$ and in 
$$(\{x_1, x_2\}, M, \sigma) * (\del \mu_i) = ((\{x_1, x_2\}, M, \sigma) * \mu_j) * \rho_{i,j}.$$

Finally, in part (j) we have the term
$$  (\{y\},  M \cup \{t\}, \sigma \cup_M t),$$
which contributes to $\del_3((\{x_1, x_2\}, M, \sigma) * \mu_i)$ and $(\{x_1, x_2\}, M, \sigma) * (\del \mu_i).$
 \end {proof}

The next step is to define an action of the right algebra-module $\R(k)$ from Section~\ref{sec:ram} on $\cpaa.$ Recall that elements of $\R(k)_{m,p}$ are required to act on $\cpaa_m$ and produce elements of $\cpaa_p$, and that $\ol{\R}(k) :=
\prod_{m,p} \R(k)_{m,p}$ has four kinds of multiplicative generators: idempotents,
chords, nilCoxeter elements, and caps.

We define the action of an idempotent $J_{S, m}$ on $(\xaa, M, \sigma) \in \cpaa_m$ to be
\begin {equation}
\label {eq:aar1}
 (\xaa, M , \sigma) \cdot  J_{S,m} =\begin{cases}
(\xaa, M, \sigma) & \text{if } S= M \cup \im(\eps_{\xaa}), \,\\
0 & \text{otherwise.}
\end {cases} 
\end {equation} 

The action of a chord $\lambda_{i,j}$ on a generator  $(\xaa, M, \sigma) \in \cpaa$ gives zero unless $i \in M \cup \im(\eps_{\xaa})$ and $\ j \not\in M \cup \im(\eps_{\xaa})$. If $ i \in \im(\eps_{\xaa})$, then the action is also zero if there exist elements of $M$ between $i$ and $j$; if there are no such elements, the action is similar to \eqref{eq:aat2}:
\begin {equation}
\label {eq:aar2}
(\xaa, M, \sigma) \cdot \lambda_{i, j} = \sum_{\yaa \in \S(\haa)} \sum_{H \in \Halfx(\xaa, \yaa; \lambda_{i,j})} U(H) \ (\yaa, M, \sigma),
\end {equation}
where $ \Halfx(\xaa, \yaa; \lambda_{i,j})$ is the set of empty half-strips with the top edge being the segment $\lambda_{i,j}$ on $\ell'.$ Here, an empty strip is required to contain no components of $\xaa$ in its interior, no $X$ markings in its interior, and also no arrows on its top edge. Note that the sum on the right-hand side of \eqref{eq:aar2} has at most one term.

If $i \in M$, we set
\begin {equation}
\label {eq:aar2.5}
(\xaa, M, \sigma) \cdot \lambda_{i, j} = (\xaa, M \cup \{j\} \setminus \{i\}, \sigma^{M, j, i}),
\end {equation}
where 
$$ \sigma^{M, j, i} =  (\sigma_{j'-1} \sigma_{j'-2} \dots \sigma_{i'+1}) \cdot \sigma$$
with $i' = \# \{u \in M | u < i \}$ and $j' =  \# \{u \in M | u < j \}.$ In the graphical representation, we obtain $\sigma^{M, j, i}$ from $\sigma$ by moving the starting point of the strand of $\sigma$ paired with the arrow at $i \in M$ a few positions to the right, so that it becomes paired with the arrow at $j$. (This is similar to Equation~\eqref{eq:smjl}, except here we have $i < j$.) See Figure~\ref{fig:cpaa25} for an example.

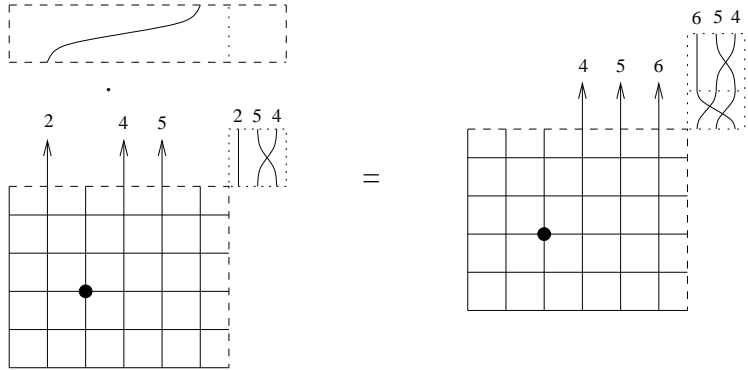
\begin{figure}[ht]
\begin{center}
\input{cpaa25.pstex_t}
\end{center}
\caption {{\bf Equation~\eqref{eq:aar2.5}.} The horizontal chord $\lambda_{2,6}$ acts on an element of $\cpaa$ by moving an arrow rightward. }
\label{fig:cpaa25}
\end{figure}

The action of a nilCoxeter element is simply:
\begin {equation}
\label {eq:aar3}
(\xaa, M, \sigma) \cdot \sigma_i = \begin {cases} (\xaa, M, \sigma \cdot \sigma_i) & \text{if } i < m, \\
0 & \text{if } i \geq m.
\end {cases}
\end {equation}

Finally, the action of a cap $\xi_i$ on $(\xaa, M, \sigma) \in \cpaa_m$ is as follows. The action results in zero unless $i \in M$ and the arrow situated at position $i \in M$ is paired with the top leftmost strand of the nilCoxeter element $\sigma$---in particular the action is zero if $m=0$. If the pairing condition is satisfied, we set
\begin {equation}
\label {eq:aar4}
(\xaa, M, \sigma) \cdot \xi_{i} = (\xaa, M \setminus \{ i \}, \sigma \setminus_M i),
\end {equation}
where $\sigma \setminus_M i$ is obtained from $\sigma$ by deleting the strand whose top end is first from the left. An example is shown in Figure~\ref{fig:cpaa3}.

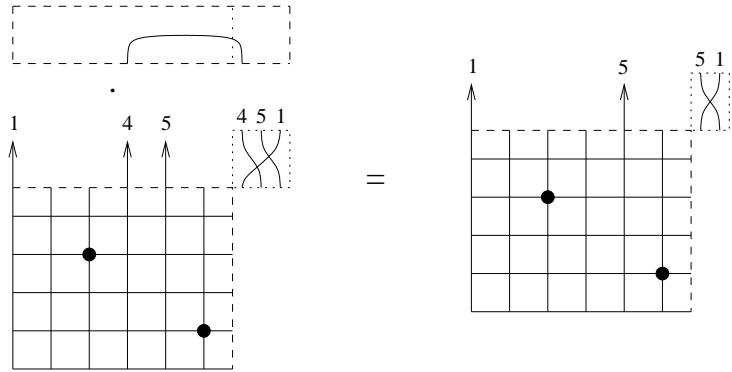
\begin{figure}[ht]
\begin{center}
\input{cpaa3.pstex_t}
\end{center}
\caption {{\bf The action of a cap on $\cpaa.$} Here, the cap $\xi_{4}$ acts by deleting the arrow and the strand labeled $4$.}
\label{fig:cpaa3}
\end{figure}

\begin {lemma}
\label {lemma:cpaaR}
The equations \eqref{eq:aar1}-\eqref{eq:aar4} turn $\cpaa$ into a unital differential $ \R(k)$-module with respect to the vertical $\cdot$ multiplication; that is, they induce well-defined chain maps
$$ \cpaa_m \otimes \R(k')_{m,n}  \longrightarrow \cpaa_{n}$$
satisfying associativity relations of the form \eqref{eq:alc2'}.
\end {lemma}

\begin {proof}
We need to check that the actions of the generators of $\R(k)$ are compatible with the relations described in Section~\ref{sec:ram}, and that the differential on $\cpaa$ is compatible with the differential on $\R(k)$.

Checking the formulas involving only the generators $J_{S,m}, \lambda_{i,j}$ and $\sigma_i$ is straightforward---many of the arguments run along the same lines as in \cite[Proposition 7.1]{LOTplanar} or Lemma~\ref{lemma:cpaaT} above. In particular, the action of
$\ol{\R}(k)$ is unital in the sense that $1_m := \sum_{S \subseteq [k]} J_{S,m}$ functions as a unit; that is, for all $(\xaa, M, \sigma) \in \cpaa_m$, we have $(\xaa, M, \sigma) \cdot 1_m = (\xaa, M, \sigma)$.

Compatibility with the relations \eqref{eq:rel5''}-\eqref{eq:rel7''}, involving only the idempotents and the caps, is also
clear. Compatibility with the relations \eqref{eq:rel8''}, \eqref{eq:rel9''} and \eqref{eq:rel10''} is
similar in spirit to the compatibility of the top algebra action with  \eqref{eq:rel3'}, \eqref{eq:rel4'}, and \eqref{eq:rel5'} respectively---see the proof of Lemma 7.4---except using caps
instead of half-chords.  Compatibility with \eqref{eq:rel11''} is illustrated in Figure~\ref{fig:cpaa2caps}. The action is compatible with \eqref{eq:rel12''} because acting first
by a cap removes the first nilCoxeter strand without otherwise disturbing the
nilCoxeter action.

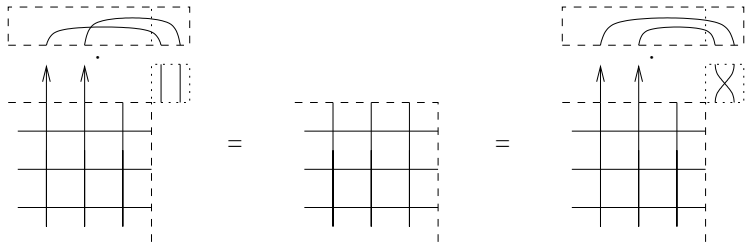
\begin{figure}[ht]
\begin{center}
\input{cpaa2caps.pstex_t}
\end{center}
\caption {{\bf Compatibility of the $\R(k)$-module action with the relation $\xi_i \cdot \xi_j = \sigma_1 \cdot \xi_j \cdot \xi_j$ for $i < j$.}}
\label{fig:cpaa2caps}
\end{figure}

To check compatibility of the differentials, it is enough to check the Leibniz rule for the action of each of the generators
$J_{S,m}, \lambda_{i,j}, \sigma_i$, and $\xi_i$.  The differentials of these
generators were given by equations \eqref{eq:rel15''}, \eqref{eq:rel16''}, \eqref{eq:rel17''}, and \eqref{eq:rel18''}, respectively.  

To check Leibniz for the action of $J_{S,m}$, note that for any $(\xaa,M,\sigma) \in
\cpaa$, all summands $(\yaa,N,\tau) \in \partial(\xaa,M,\sigma)$ have $N \cup
\im(\epsilon_{\yaa}) = M \cup \im(\epsilon_{\xaa})$; this implies that
$$\partial((\xaa,M,\sigma) \cdot J_{S,m}) = (\partial(\xaa,M,\sigma)) \cdot
J_{S,m}.$$
The Leibniz rule for the action of $\lambda_{i,j}$ follows
from a generalization of the argument used in \cite[Proposition 7.1]{LOTplanar}.
 The Leibniz rule for the action of $\sigma_i$ follows directly from the definitions. Finally, Leibniz for the cap action follows by a case by case analysis analogous to the
proof of Leibniz for half-chords in Lemma~\ref{lemma:cpaaT}.  Two of these cases are
illustrated in Figure~\ref{fig:cpaa_leibniz2}.
\end {proof}

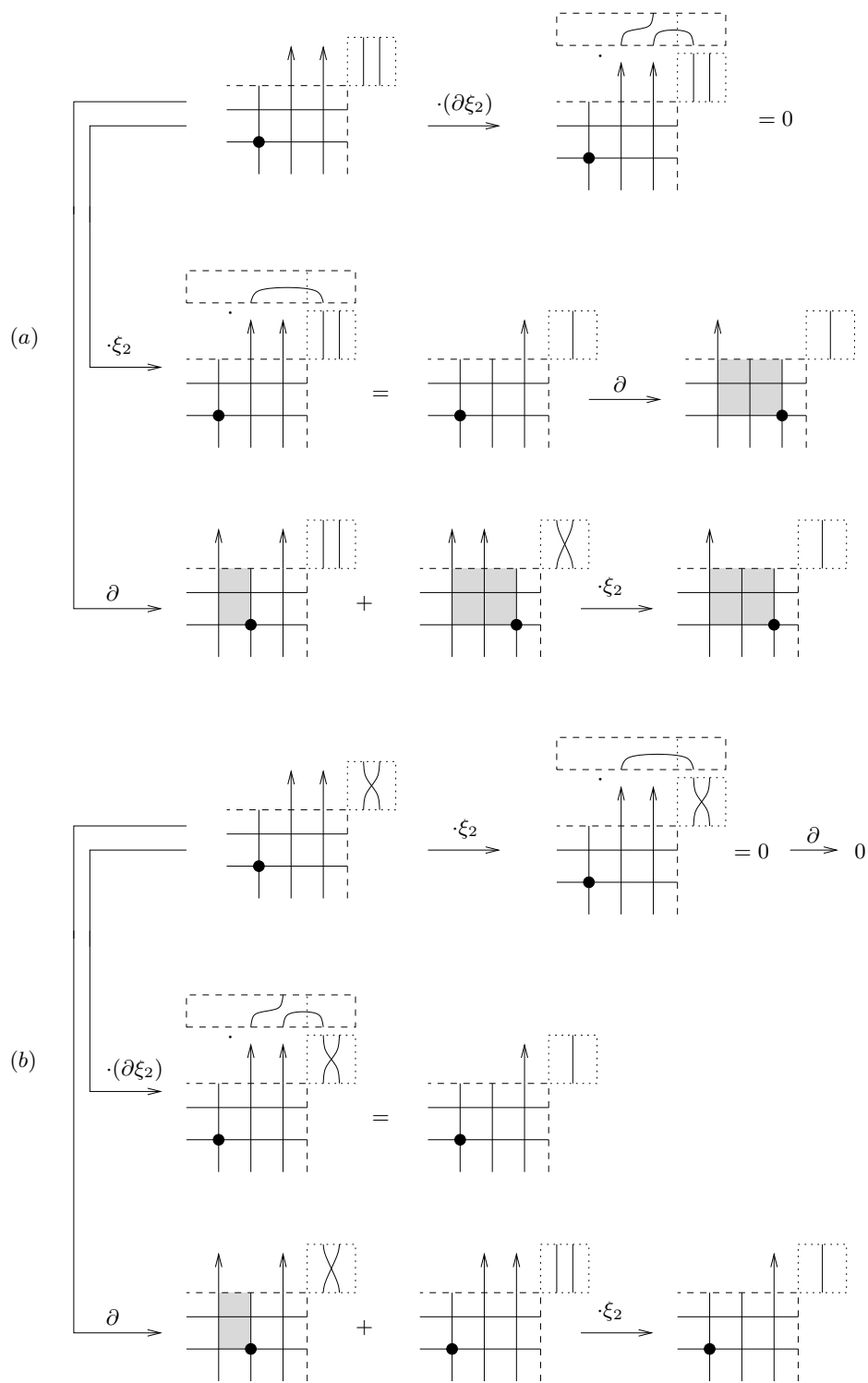
\begin{figure}[p]
\begin{center}
\input{cpaa_leibniz2.pstex_t}
\end{center}
\caption {{\bf The Leibniz rule with respect to multiplication by a cap $\xi_2$.} The final terms cancel each other out in pairs.}
\label{fig:cpaa_leibniz2}
\end{figure}

\begin {lemma}
\label{lemma:cpaaA}
For any $m, m', p$ and any $(\xaa, M, \sigma)\in \cpaa_m, t\in \T(k')_{m'}, r \in \R(k)_{m,p}$ and $a \in \nil_{m'},$ the following local commutation relation holds:
\begin {equation}
\label {eq:cpaa_lc}
 ((\xaa, M, \sigma) * t) \cdot (r * a) = ((\xaa, M, \sigma) \cdot r) * (t \cdot a).
 \end {equation}
\end {lemma}

\begin {proof}
Let us say that  Equation~\eqref{eq:cpaa_lc} is $(r, a)$-satisfied if it is satisfied for the given values of $(r,a)$ and for any $(\xaa, M, \sigma)$ and $t$ as in the hypotheses. Local commutativity for the right algebra-module $\R(k)$ implies that if \eqref{eq:cpaa_lc} is $(r, a)$-satisfied and $(r', a')$-satisfied, then \eqref{eq:cpaa_lc} is $(r \cdot r', a \cdot a')$-satisfied. Since $(r, a) = (r \cdot J_p, 1_{m'} \cdot a)$ (where $J_p$ denotes the sum of the idempotents $J_{S, p}$ over all possible $S$), it suffices to check that  \eqref{eq:cpaa_lc} is $(r, 1_{m'})$-satisfied and $(J_p, a)$-satisfied. Local commutativity with respect to $(J_p, a)$ follows from the structure of $\T(k')$ as a top algebra-module. 

To prove local commutativity with respect to $(r, 1_{m'})$, in light of the decomposition trick above, it suffices to prove it when $r$ is one of the following elements of $\R(k)$, which we refer to as  {\em local generators}: $$J_{S, m}, \ J_m \lambda_{i,j} J_m, \ J_m \sigma_i J_m \text{ or } J_m \xi_{i,j} J_{m-1}.$$ Note that the multiplicative generators of $\R(k)$ can be written as (infinite) sums of these
local generators. The local generators are better-suited for proving a relation that is only expected to hold when we fix the indices $m$ and $p$ in $\R(k)_{m,p}.$ We refer to the four kinds of generators above as an idempotent, a local chord, a local nilCoxeter element, and a local cap.

We can play a similar decomposition trick with respect to the horizontal multiplication, and reduce the problem to the case when $t$ is an idempotent, a chord $\rho_{i,j}$, or a half-chord $\mu_i$.   Note that these three types of elements are already local in the sense that they live in a given $\T(k')_{m'}$.  With $a$ being the identity and both $t$ and $r$ being local generators, the verification of \eqref{eq:cpaa_lc} is governed by a small number of cases. The cases involving an idempotent are clear, as are the cases
where the pair of elements $(r,t)$ are (nilCoxeter, chord), (cap, chord), and
(nilCoxeter, half-chord).  The remaining three cases are illustrated in Figures~\ref{fig:cpaa_lc1} through \ref{fig:cpaa_lc3b}. 
\end {proof}

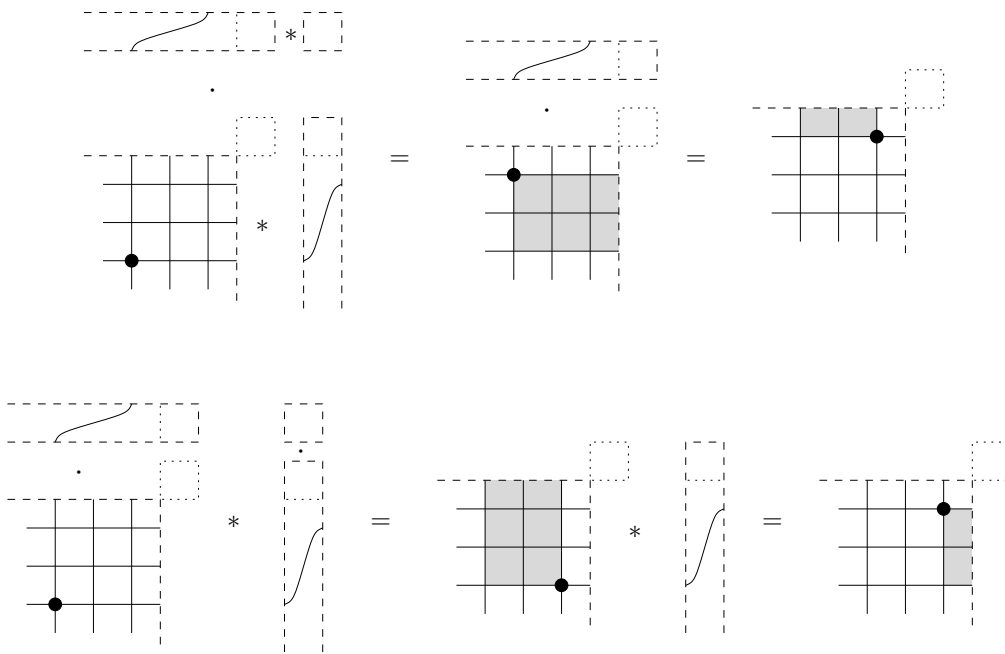
\begin{figure}[p]
\begin{center}
\input{cpaa_lc1.pstex_t}
\end{center}
\caption {{\bf Local commutation in $\cpaa$ for two chords.} The horizontal and vertical actions of local chords commute. }
\label{fig:cpaa_lc1}
\end{figure}

\begin{figure}[p]
\begin{center}
\vspace*{20pt}
\input{cpaa_lc2.pstex_t}
\end{center}
\caption {{\bf Local commutation in $\cpaa$ for a half-chord and a chord, case (1).} The horizontal action of a local half-chord and the vertical action of a local chord commute.}
\label{fig:cpaa_lc2}
\end{figure}
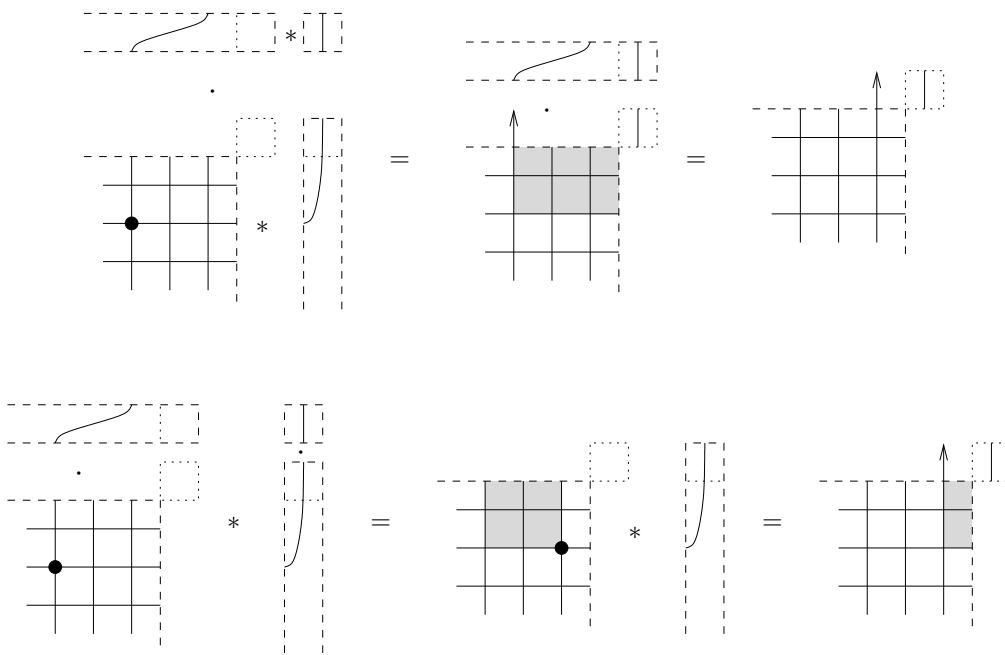

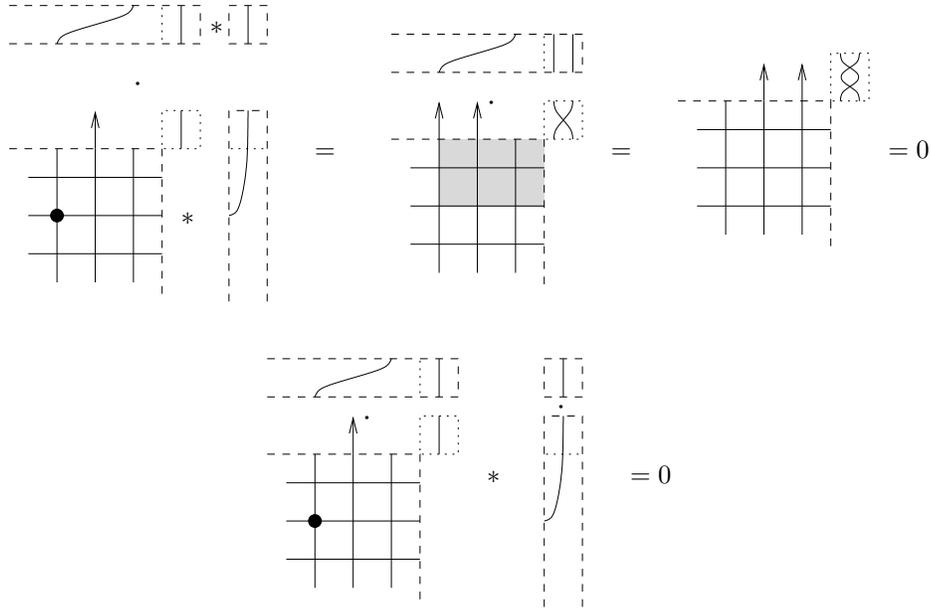
\begin{figure}[ht]
\begin{center}
\input{cpaa_lc2b.pstex_t}
\end{center}
\caption {{\bf Local commutation in $\cpaa$ for a half-chord and a chord, case (2).} The horizontal action of a local half-chord and the vertical action of a local chord commute.}
\label{fig:cpaa_lc2b}
\end{figure}

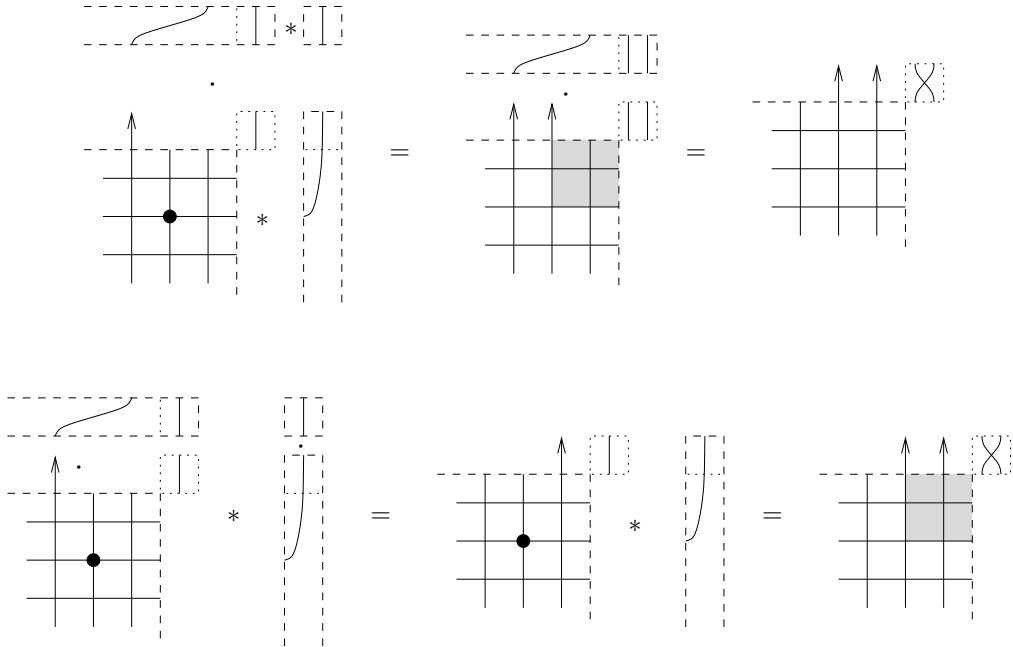
\begin{figure}[ht]
\begin{center}
\input{cpaa_lc2c.pstex_t}
\end{center}
\caption {{\bf Local commutation in $\cpaa$ for a half-chord and a chord, case (3).} The horizontal action of a local half-chord and the vertical action of a local chord commute.}
\label{fig:cpaa_lc2c}
\end{figure}

\begin{figure}[ht]
\begin{center}
\input{cpaa_lc3.pstex_t}
\end{center}
\caption {{\bf Local commutation in $\cpaa$ for a half-chord and a cap, case (1).} The horizontal action of a local half-chord and the vertical action of a local cap commute.}
\label{fig:cpaa_lc3}
\end{figure}
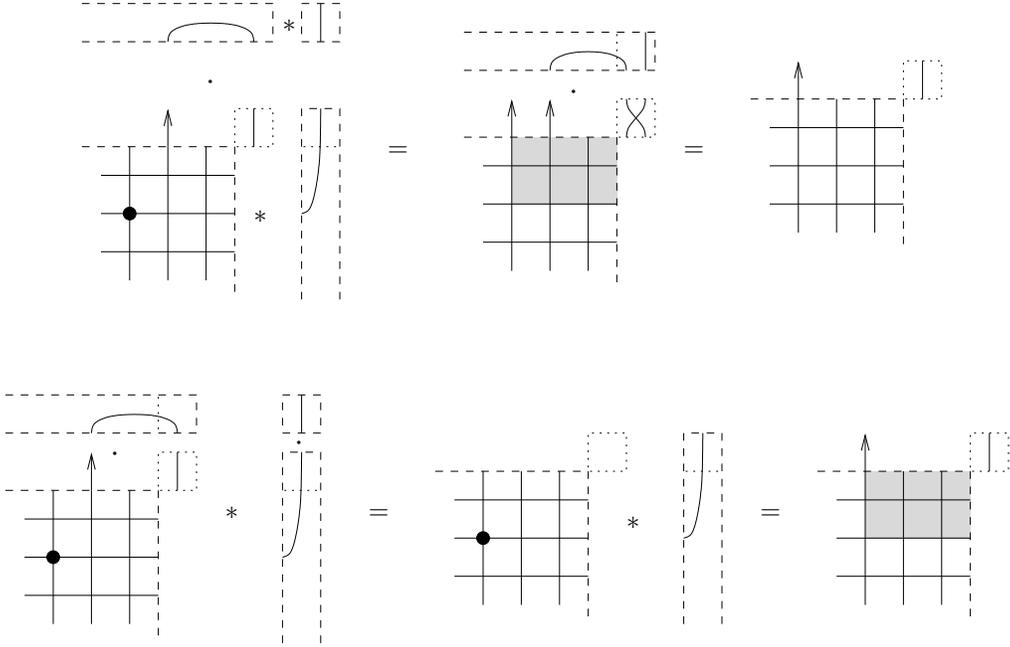

\begin{figure}[ht]
\begin{center}
\input{cpaa_lc3b.pstex_t}
\end{center}
\caption {{\bf Local commutation in $\cpaa$ for a half-chord and a cap, case (2).} The horizontal action of a local half-chord and the vertical action of a local cap commute.}
\label{fig:cpaa_lc3b}
\end{figure}
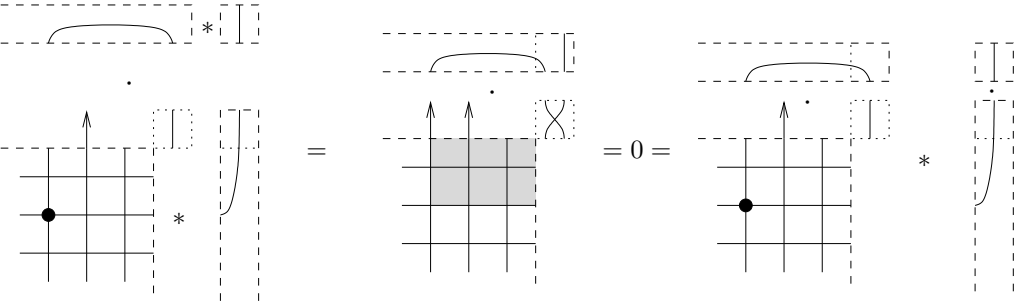

Lemmas~\ref{lemma:dcpaa}, \ref{lemma:cpaaT}, \ref{lemma:cpaaR}, \ref{lemma:cpaaA} together show that $\cpaa$ is indeed a top-right differential 2-module over $\T(k')$ and $\R(k).$  The grading on $\cpaa$ will be described in Section~\ref{sec:bigradeA}.

\subsection {The cornered 2-module $\cpad$}
\label {sec:cpad_module}
Our next goal is to define the bottom-right module 
$$\cpad = \{\cpad_p\}_{p \geq 0}$$ over $\R(k)$ and $\B(N-k').$

A partial planar generator $\xad$ in the quadrant $\had$ is defined to be a collection of points $\{p_i \in \alpha_i \cap \beta_{\eps(i)} | i \in S \}$, where $S= S_{\xad} \subseteq \{k'+1, \dots, N\}$ is a subset, and $\eps=\eps_{\xad} : S \to [k]$ is an injection. We denote by $\S(\had)$ the set of such partial planar generators, and by $\kk\langle \S(\had) \rangle$ the $\kk$-module freely generated by them. Let $\J(k)_0$ be the subalgebra of $\ol \R(k)$ generated by the idempotents $J_{S,0}$---see Section~\ref{sec:ram}. We let $\J(k)_0$ act on $\kk\langle \S(\had) \rangle$ by
$$ J_{S,0} \cdot \xad =  \begin{cases}
\xad & \text{if } \im(\eps_{\xad}) = [k] \setminus S,\\
0 & \text{otherwise.}
\end {cases} $$ 

We then let
\begin {equation}
\label {eq:pear}
 \cpad_p = \R(k)_{p,0} \odot_{\J( k)_0} \kk \langle \S(\had) \rangle.
 \end {equation}
Thus, a $\kk$-basis of $\cpad_p$ consists of the elements of the form $(S,  T, S', \emptyset, \phi, \psi, \psi') \cdot \xad,$ where the data in parentheses is a generator of $\R(k)_{p,0}$ as in Section~\ref{sec:ram}, and $\xad \in \S(\had)$ satisfies $\im(\eps_{\xad})= [k] \setminus T.$  Equation~\eqref{eq:pear} implicitly induces a vertical action:
$$ \cdot: \R(k)_{m, p} \otimes \cpad_p \to \cpad_m.$$

We define the differential on $\cpad$ in a manner similar to the definition of the differential on $\cpd$, except using half-strips with one edge on the interface $\ell'$ rather than $\ell$; see Figure~\ref{fig:cpad0}. 
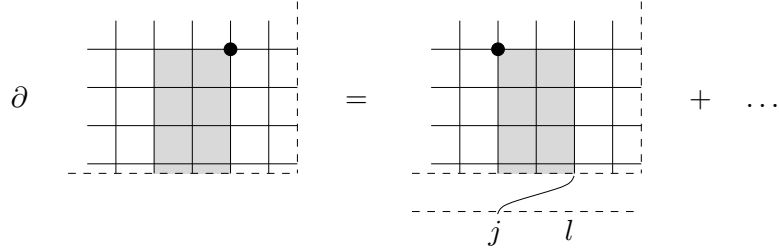
\begin{figure}[ht]
\begin{center}
\input{cpad0.pstex_t}
\end{center}
\caption {{\bf The relation $\del \xad = \lambda_{j,l} \cdot \yad + \dots$ in $\cpad.$} The shaded area is a half-strip.}
\label{fig:cpad0}
\end{figure}
\nid More precisely, suppose $\yad \in \S(\had)$ is obtained from $\xad$ by shifting one component $p_i$ from the column $\beta_l$ to the column $\beta_{j}$ (in the same row $\alpha_i$), such that $j < l.$ We then let $H$ be the half-strip bounded by $\beta_j, \alpha_i, \beta_l$ and $\ell'.$ If $H$ has no components of $\xad$ or $X$ markings inside, we set $\Halfx(\lambda_{j,l}; \xad, \yad) = \{H\},$ and define $U(H)$ as in Section~\ref{sec:cpad}. In all other cases, we set $\Halfx(\lambda_{j,l}; \xad, \yad) = \emptyset.$  The differential of a generator $\xad \in \S(\had)$ is then defined as the sum
\begin {eqnarray*}
\del \xad &=& \sum_{\yad \in \S(\had)} \sum_{R \in \Rectx(\xad, \yad)} U(R)\ \yad \\
& + & \sum_{\yad \in \S(\had)} \sum_{j,l} \sum_{H \in \Halfx(\lambda_{j,l}; \xad, \yad)} U(H) \ \lambda_{j,l} \cdot \yad. 
\end {eqnarray*}

\nid We extend $\del$ to the whole module $\cpad$ using the Leibniz rule with respect to vertical multiplication.

We also define a horizontal action 
$$* : \cpad_p \otimes \B(N-k')_m \to \cpad_{p+m},$$ 
as follows. We first define the action of the (horizontal) generators of $\B(N-k')$ on $\xad \in \S(\had).$  Let us make a minor notational change compared to Section~\ref{sec:stam}: we will add $k'$ to the height of the coordinates in $\B(N-k')$, so that what was $\nu_{j}$ in Section~\ref{sec:stam} becomes $\nu_{k'+j},$ etc. Then, idempotents and chords act just as in $\cpa$:
$$ \xad *  I_S = \begin{cases}
\xad & \text{if } S = S_{\xad} \\
0 & \text{otherwise,}
\end {cases}  $$
and
$$ \xad * \rho_{i, j} = \sum_{\yad \in \S(\had)} \sum_{H \in \Halfx(\xad, \yad; \rho_{i,j})} U(H) \ \yad.$$

The action of the half-chords involves quarter-strips, as follows. In this setting, a quarter-strip is a rectangle such that its top-left corner is a lattice point $(j,i) \in \had$, its right edge is on $\ell$ ending at height $i,$ and its bottom edge is on $\ell'$ starting at coordinate $j.$ If we have such a quarter-strip $Q$ with a corner at $(j,i)$ and a generator $\xad$ with $i \not \in S_{\xad}$ and $j \not \in \im(\eps_{\xad}),$ and if $Q$ contains no $X$ markings nor points of $\xad$ in its interior, then we set $\Quartx(\xi_j; \xad; \nu_i) = \{Q\},$ and let $U(Q)$ be the product of the $U$ variables corresponding to the $O$ markings inside $Q.$ For all other $\xad, i,$ and $j$, we set $\Quartx(\xi_j; \xad; \nu_i) = \emptyset.$ We then define
\begin {equation}
\xad * \nu_i = \sum_{j } \sum_{Q \in \Quartx(\xi_j; \xad ; \nu_{i})} U(Q) \ \xi_{j} \cdot (\xad \cup \{(j,i)\}).
\end {equation}

\nid See Figure~\ref{fig:cpad1} for an example.
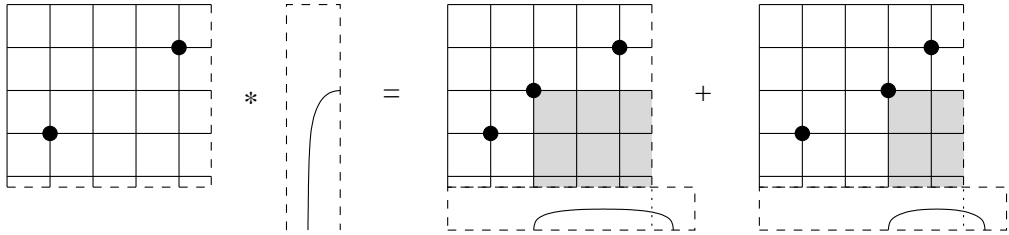
\begin{figure}[ht]
\begin{center}
\input{cpad1.pstex_t}
\end{center}
\caption {{\bf The action of a half-chord on $\cpad.$} Here, the half-chord $\nu_{k'+3}$ acts on an element $\xad.$ The shaded areas are the corresponding quarter-strips.}
\label{fig:cpad1}
\end{figure}

We extend this action of the generators $I_S,
\rho_{i,j}$, and $\nu_i$ on $\kk \langle \S(\had) \rangle$ to all of $\cpad$ by requiring local
commutativity of the $\R(k)$ and $\B(N-k')$ actions.  It remains to check that this
action is compatible with the relations defining the algebra-module $\B(N-k')$, and
that the $\B(N-k')$ action satisfies the Leibniz rule---we leave these verifications as exercises.  All told, the above actions give $\cpad$ the structure of a bottom-right 2-module over $\R(k)$ and $\B(k')$.

\subsection {Bigradings}
\label {sec:bigradeA}
Recall that in Section~\ref{sec:cpad} we equipped $CPA^-(\H^A)$ and $\alg(N)$ with suitable Alexander and Maslov gradings. Our goal here is to give similar gradings to the algebra-modules and 2-modules arising in the slicing of $\H^A$. 

For the top and bottom algebra-modules $\T(k')$ and $\B(N-k')$, we use the same formulas as for $\alg(N),$ that is, Equations \eqref{eq:alex} and \eqref{eq:maslov}:
\begin {eqnarray*}
 A(a) &=& L_X(a) - L_O(a) \\
  \mu(a) &=& \cro(a) - 2L_O(a),
  \end {eqnarray*}
where $a$ is a strand diagram which can have some ends along the central horizontal line $\ell'$. If we let the Alexander grading on the 2-algebra $\nil$ be identically zero, and let the Maslov grading be the crossing count $\cro$, we see that $\T(k')$ and $\B(N-k')$ are differential bigraded modules with respect to the vertical action of $\nil$. Further, the isomorphism
$$ \alg(N) \cong \T(k') \odot_{\nil} \B(N-k')$$
from Theorem~\ref{thm:bnt} preserves the bigradings.

On the right algebra-module $\R(k) = \{\R(k)_{m,p}\}$ we define gradings as follows. Let $L'_O, L'_X \subset \{3/2,\dots,N-1/2\}$ be the $x$-coordinates of the $X$'s and $O$'s under the horizontal line $\ell'$ in the original grid diagram $\H$. (In other words, if $\H = \haa \cup \had \cup \hda \cup \hdd$ as in the introduction, we look at the markings in $\haa \cup \hda$.) Let also $\Oaa$ and $\Xaa$ be the 
sets of $O$'s and $X$'s in the diagram $\haa.$ 

Let $a \in \R(k, s)_{m,p}$ be a basis element, viewed as a strand diagram with $s-(m-p)$ strands on the left, $p$ strands on the right, and $m-p$ caps. We let $L'_X(a)$ be the total intersection number of $a$ with the vertical lines $x=l$ for $l \in L'_X,$ and we define $L'_O(a)$ similarly. Then set:
\begin {eqnarray}
\label {eq:RA}
 A(a) &=& L'_X(a) - L'_O(a) - (m-p)(\# \Xaa - \# \Oaa)\\
 \label {eq:Rmu} \mu(a) &=& \cro(a) - 2L'_O(a) + 2(m-p)(\# \Oaa) - (m-p)(s-(m-p+1)/2).
  \end {eqnarray}
 
To put it differently, compared to what we would naively expect from Equations \eqref{eq:alex} and \eqref{eq:maslov}, here we adjust the quantities $L'_X$ and $L'_O$ by subtracting the number of all markings of that type in $\haa$, once for every cap in $a$. We also adjust the Maslov index by subtracting the quantity $v = s-(m-p+1)/2$, once for every cap in $a$. Because the number $m-p$ of caps is additive under vertical multiplication in $\R(k)$ (and does not change under the horizontal 2-algebra action), and the same is, perhaps unexpectedly, true for the quantity $(m-p)v$, it follows that the differential and the multiplications on $\R(k)$ behave as desired with respect to the two gradings.
   
Let us turn to the module $\cpaa$. Given a basis element $(\xaa, M, \sigma)$, we set
\begin {eqnarray*}
A(\xaa, M, \sigma) &=& \I(\Xaa, \xaa \cup M) - \I(\Oaa, \xaa \cup M) \\
\mu(\xaa, M, \sigma) &=& \I(\xaa, \xaa \cup M) + \cro(\sigma) - 2\I(\Oaa, \xaa \cup M), 
\end {eqnarray*}
where we view $M \subset \{1, \dots, k\}$ as a subset of the plane by fixing the arrow $i\in M$ to be at the point $(i, k'+3/4) \in \rr^2$, just as in the graphical representation (see for example Figure~\ref{fig:cpaa}). Recall that $\cpaa$ is a module over $\kk = \k[U_1, \dots, U_{N-1}]$ and the bigrading on this coefficient ring is $A(U_i) = -1$ and $\mu(U_i)
= -2$.

We can also define bigradings on $\cpad$, as follows. For $\xad \in \S(\had)$, let 
\begin {equation}
\label {eq:xa_d}
\x^{A} = \xad \cup \{(i, k'+3/4)\mid i \in [k] \setminus \im(\eps_{\xad}) \}.
\end {equation}

\nid Then define
\begin {eqnarray*}
A(\xad) &=& \I(\X^A, \xad) - \I(\O^A, \xad) \\
\mu(\xad) &=& \I(\x^A, \xad)  - 2\I(\O^A, \xad), 
\end {eqnarray*}
and extend these gradings to all of $\cpad$ by combining them with the ones on the right algebra-module $\R(k)$. Note that, if $\xad$ is the intersection of a generator $\x^A \in \S(\H^A)$ with $\had$, we could have alternatively defined $\mu(\xad)$ using this $\x^A$ instead of the one defined by \eqref{eq:xa_d}, and obtained the same answer. However, we use \eqref{eq:xa_d} in general because if $\xad$ has fewer than $N-k'$ components, it cannot be completed to an element of $\S(\H^A)$.

With these definitions in place, we have:

\begin {lemma}
\label {lem:big}
The differentials on the 2-modules $\cpaa$ and $\cpad$ preserve the Alexander gradings and decrease the Maslov gradings by one. Also, the actions of the right, top and bottom algebra-modules on $\cpaa$ and $\cpad$ are compatible with the bigradings.
\end {lemma}

\begin {proof}
This is similar to the proofs of Propositions 6.3 and 7.2 in \cite{LOTplanar} (about the corresponding statements for $CPA^-$ and $CPD^-$). In the case of $\cpaa$, we need to view the elements of $M$ as components of $\xaa$ moved to the line $y=k'+3/4$; this is why we use $\xaa \cup M$ in the definition of the bigrading. 

Compared to \cite{LOTplanar}, the new ingredient here is the action of the caps $\xi_{i} \in \R(k)$. On $\cpaa$, a cap acts by deleting an element of $M$ and a corresponding strand in $\sigma$. Suppose we have
\begin {equation}
\label {eq:aar4again}
(\xaa, M, \sigma) \cdot \xi_{i} = (\xaa, M \setminus \{ i \}, \sigma \setminus_M i), 
\end {equation}
as in Equation~\eqref{eq:aar4}. The relations
\begin {eqnarray*}
\I(\Xaa, \xaa \cup M) &=& \I(\Xaa, \xaa \cup (M \setminus \{i\})) +  \I(\Xaa,  \{(i, k'+3/4)\})  \\
&=&   \I(\Xaa, \xaa \cup (M \setminus \{i\})) + \# \Xaa - L_X'(\xi_{i})
\end {eqnarray*}
and 
\begin {eqnarray*}
\I(\Oaa, \xaa \cup M) &=& \I(\Oaa, \xaa \cup (M \setminus \{i\})) +  \I(\Oaa,  \{(i, k'+3/4)\})  \\
&=&   \I(\Oaa, \xaa \cup (M \setminus \{i\})) + \# \Oaa - L_O'(\xi_{i})
\end {eqnarray*}
imply that the two sides of \eqref{eq:aar4again} have the same Alexander grading. The proof that the Maslov grading is preserved is similar.  That the $\R(k)$ action on $\cpad$ is grading preserving is immediate.  That the actions of the top and bottom algebra-modules are also grading preserving is a straightforward, if tedious, computation.
\end {proof}

\subsection {The pairing theorem for $CPA^-$}
\begin {theorem}
\label {thm:cpa}
Let $\H^A = \haa \cup_{\ell'} \had$ be a partial planar grid diagram (of type $A$) of height $N$ and width $k$, horizontally sliced into $\haa$ and $\had,$ with the height of $\haa$ being $k'.$ There is an isomorphism of differential bigraded modules over $\alg(N):$ 
\begin {equation}
\label {eq:cpaha}
CPA^-(\H^A) \cong \bigl(\cpaah\bigr) \odot_{\R(k)} \big(\cpadh \bigr).
\end {equation}
\end {theorem}

\begin {proof}
A typical generator of the right-hand side of \eqref{eq:cpaha} is of the form
$$ (\xaa, M, \sigma) \odot (\phi \cdot \xad) = ((\xaa, M, \sigma) \cdot \phi) \odot \xad,$$
for some $\xaa \in S(\haa), \xad \in S(\had), \phi \in \R(k)_{p,0}.$
The generator vanishes unless $M$ has exactly $p$ elements. If $M$ has $p$ elements, applying $\phi$ to $(\xaa, M, \sigma)$ yields either zero or a new generator $(\yaa, \emptyset, e_0)$, which we simply call $\yaa$ (by a slight abuse of notation). In light of the actions of the idempotents, for the tensor product $\yaa \odot \xad$ not to vanish we need to have a disjoint union decomposition
$$ [k] = \im(\eps_{\yaa}) \coprod \im(\eps_{\xad}).$$

\nid If this holds, then we can identify $\yaa \odot \xad$ with the generator $\yaa \cup \xad $ of $CPA^-(\H^A).$ This identification preserves the Alexander and Maslov gradings.
This produces the isomorphism \eqref{eq:cpaha}, at least at the level of (bigraded) chain complexes. The proof that the differentials coincide is similar to the analogous result for the decomposition of $CP^-(\H)$ given in Theorem~\ref{thm:lot2}, but using the interface $\ell'$ rather than $\ell.$

We also need to check that the actions of $\alg(N) = \T(k') \odot_{\nil} \B(N-k')$ on the two sides of \eqref{eq:cpaha} coincide. This is clear for the actions of idempotents, and for the actions of chords which lie entirely either below $\ell'$ or above $\ell'.$ 

The interesting case is that of a chord $\rho_{i,j}$ which crosses the interface $\ell'.$ Suppose in $CPA^-$ we have a half-strip $$H \in \Halfx(\xaa \cup \xad, \yaa \cup \yad; \rho_{i,j}),$$ 
where $\yaa = \xaa \setminus \{(l,i)\}$ and $\yad = \xad \cup \{(l,j)\}$. Let $H^{AA}$ denote the quarter strip $H \cap \haa$, and $H^{AD}$ the quarter strip $H \cap \had$.

 We then have
\begin {equation}
\label {eq:recoverA}
(\xaa \cup \xad) * \rho_{i,j} =   (\xaa \odot \xad) * (\mu_i \odot \nu_j)\phantom{*******************}
\end {equation}
\begin {eqnarray*}
 &=& (\xaa * \mu_i) \odot (\xad * \nu_j) \\
&=& U(H^{AA})\ (\yaa, \{l\}, \id)  \odot \sum_p \sum_{Q \in \Quartx(\xi_p ; \xad ; \nu_j)}  U(Q) \ 
\xi_p \cdot (\xad \cup \{(p,j)\}) \\
&=& U(H^{AA}) U(H^{AD})\ (\yaa, \{l\}, \id) \odot (\xi_l \cdot \yad) \\
&=& U(H)\ ((\yaa,\{l\}, \id) \cdot \xi_l) \odot \yad \\
&=&  U(H) \ (\yaa \cup \yad),
\end {eqnarray*}
recovering the algebra action on $CPA^-.$ See Figure~\ref{fig:pairA}.\end {proof}

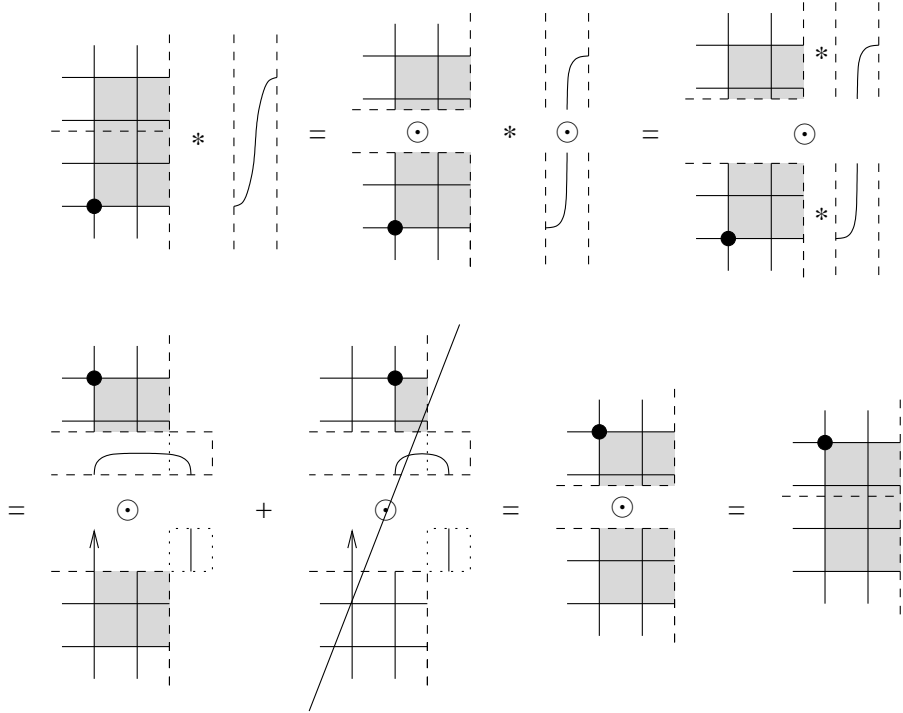
\begin{figure}[ht]
\begin{center}
\input{pairA.pstex_t}
\end{center}
\caption {{\bf Equation~\eqref{eq:recoverA}.} The action of $\alg(N)$ on $CPA^-$ can be recovered from the tensor product of $\cpaa$ and $\cpad.$ The crossed-out term is zero, because the arrow does not match the left end of the cap.}
\label{fig:pairA}
\end{figure}

\section {Reconstructing $CPD^-$ from cornered 2-modules} \label{sec:cpd}

We now turn attention to $\H^D$, a partial grid diagram of height $N$ and width $N-k$, bounded on the left by a vertical line $\ell,$ as in Section~\ref{sec:cpad}. We consider a horizontal slicing of $\H^D$ by the horizontal line $\ell'$ at height $k'+3/4$, just as we did for $\H^A$ in Section~\ref{sec:cpa}. We denote the quadrant below $\ell'$ (and right of $\ell$) by $\hda,$ and the quadrant above $\ell'$ by $\hdd.$ See the right-hand side of Figure~\ref{fig:dada}.

Recall that to $\H^D$ we can associate a differential graded module $CPD^- = CPD^-(\H^D)$ over the algebra $\alg(N)\cong \T(k') \odot_{\nil} \B(N-k').$ The goal of this section is to present a tensor product decomposition of the form
$$ CPD^- = \cpda \odot_{\L(N-k)} \cpdd,$$
where $\L(N-k)$ is a left algebra-module over $\nil,$ $\cpda = \cpdah$ is a top-left 2-module over $\L(N-k)$ and $\T(k')$, and $\cpdd  = \cpddh$ is a bottom-left 2-module over $\L(N-k)$ and $\B(N-k')$.

\subsection {The left algebra-module}
\label {sec:lam}

We define the left algebra-module 
$$\L(k) = \{ \L(k)_{m,p} | m, p \geq 0\}$$
by analogy with the right algebra-module $\R(k)$ from Section~\ref{sec:ram}, except now we will allow {\em cups} instead of caps.  

More precisely, each piece $\L(k)_{m,p}$ is nonzero only for $0 \leq p-m \leq k,$ in which case it has an additional decomposition
$$ \L(k)_{m,p} = \bigoplus_{t=p-m}^k \L(k, t)_{m,p}.$$

\nid As a $\kk$-module, the term $\L(k,t)_{m,p}$ is freely generated by data of the form $(S, T, T', M, \phi, \psi, \psi'),$ where:
\begin {itemize}
\item $S, T \subseteq [k] = \{1, \dots, k\}$ are subsets with $t-(p-m)$ elements each;
\item $T'$ is a subset of $[k]$ with $p-m$ elements such that $T \cap T' = \emptyset$; 
\item $\phi: S \to T$ is a bijection satisfying $\phi(i) \geq i$ for all $i\in S;$ 
\item $M$ is a subset of $[p]$ of cardinality $m;$ 
\item $\psi: [m] \to M$ and $\psi' : ([p] \setminus M) \to T'$ are bijections.
\end {itemize}

Graphically, we represent such a generator by a rectangular diagram split into two halves by a vertical dotted line. To the right of this line we have $k$ bullets at the bottom and $k$ at the top, and $t-(p-m)$ rightward-veering strands joining some of the bottom bullets (corresponding to the subset $S$) to some of the top ones (corresponding to the subset $T$); the strands link bullets according to the bijection $\phi$. To the left of the dotted line we have $m$ bullets at the bottom and $p$ at the top, with $m$ of the top ones (corresponding to $M$) joined with the bottom ones according to $\psi.$ Finally, $p-m$ of the top bullets to the left are joined with the remaining top bullets to the right (corresponding to $T'$), according to $\psi'$. For example,
$$\includegraphics[scale=0.8]{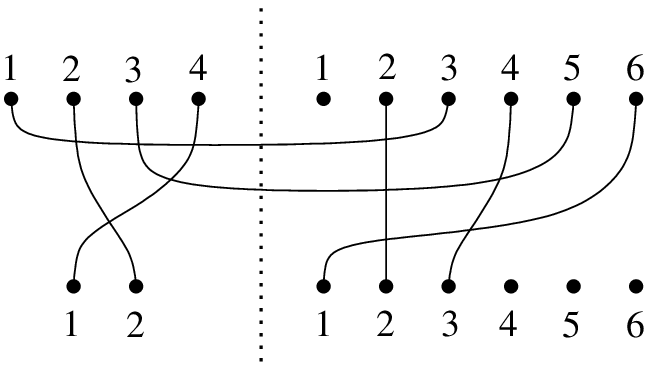}$$
represents a generator $(S, T, T', M, \phi, \psi, \psi')$, with $k=6, t=5, m=2, p=4, S=\{1,2,3\}, T=\{2,4,6\}, T' =\{3,5\},$ and $M = \{2,4\}.$

We define a differential on $\L(k)_{m,p}$ as the sum over all ways of smoothing a crossing. We also have two multiplications:
$$ \cdot: \L(k)_{m,m'} \otimes \L(k)_{m', p} \to \L(k)_{m,p} $$
and
$$  * : \nil_n \otimes \L(k)_{m,p} \to \L(k)_{m+n, p+n}, $$
by vertical and horizontal concatenation of diagrams, respectively, with double crossings being set to zero. Note that $\L(k)$ is multiplicatively finite, because $\L(k)_{m,p}=0$ for $m > p$.

With respect to the vertical multiplication of $\ol \L(k) = \prod_{m,p,t} \L(k,t)_{m,p}$, we have a set of  multiplicative generators consisting of: 
\begin {itemize}
\item idempotents $J_{S,m}=(S, S,\emptyset, [m], \id_S, \id_{[m]}, \emptyset)$ for $S \subseteq [k]$;
\item chords $\lambda_{i,j}$ from the $i\th$ bottom to the $j\th$ top bullet on the right-hand side; each $\lambda_{i,j}$ is an infinite sum of all septuples differing from a $J_{S, m}$ only in positions $i$ and $j$ on the right;
\item nilCoxeter elements $\sigma_i$; each $\sigma_i$ is an infinite sum of all elements $\sigma_{S, m, i}$ for $m > i$, where $\sigma_{S, m, i}$ is obtained from the idempotent $J_{S, m}$ by introducing a twist between the $i\th$ and the $(i+1)\th$ strands on the left-hand side of the dotted line;
\item cups of the form $\zeta_{i}$, drawn as strands joining the rightmost bullet on the top left to the $i\th$ bullet on the top right; that is, each cup $\zeta_{i}$ is the infinite sum of all $(S, S,\{i\}, [m], \id_S, \psi, \psi')$, where $\psi'$ is the unique bijection taking the $(m+1)\st$ bullet on the top left to $\{i\}$, and $\psi$ is the unique order-preserving bijection from $[m]$ to $[m] \subset [m+1]$.
\end {itemize}

\nid These generators satisfy relations similar to those for $\R(k)$ described in Section~\ref{sec:ram}.

From now on we will focus on the left algebra module $\L(N-k)$ corresponding to the right-hand side of the interface $\ell'$ in the planar grid diagram. We will therefore make a minor notational change (just as we did for $\B(N-k')$), in that we will add $k$ to each element of $S, T, T'$ when discussing generators of $\L(N-k).$ For example, instead of $\lambda_{3,4}$ we will write $\lambda_{k+3, k+4}$.

\subsection {The cornered 2-module $\cpda$}

Our next goal is to construct from the quadrant $\hda$ a top-left 2-module $\cpda = \cpdah$ over $\L(N-k)$ and $\T(k')$. The construction is reminiscent of that of $\cpad$ from Section~\ref{sec:cpad_module}. 

A partial planar generator $\xda$ in the quadrant $\hda$ is defined to be a collection of points $\{p_i \in \alpha_i \cap \beta_{\eps(i)} | i \in S \}$, where $S= S_{\xda} \subseteq \{1, \dots, k'\}$ is a subset, and $\eps=\eps_{\xda} : S \to \{k+1, \dots, N\}$ is an injection. We denote by $\S(\hda)$ the set of such partial planar generators, and by $\kk\langle \S(\hda) \rangle$ the $\kk$-module freely generated by them. Let $\I(k')$ be the subalgebra of $\ol \T(k')$ generated by the idempotents $I_{T}$. We let $\I(k')$ act on $\kk\langle \S(\hda) \rangle$ by
$$ I_{T} \cdot \xda =  \begin{cases}
\xda & \text{if } T = [k'] \setminus S_{\xda},\\
0 & \text{otherwise}.
\end {cases} $$ 

\nid We then set
\begin {equation}
\label {eq:apple}
 \cpda_m = \T(k')_m \oast_{\I(k')} \kk\langle \S(\hda) \rangle.
 \end {equation}

This equation makes transparent the horizontal multiplication action of $\T(k')$ on $\cpda.$ The differential of an element $\xda \in \cpda$ is defined by a count of empty rectangles and half-strips, exactly as in $CPD^-$:
\begin {eqnarray*}
\del \xda &=& \sum_{\yda \in \S(\hda)} \sum_{R \in \Rectx(\xda, \yda)} U(R)\ \yda \\
& + & \sum_{\yda \in \S(\hda)} \sum_{i,j} \sum_{H \in \Halfx(\rho_{i,j};\xda, \yda)} U(H) \ \rho_{i,j} * \yda. 
\end {eqnarray*}

\nid The differential is extended to all of $\cpda$ by imposing the Leibniz rule with respect to horizontal multiplication. 

Lastly, we define a vertical multiplication action of $\L(N-k)$ on $\cpda$, that is a collection of maps $\cpda_m \otimes \L(N-k)_{m,p} \ra \cpda_p$.  Let $\phi$ denote an element of $\T(k')_m$.  Idempotents, chords, and nilCoxeter elements act as expected:
\begin{align*}
(\phi * \xda) \cdot J_{S, m} &= \begin{cases} \phi * \xda & \text{if } S= \im(\eps_{\xda}) \\
0 & \text{otherwise} \end{cases}
\\ 
(\phi * \xda) \cdot \lambda_{i,j}  &= \sum_{\yda \in \S(\hda)} \sum_{H \in \Halfx(\xda, \yda; \lambda_{i,j})} U(H) \ \phi * \yda.
\\
(\phi * \xda) \cdot \sigma_i &= \begin{cases} (\phi \cdot \sigma_i) * \xda & \text{if } i < m \\ 0 & \text{otherwise} \end{cases}
\end{align*}

The action of cups involves counts of quarter-strips. More precisely, we have
$$ (\phi * \xda) \cdot \zeta_{j} = \sum_i \sum_{Q \in \Quartx(\mu_i; \xda; \zeta_j)} U(Q) \ (\phi * \mu_i) * (\xad \cup \{(j,i)\}),$$
where $\Quartx(\mu_i; \xda; \zeta_j)$ denotes the set (with at most one element) of empty quarter-strips with the left edge on $\ell$, the top edge on $\ell'$, and the bottom right corner at a point $(j,i)$ such that $i \not \in S_{\xda}$ and $j \not \in \im(\eps_{\xda}).$ An example is shown in Figure~\ref{fig:cpda1}.

\begin{figure}[ht]
\begin{center}
\input{cpda1.pstex_t}
\end{center}
\caption {{\bf The action of a cup on $\cpda.$} The cup is $\zeta_{3}.$ The shaded area is the respective quarter-strip.}
\label{fig:cpda1}
\end{figure}
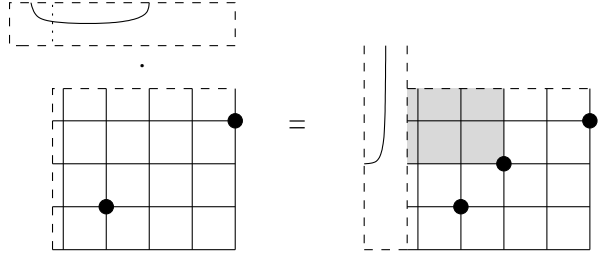

\subsection {The cornered 2-module $\cpdd$}
Next, we construct from the quadrant $\hdd$ a bottom-left 2-module $\cpdd = \cpddh$ over $\L(N-k)$ and $\B(N-k')$. 

We define a partial planar generator $\xdd$ in the quadrant $\hdd$ to be a collection of points $\{p_i \in \alpha_i \cap \beta_{\eps(i)} | i \in S \}$, where $S= S_{\xdd} \subseteq \{k'+1, \dots, N\}$ is a subset, and $\eps=\eps_{\xdd} : S \to \{k+1, \dots, N\}$ is an injection. We denote by $\S(\hdd)$ the set of such partial planar generators, and by $\kk\langle \S(\hdd) \rangle$ the $\kk$-module freely generated by them. We let the idempotent subalgebra $\I(N-k') \subset \B(N-k')$ act on $\kk\langle \S(\hdd) \rangle$ by
$$ I_{T} \cdot \xdd =  \begin{cases}
\xdd & \text{if } T = \{k'+1, \dots, N\}  \setminus S_{\xdd},\\
0 & \text{otherwise}.
\end {cases} $$ 

For any fixed $p \geq 0$, we define a subalgebra $\J(N-k)_p \subset \L(N-k)$ to be generated (via vertical multiplication) by the  idempotents $J_{S, p}$ and the elements $\sigma_i \cdot J_{S, p}$ for all possible $S$ and $i < p$. Thus, a $\kk$-basis of $\J(N-k)_p$ is given by elements of the form $\sigma \cdot J_{S,p}$, for $S \subseteq \{k+1, \dots, N \}$ and $\sigma \in \nil_p.$  

We let the subalgebra $\J(N-k)_p$ act on $\B(N-k')_p \oast_{\I(N-k')} \kk\langle \S(\hdd)\rangle$ by
$$ (\sigma' \cdot J_{S',p}) \cdot \bigl((\sigma \cdot (S, T, \phi)) * \xdd \bigr ) =  \begin{cases}
(\sigma' \cdot \sigma \cdot (S, T, \phi)) * \xdd & \text{if } S' = \{k+1, \dots, N\}  \setminus \im(\eps_{\xdd}),\\
0 & \text{otherwise}
\end {cases} $$ 

\nid We then set
\begin {equation}
\label {eq:quince}
 \cpdd_p = \bigoplus_{p' = p}^{N-k'} \L(N-k)_{p, p'} \odot_{\J(N-k)_{p'}} \bigl( \B(N-k')_{p'} \oast_{\I(N-k')} \kk\langle \S(\hdd)\rangle \bigr)
 \end {equation}

Thus, a $\kk$-basis for $\cpdd$ consists of elements of the form $$(S, T, T', M, \phi, \psi, \psi') \cdot \bigl(  (S'', T'', \phi'') * \xdd \bigr),$$ with $(S, T, T', M, \phi, \psi, \psi') \in \L(N-k)_{p',p}$ and $(S'', T'', \phi'') \in \B(N-k')_{p'}$, such that $T \cup T' =  \{k+1, \dots, N\}  \setminus \im(\eps_{\xdd})$ and $T'' = \{k'+1, \dots, N\}  \setminus S_{\xdd}.$  Diagrammatically, a basis element of $\cpdd$ is represented by a collection of points in $\hdd$, together with chord  pictures attached to the left and bottom of $\hdd$, as in Figure~\ref{fig:cpdd}. 

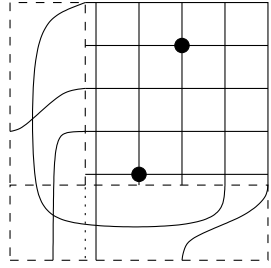
\begin{figure}[ht]
\begin{center}
\input{cpdd.pstex_t}
\end{center}
\caption {{\bf An element of $\cpdd.$} }
\label{fig:cpdd}
\end{figure}

By construction, $\cpdd$ comes equipped with bottom and left actions by  $\L(N-k)$ and $\B(N-k'),$ respectively, in such a way that local commutation is respected. It only remains to define the differential on $\cpdd$. The differential of some $\xdd \in \S(\hdd)$ involves a count of empty rectangles, empty half-strips (of two kinds: with the left edge on $\ell,$ or with the bottom edge on $\ell'$), as well as empty  quarter-strips as in Figure~\ref{fig:cpdd1}. More precisely,
\begin {eqnarray*} 
\del \xdd &=& \sum_{\ydd \in \S(\hdd)} \sum_{R \in \Rectx(\xdd, \ydd)} U(R)\ \ydd \\
& + & \sum_{\ydd \in \S(\hdd)} \sum_{i,j} \sum_{H \in \Halfx(\rho_{i,j};\xdd, \ydd)} U(H) \ \rho_{i,j} * \ydd\\
&+&  \sum_{\ydd \in \S(\hdd)} \sum_{i,j} \sum_{H \in \Halfx(\lambda_{i,j};\xdd, \ydd)} U(H) \ \lambda_{i,j} \cdot \ydd \\
& + &  \sum_{(i,j) \in \xdd} \sum_{Q \in \Quartx(\zeta_i, \nu_j;\xdd)} U(Q) \ \zeta_{i} \cdot (\nu_j * (\xdd \setminus \{(i,j)\} )).
\end {eqnarray*}

Here, $\Quartx(\zeta_i, \nu_j;\xdd)$ denotes the set of empty quarter-strips with the top right vertex at $(i, j) \in \xdd,$ the left edge on $\ell$, and the bottom edge on $\ell'.$ An example of the action of $\del$ is given in Figure~\ref{fig:cpdd1}. 

Once the differential is defined on generators $\xdd,$ we can uniquely extend it to the whole 2-module $\cpdd$ by requiring the Leibniz rules to be satisfied.

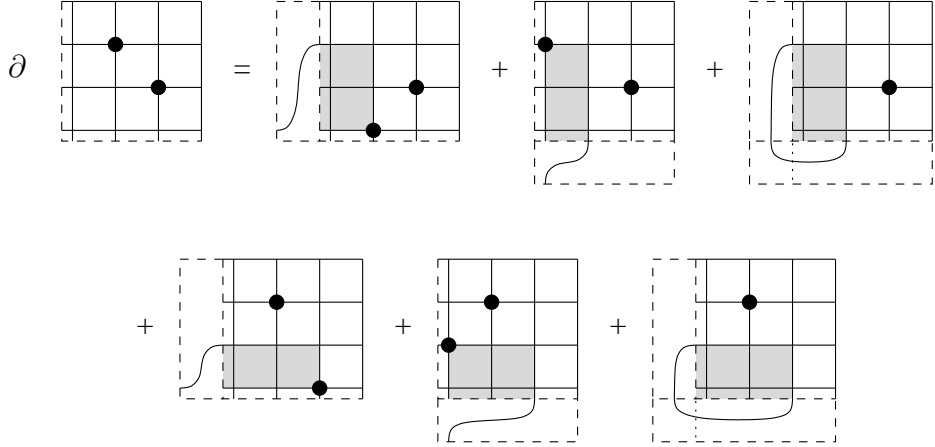
\begin{figure}[ht]
\begin{center}
\input{cpdd1.pstex_t}
\end{center}
\caption {{\bf The differential on $\cpdd.$} }
\label{fig:cpdd1}
\end{figure}

\subsection {Bigradings}
\label {sec:bigradeD}

We can equip the 2-modules $\cpda$ and $\cpdd$ with Alexander and Maslov gradings, just as we did for $\cpaa$ and $\cpad$ in Section~\ref{sec:bigradeA}. 

We look at the complete planar grid diagram $\H = \haa \cup \had \cup \hda \cup \hdd$, and use the notations $\Os, \Xs, \Oaa, \Xaa$ as in Section~\ref{sec:bigradeA}.  We first define a bigrading on the left algebra module $\L(N-k)$ by
\begin {eqnarray*}
 A(a) &=& L'_X(a) - L'_O(a) + (p-m)(\# \Xaa - \# \Oaa)\\
  \mu(a) &=& \cro(a) - 2L'_O(a) - 2(p-m)(\# \Oaa)  - (p-m)(k'-t+(p-m-1)/2).
  \end {eqnarray*}
for $a \in \L(N-k, t)_{m,p}$. 
These are analogous to equations \eqref{eq:RA} and \eqref{eq:Rmu}, but now $p-m$ is positive, and represents the number of cups in $a$.  Note that the grading on $\L(N-k)$ depends on $k'$, the height of the horizontal slicing of the planar grid.

To define the gradings on $\cpda$, we combine the ones on $\T(k')$ with those defined on generators $\xda \in \S(\hda)$ by
\begin {eqnarray*}
A(\xda) &=& \I(\Xs, \xda) - \I(\Os, \xda) \\
\mu(\xda) &=& \I(\ax, \xda)  - 2\I(\Os, \xda), 
\end {eqnarray*}
where $\ax = \xda \cup \{(k+3/4, i) \mid i \in [k']\setminus S_{\xda}\}$.  When $\xda$ is the intersection of a generator $\x \in \S(\H)$ with $\hda$, the Maslov grading of $\xda$ can be obtained using this $\x$ in place of $\ax$ in the above formula.

To define the gradings on $\cpdd$, we combine the ones on $\B(N-k')$ and $\L(N-k)$ with those defined on generators $\xdd \in \S(\hdd)$ by
\begin {eqnarray*}
A(\xdd) &=& \I(\Xs, \xdd) - \I(\Os, \xdd) \\
\mu(\xdd) &=& \I(\x^\urcorner, \xdd) + (\#\xdd + k + k' - N)(\#\xdd)  - 2\I(\Os, \xdd), 
\end {eqnarray*}
where 
$$\x^\urcorner = \xdd \cup \bigl\{(k+3/4, i) \mid i \in \{k'+1, \dots, N\} \setminus S_{\xdd} \bigr\}  \cup \bigl\{(i, k'+3/4)\mid i \in \{k+1, \dots, N\} \setminus \im(\eps_{\xdd}) \bigr \}.$$  When $\xdd$ is the intersection of a generator $\x \in \S(\H)$ with $\hdd$, the sum $\I(\x^\urcorner, \xdd) + (\#\xdd + k + k' - N)(\#\xdd)$ is equal to $\I(\x,\xdd)$---the first term accounts for $\I(\xad \cup \xda \cup \xdd, \xdd)$ and the second term accounts for $\I(\xaa,\xdd)$.  Note, though, that $\xdd$ need not admit any extension to a generator $\x$ and so in general, we must use the given formula for the Maslov grading.

The proof of the following is analogous to that of Lemma~\ref{lem:big}.

\begin {lemma}
The differentials on the 2-modules $\cpda$ and $\cpdd$ preserve the Alexander gradings and decrease the Maslov gradings by one. Also, the actions of the left, top and bottom algebra-modules on $\cpda$ and $\cpdd$ are compatible with the bigradings.
\end {lemma}

\subsection {The pairing theorem for $CPD^-$}
\begin {theorem}
\label {thm:cpd}
Let $\H^D = \hda \cup_{\ell'} \hdd$ be a partial planar grid diagram (of type $D$) of height $N$ and width $N-k$, horizontally sliced into $\hda$ and $\hdd,$ with the height of $\hda$ being $k'.$ There is an isomorphism of differential bigraded modules over $\alg(N):$ 
\begin {equation}
\label {eq:cpd_deco}
CPD^-(\H^D) \cong \bigl( \cpdah \bigr) \odot_{\L(N-k)} \bigl( \cpddh \bigr).
\end {equation}
\end {theorem}

\begin {proof}
The two sides
of ~\eqref{eq:cpd_deco} are certainly isomorphic as graded $\kk$-modules, and the obvious
isomorphism respects the action of $\alg(N) \cong \T(k') \odot_{\nil} \B(N-k')$ by horizontal multiplication on the left. It remains to  identify the differentials. The differential on $ CPD^-(\H^D)$ includes counts of empty rectangles and empty half-strips with one edge on $\ell.$ The rectangles and half-strips contained in either $\hda$ or $\hdd$ appear in one of the factors on the right-hand side. The rectangles that cross $\ell'$ are recovered from the tensor product on the right-hand side just as in the proof of Theorem~\ref{thm:lot2}.

The new contributions to consider are half-strips (with one edge on $\ell$) that cross the interface $\ell'.$ Suppose in $CPD^-$ we have such a half-strip $$H \in \Halfx(\rho_{i,j};\xda \odot \xdd, \yda \odot \ydd ).$$ We see that 
\begin {equation}
\label {eq:recoverD}
\del (\xda \odot \xdd)  =   \xda \odot (\del \xdd) + \dots = \xda \odot (\zeta_{l} \cdot (\nu_j * \ydd)) + \dots
\end {equation}
$$= (\xda \cdot \zeta_{l}) \odot (\nu_j * \ydd) + \dots  = (\mu_i * \yda) \odot (\nu_j * \ydd) + \dots $$
$$= (\mu_i \cdot \nu_j) * (\yda \odot \ydd) + \dots =  \rho_{i,j} * (\yda \odot \ydd) + \dots,$$
recovering the differential on $CPD^-.$ See Figure~\ref{fig:pairD}.\end {proof}

\begin{figure}[ht]
\begin{center}
\input{pairD.pstex_t}
\end{center}
\caption {{\bf Equation~\eqref{eq:recoverD}.} The differential on $CPD^-$ can be recovered from the tensor product of $\cpda$ and $\cpdd.$ }
\label{fig:pairD}
\end{figure}
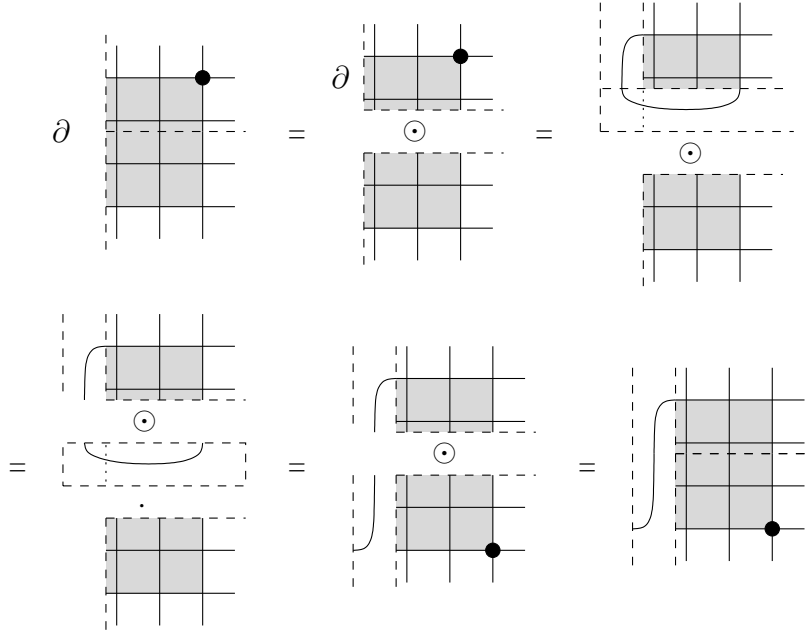

Theorems ~\ref{thm:cpa} and ~\ref{thm:cpd} together form the statement of Theorem~\ref{thm:main2}. As mentioned in the Introduction, by combining them with Theorem~\ref{thm:lot2}, we have the following:

\begin {corollary}
\label {cor:bitensor}
Let 
$$\H = (\haa \cup_{\ell'} \had) \cup_{\ell} (\hda \cup_{\ell'} \hdd)$$ be a planar grid diagram of size $N$, vertically and horizontally sliced into four quadrants. Then, there is an isomorphism of bigraded chain complexes over $\kk$:
$$CP^-(\H) \cong \odot \oast \bigr(\cpaah, \cpadh, \cpdah, \cpddh \bigl).$$
\end {corollary}

\bibliographystyle{custom}
\bibliography{biblio}

\end{document}

%% file: dada.pstex_t
\begin{picture}(0,0)%
\includegraphics{dada.pstex}%
\end{picture}%
\setlength{\unitlength}{1973sp}%
\begingroup\makeatletter\ifx\SetFigFont\undefined%
\gdef\SetFigFont#1#2#3#4#5{%
  \reset@font\fontsize{#1}{#2pt}%
  \fontfamily{#3}\fontseries{#4}\fontshape{#5}%
  \selectfont}%
\fi\endgroup%
\begin{picture}(4822,4764)(451,-4558)
\put(451,-2236){\makebox(0,0)[lb]{\smash{{\SetFigFont{10}{12.0}{\rmdefault}{\mddefault}{\updefault}{\color[rgb]{0,0,0}$\ell'$}%
}}}}
\put(1351,-1336){\makebox(0,0)[lb]{\smash{{\SetFigFont{8}{9.6}{\rmdefault}{\mddefault}{\updefault}{\color[rgb]{0,0,0}$X$}%
}}}}
\put(2551,-736){\makebox(0,0)[lb]{\smash{{\SetFigFont{8}{9.6}{\rmdefault}{\mddefault}{\updefault}{\color[rgb]{0,0,0}$X$}%
}}}}
\put(3151,-1936){\makebox(0,0)[lb]{\smash{{\SetFigFont{8}{9.6}{\rmdefault}{\mddefault}{\updefault}{\color[rgb]{0,0,0}$X$}%
}}}}
\put(4426,-2536){\makebox(0,0)[lb]{\smash{{\SetFigFont{8}{9.6}{\rmdefault}{\mddefault}{\updefault}{\color[rgb]{0,0,0}$X$}%
}}}}
\put(3751,-3736){\makebox(0,0)[lb]{\smash{{\SetFigFont{8}{9.6}{\rmdefault}{\mddefault}{\updefault}{\color[rgb]{0,0,0}$X$}%
}}}}
\put(1951,-1336){\makebox(0,0)[lb]{\smash{{\SetFigFont{8}{9.6}{\rmdefault}{\mddefault}{\updefault}{\color[rgb]{0,0,0}$O$}%
}}}}
\put(3151,-3136){\makebox(0,0)[lb]{\smash{{\SetFigFont{8}{9.6}{\rmdefault}{\mddefault}{\updefault}{\color[rgb]{0,0,0}$O$}%
}}}}
\put(4351,-736){\makebox(0,0)[lb]{\smash{{\SetFigFont{8}{9.6}{\rmdefault}{\mddefault}{\updefault}{\color[rgb]{0,0,0}$O$}%
}}}}
\put(1351,-1936){\makebox(0,0)[lb]{\smash{{\SetFigFont{8}{9.6}{\rmdefault}{\mddefault}{\updefault}{\color[rgb]{0,0,0}$O$}%
}}}}
\put(3751,-2536){\makebox(0,0)[lb]{\smash{{\SetFigFont{8}{9.6}{\rmdefault}{\mddefault}{\updefault}{\color[rgb]{0,0,0}$O$}%
}}}}
\put(2551,-3736){\makebox(0,0)[lb]{\smash{{\SetFigFont{8}{9.6}{\rmdefault}{\mddefault}{\updefault}{\color[rgb]{0,0,0}$O$}%
}}}}
\put(1951,-3136){\makebox(0,0)[lb]{\smash{{\SetFigFont{8}{9.6}{\rmdefault}{\mddefault}{\updefault}{\color[rgb]{0,0,0}$X$}%
}}}}
\put(3151, 14){\makebox(0,0)[lb]{\smash{{\SetFigFont{10}{12.0}{\rmdefault}{\mddefault}{\updefault}{\color[rgb]{0,0,0}$\ell$}%
}}}}
\put(5176,-4486){\makebox(0,0)[lb]{\smash{{\SetFigFont{10}{12.0}{\rmdefault}{\mddefault}{\updefault}{\color[rgb]{0,0,0}$\hda$}%
}}}}
\put(5101,-61){\makebox(0,0)[lb]{\smash{{\SetFigFont{10}{12.0}{\rmdefault}{\mddefault}{\updefault}{\color[rgb]{0,0,0}$\hdd$}%
}}}}
\put(526,-4486){\makebox(0,0)[lb]{\smash{{\SetFigFont{10}{12.0}{\rmdefault}{\mddefault}{\updefault}{\color[rgb]{0,0,0}$\haa$}%
}}}}
\put(601,-61){\makebox(0,0)[lb]{\smash{{\SetFigFont{10}{12.0}{\rmdefault}{\mddefault}{\updefault}{\color[rgb]{0,0,0}$\had$}%
}}}}
\end{picture}%

%% file: nilcox1.pstex_t
\begin{picture}(0,0)%
\includegraphics{nilcox1.pstex}%
\end{picture}%
\setlength{\unitlength}{2368sp}%
\begingroup\makeatletter\ifx\SetFigFont\undefined%
\gdef\SetFigFont#1#2#3#4#5{%
  \reset@font\fontsize{#1}{#2pt}%
  \fontfamily{#3}\fontseries{#4}\fontshape{#5}%
  \selectfont}%
\fi\endgroup%
\begin{picture}(3024,2695)(1189,-3044)
\put(3526,-2986){\makebox(0,0)[lb]{\smash{{\SetFigFont{10}{12.0}{\rmdefault}{\mddefault}{\updefault}{\color[rgb]{0,0,0}$4$}%
}}}}
\put(1726,-2986){\makebox(0,0)[lb]{\smash{{\SetFigFont{10}{12.0}{\rmdefault}{\mddefault}{\updefault}{\color[rgb]{0,0,0}$1$}%
}}}}
\put(2326,-2986){\makebox(0,0)[lb]{\smash{{\SetFigFont{10}{12.0}{\rmdefault}{\mddefault}{\updefault}{\color[rgb]{0,0,0}$2$}%
}}}}
\put(2926,-2986){\makebox(0,0)[lb]{\smash{{\SetFigFont{10}{12.0}{\rmdefault}{\mddefault}{\updefault}{\color[rgb]{0,0,0}$3$}%
}}}}
\end{picture}%

%% file: nilcoxrel.pstex_t
\begin{picture}(0,0)%
\includegraphics{nilcoxrel.pstex}%
\end{picture}%
\setlength{\unitlength}{1973sp}%
\begingroup\makeatletter\ifx\SetFigFont\undefined%
\gdef\SetFigFont#1#2#3#4#5{%
  \reset@font\fontsize{#1}{#2pt}%
  \fontfamily{#3}\fontseries{#4}\fontshape{#5}%
  \selectfont}%
\fi\endgroup%
\begin{picture}(9024,3661)(1489,-3410)
\put(8926,-436){\makebox(0,0)[lb]{\smash{{\SetFigFont{10}{12.0}{\rmdefault}{\mddefault}{\updefault}{\color[rgb]{0,0,0}$\cdots$}%
}}}}
\put(5326,-436){\makebox(0,0)[lb]{\smash{{\SetFigFont{10}{12.0}{\rmdefault}{\mddefault}{\updefault}{\color[rgb]{0,0,0}$\cdots$}%
}}}}
\put(7426,-2536){\makebox(0,0)[lb]{\smash{{\SetFigFont{9}{10.8}{\rmdefault}{\mddefault}{\updefault}{\color[rgb]{0,0,0}$\del$}%
}}}}
\end{picture}%

%% file: nilcox2.pstex_t
\begin{picture}(0,0)%
\includegraphics{nilcox2.pstex}%
\end{picture}%
\setlength{\unitlength}{1973sp}%
\begingroup\makeatletter\ifx\SetFigFont\undefined%
\gdef\SetFigFont#1#2#3#4#5{%
  \reset@font\fontsize{#1}{#2pt}%
  \fontfamily{#3}\fontseries{#4}\fontshape{#5}%
  \selectfont}%
\fi\endgroup%
\begin{picture}(5787,5778)(-2774,-2773)
\put(-2774,-1561){\makebox(0,0)[lb]{\smash{{\SetFigFont{9}{10.8}{\rmdefault}{\mddefault}{\updefault}{\color[rgb]{0,0,0}$\alpha_6$}%
}}}}
\put(-2774,1439){\makebox(0,0)[lb]{\smash{{\SetFigFont{9}{10.8}{\rmdefault}{\mddefault}{\updefault}{\color[rgb]{0,0,0}$\alpha_1$}%
}}}}
\put(-2774,839){\makebox(0,0)[lb]{\smash{{\SetFigFont{9}{10.8}{\rmdefault}{\mddefault}{\updefault}{\color[rgb]{0,0,0}$\alpha_2$}%
}}}}
\put(-2774,239){\makebox(0,0)[lb]{\smash{{\SetFigFont{9}{10.8}{\rmdefault}{\mddefault}{\updefault}{\color[rgb]{0,0,0}$\alpha_3$}%
}}}}
\put(-2774,-361){\makebox(0,0)[lb]{\smash{{\SetFigFont{9}{10.8}{\rmdefault}{\mddefault}{\updefault}{\color[rgb]{0,0,0}$\alpha_4$}%
}}}}
\put(-2774,-961){\makebox(0,0)[lb]{\smash{{\SetFigFont{9}{10.8}{\rmdefault}{\mddefault}{\updefault}{\color[rgb]{0,0,0}$\alpha_5$}%
}}}}
\put(-1274,2789){\makebox(0,0)[lb]{\smash{{\SetFigFont{9}{10.8}{\rmdefault}{\mddefault}{\updefault}{\color[rgb]{0,0,0}$\beta_1$}%
}}}}
\put(1726,2789){\makebox(0,0)[lb]{\smash{{\SetFigFont{9}{10.8}{\rmdefault}{\mddefault}{\updefault}{\color[rgb]{0,0,0}$\beta_6$}%
}}}}
\put(-674,2789){\makebox(0,0)[lb]{\smash{{\SetFigFont{9}{10.8}{\rmdefault}{\mddefault}{\updefault}{\color[rgb]{0,0,0}$\beta_2$}%
}}}}
\put(-74,2789){\makebox(0,0)[lb]{\smash{{\SetFigFont{9}{10.8}{\rmdefault}{\mddefault}{\updefault}{\color[rgb]{0,0,0}$\beta_3$}%
}}}}
\put(526,2789){\makebox(0,0)[lb]{\smash{{\SetFigFont{9}{10.8}{\rmdefault}{\mddefault}{\updefault}{\color[rgb]{0,0,0}$\beta_4$}%
}}}}
\put(1126,2789){\makebox(0,0)[lb]{\smash{{\SetFigFont{9}{10.8}{\rmdefault}{\mddefault}{\updefault}{\color[rgb]{0,0,0}$\beta_5$}%
}}}}
\end{picture}%

%% file: nilcox3.pstex_t
\begin{picture}(0,0)%
\includegraphics{nilcox3.pstex}%
\end{picture}%
\setlength{\unitlength}{2368sp}%
\begingroup\makeatletter\ifx\SetFigFont\undefined%
\gdef\SetFigFont#1#2#3#4#5{%
  \reset@font\fontsize{#1}{#2pt}%
  \fontfamily{#3}\fontseries{#4}\fontshape{#5}%
  \selectfont}%
\fi\endgroup%
\begin{picture}(5112,4755)(6076,-5698)
\put(6376,-4861){\makebox(0,0)[lb]{\smash{{\SetFigFont{8}{9.6}{\rmdefault}{\mddefault}{\updefault}{\color[rgb]{0,0,0}$\alpha_3$}%
}}}}
\put(7351,-1111){\makebox(0,0)[lb]{\smash{{\SetFigFont{8}{9.6}{\rmdefault}{\mddefault}{\updefault}{\color[rgb]{0,0,0}$\beta_1$}%
}}}}
\put(7951,-1111){\makebox(0,0)[lb]{\smash{{\SetFigFont{8}{9.6}{\rmdefault}{\mddefault}{\updefault}{\color[rgb]{0,0,0}$\beta_2$}%
}}}}
\put(8551,-1111){\makebox(0,0)[lb]{\smash{{\SetFigFont{8}{9.6}{\rmdefault}{\mddefault}{\updefault}{\color[rgb]{0,0,0}$\beta_3$}%
}}}}
\put(6076,-2086){\makebox(0,0)[lb]{\smash{{\SetFigFont{8}{9.6}{\rmdefault}{\mddefault}{\updefault}{\color[rgb]{0,0,0}$\gamma_1$}%
}}}}
\put(6376,-1786){\makebox(0,0)[lb]{\smash{{\SetFigFont{8}{9.6}{\rmdefault}{\mddefault}{\updefault}{\color[rgb]{0,0,0}$\gamma_2$}%
}}}}
\put(6676,-1486){\makebox(0,0)[lb]{\smash{{\SetFigFont{8}{9.6}{\rmdefault}{\mddefault}{\updefault}{\color[rgb]{0,0,0}$\gamma_3$}%
}}}}
\put(6376,-3661){\makebox(0,0)[lb]{\smash{{\SetFigFont{8}{9.6}{\rmdefault}{\mddefault}{\updefault}{\color[rgb]{0,0,0}$\alpha_1$}%
}}}}
\put(6376,-4261){\makebox(0,0)[lb]{\smash{{\SetFigFont{8}{9.6}{\rmdefault}{\mddefault}{\updefault}{\color[rgb]{0,0,0}$\alpha_2$}%
}}}}
\end{picture}%

%% file: modseq.pstex_t
\begin{picture}(0,0)%
\includegraphics{modseq.pstex}%
\end{picture}%
\setlength{\unitlength}{2368sp}%
\begingroup\makeatletter\ifx\SetFigFont\undefined%
\gdef\SetFigFont#1#2#3#4#5{%
  \reset@font\fontsize{#1}{#2pt}%
  \fontfamily{#3}\fontseries{#4}\fontshape{#5}%
  \selectfont}%
\fi\endgroup%
\begin{picture}(4824,4824)(589,-4273)
\put(4501,-286){\makebox(0,0)[lb]{\smash{{\SetFigFont{10}{12.0}{\rmdefault}{\mddefault}{\updefault}{\color[rgb]{0,0,0}$BL$}%
}}}}
\put(1276,-1336){\makebox(0,0)[lb]{\smash{{\SetFigFont{10}{12.0}{\rmdefault}{\mddefault}{\updefault}{\color[rgb]{0,0,0}$p$}%
}}}}
\put(4576,-2536){\makebox(0,0)[lb]{\smash{{\SetFigFont{10}{12.0}{\rmdefault}{\mddefault}{\updefault}{\color[rgb]{0,0,0}$m$}%
}}}}
\put(1276,-2536){\makebox(0,0)[lb]{\smash{{\SetFigFont{10}{12.0}{\rmdefault}{\mddefault}{\updefault}{\color[rgb]{0,0,0}$m$}%
}}}}
\put(2926,-3586){\makebox(0,0)[lb]{\smash{{\SetFigFont{10}{12.0}{\rmdefault}{\mddefault}{\updefault}{\color[rgb]{0,0,0}$\T$}%
}}}}
\put(2926,-2536){\makebox(0,0)[lb]{\smash{{\SetFigFont{10}{12.0}{\rmdefault}{\mddefault}{\updefault}{\color[rgb]{0,0,0}$m$}%
}}}}
\put(2926,-1936){\makebox(0,0)[lb]{\smash{{\SetFigFont{10}{12.0}{\rmdefault}{\mddefault}{\updefault}{\color[rgb]{0,0,0}$\A$}%
}}}}
\put(2926,-1336){\makebox(0,0)[lb]{\smash{{\SetFigFont{10}{12.0}{\rmdefault}{\mddefault}{\updefault}{\color[rgb]{0,0,0}$p$}%
}}}}
\put(2926,-286){\makebox(0,0)[lb]{\smash{{\SetFigFont{10}{12.0}{\rmdefault}{\mddefault}{\updefault}{\color[rgb]{0,0,0}$\B$}%
}}}}
\put(4576,-1936){\makebox(0,0)[lb]{\smash{{\SetFigFont{10}{12.0}{\rmdefault}{\mddefault}{\updefault}{\color[rgb]{0,0,0}$\L$}%
}}}}
\put(4576,-1336){\makebox(0,0)[lb]{\smash{{\SetFigFont{10}{12.0}{\rmdefault}{\mddefault}{\updefault}{\color[rgb]{0,0,0}$p$}%
}}}}
\put(1276,-1936){\makebox(0,0)[lb]{\smash{{\SetFigFont{10}{12.0}{\rmdefault}{\mddefault}{\updefault}{\color[rgb]{0,0,0}$\R$}%
}}}}
\put(1201,-286){\makebox(0,0)[lb]{\smash{{\SetFigFont{10}{12.0}{\rmdefault}{\mddefault}{\updefault}{\color[rgb]{0,0,0}$BR$}%
}}}}
\put(1201,-3586){\makebox(0,0)[lb]{\smash{{\SetFigFont{10}{12.0}{\rmdefault}{\mddefault}{\updefault}{\color[rgb]{0,0,0}$TR$}%
}}}}
\put(4501,-3586){\makebox(0,0)[lb]{\smash{{\SetFigFont{10}{12.0}{\rmdefault}{\mddefault}{\updefault}{\color[rgb]{0,0,0}$TL$}%
}}}}
\end{picture}%

%% file: cornered.pstex_t
\begin{picture}(0,0)%
\includegraphics{cornered.pstex}%
\end{picture}%
\setlength{\unitlength}{2368sp}%
\begingroup\makeatletter\ifx\SetFigFont\undefined%
\gdef\SetFigFont#1#2#3#4#5{%
  \reset@font\fontsize{#1}{#2pt}%
  \fontfamily{#3}\fontseries{#4}\fontshape{#5}%
  \selectfont}%
\fi\endgroup%
\begin{picture}(6214,5284)(-879,-5939)
\put(4951,-811){\makebox(0,0)[lb]{\smash{{\SetFigFont{10}{12.0}{\rmdefault}{\mddefault}{\updefault}{\color[rgb]{0,0,0}$P$}%
}}}}
\put(4051,-2461){\makebox(0,0)[lb]{\smash{{\SetFigFont{10}{12.0}{\rmdefault}{\mddefault}{\updefault}{\color[rgb]{0,0,0}$Z$}%
}}}}
\put(2821,-3291){\makebox(0,0)[lb]{\smash{{\SetFigFont{10}{12.0}{\rmdefault}{\mddefault}{\updefault}{\color[rgb]{0,0,0}$\alpha_3^a$}%
}}}}
\put(1031,-4141){\makebox(0,0)[lb]{\smash{{\SetFigFont{10}{12.0}{\rmdefault}{\mddefault}{\updefault}{\color[rgb]{0,0,0}$\beta_2$}%
}}}}
\put(-39,-2876){\makebox(0,0)[lb]{\smash{{\SetFigFont{10}{12.0}{\rmdefault}{\mddefault}{\updefault}{\color[rgb]{0,0,0}$\beta_1$}%
}}}}
\put(-154,-1936){\makebox(0,0)[lb]{\smash{{\SetFigFont{10}{12.0}{\rmdefault}{\mddefault}{\updefault}{\color[rgb]{0,0,0}$\alpha_1^a$}%
}}}}
\put(696,-1656){\makebox(0,0)[lb]{\smash{{\SetFigFont{10}{12.0}{\rmdefault}{\mddefault}{\updefault}{\color[rgb]{0,0,0}$\alpha_2^a$}%
}}}}
\put(-879,-3976){\makebox(0,0)[lb]{\smash{{\SetFigFont{10}{12.0}{\rmdefault}{\mddefault}{\updefault}{\color[rgb]{0,0,0}$\Sigma$}%
}}}}
\put(-324,-5181){\makebox(0,0)[lb]{\smash{{\SetFigFont{10}{12.0}{\rmdefault}{\mddefault}{\updefault}{\color[rgb]{0,0,0}$\beta_3$}%
}}}}
\put(3516,-3656){\makebox(0,0)[lb]{\smash{{\SetFigFont{10}{12.0}{\rmdefault}{\mddefault}{\updefault}{\color[rgb]{0,0,0}$\alpha_4^a$}%
}}}}
\put(366,-5776){\makebox(0,0)[lb]{\smash{{\SetFigFont{10}{12.0}{\rmdefault}{\mddefault}{\updefault}{\color[rgb]{0,0,0}$\alpha_1^c$}%
}}}}
\end{picture}%

%% file: split.pstex_t
\begin{picture}(0,0)%
\includegraphics{split.pstex}%
\end{picture}%
\setlength{\unitlength}{2368sp}%
\begingroup\makeatletter\ifx\SetFigFont\undefined%
\gdef\SetFigFont#1#2#3#4#5{%
  \reset@font\fontsize{#1}{#2pt}%
  \fontfamily{#3}\fontseries{#4}\fontshape{#5}%
  \selectfont}%
\fi\endgroup%
\begin{picture}(6214,5284)(-879,-5939)
\put(4951,-811){\makebox(0,0)[lb]{\smash{{\SetFigFont{10}{12.0}{\rmdefault}{\mddefault}{\updefault}{\color[rgb]{0,0,0}$P$}%
}}}}
\put(4051,-2461){\makebox(0,0)[lb]{\smash{{\SetFigFont{10}{12.0}{\rmdefault}{\mddefault}{\updefault}{\color[rgb]{0,0,0}$Z$}%
}}}}
\put(2821,-3291){\makebox(0,0)[lb]{\smash{{\SetFigFont{10}{12.0}{\rmdefault}{\mddefault}{\updefault}{\color[rgb]{0,0,0}$\alpha_1^a$}%
}}}}
\put(-39,-2876){\makebox(0,0)[lb]{\smash{{\SetFigFont{10}{12.0}{\rmdefault}{\mddefault}{\updefault}{\color[rgb]{0,0,0}$\alpha_1^c$}%
}}}}
\put(-154,-1936){\makebox(0,0)[lb]{\smash{{\SetFigFont{10}{12.0}{\rmdefault}{\mddefault}{\updefault}{\color[rgb]{0,0,0}$\beta_1^a$}%
}}}}
\put(696,-1656){\makebox(0,0)[lb]{\smash{{\SetFigFont{10}{12.0}{\rmdefault}{\mddefault}{\updefault}{\color[rgb]{0,0,0}$\beta_2^a$}%
}}}}
\put(-879,-3976){\makebox(0,0)[lb]{\smash{{\SetFigFont{10}{12.0}{\rmdefault}{\mddefault}{\updefault}{\color[rgb]{0,0,0}$\Sigma$}%
}}}}
\put(-324,-5181){\makebox(0,0)[lb]{\smash{{\SetFigFont{10}{12.0}{\rmdefault}{\mddefault}{\updefault}{\color[rgb]{0,0,0}$\beta_2^c$}%
}}}}
\put(3516,-3656){\makebox(0,0)[lb]{\smash{{\SetFigFont{10}{12.0}{\rmdefault}{\mddefault}{\updefault}{\color[rgb]{0,0,0}$\alpha_2^a$}%
}}}}
\put(366,-5776){\makebox(0,0)[lb]{\smash{{\SetFigFont{10}{12.0}{\rmdefault}{\mddefault}{\updefault}{\color[rgb]{0,0,0}$\alpha_2^c$}%
}}}}
\put(981,-4141){\makebox(0,0)[lb]{\smash{{\SetFigFont{10}{12.0}{\rmdefault}{\mddefault}{\updefault}{\color[rgb]{0,0,0}$\beta_1^c$}%
}}}}
\end{picture}%

%% file: splitcircle.pstex_t
\begin{picture}(0,0)%
\includegraphics{splitcircle.pstex}%
\end{picture}%
\setlength{\unitlength}{2368sp}%
\begingroup\makeatletter\ifx\SetFigFont\undefined%
\gdef\SetFigFont#1#2#3#4#5{%
  \reset@font\fontsize{#1}{#2pt}%
  \fontfamily{#3}\fontseries{#4}\fontshape{#5}%
  \selectfont}%
\fi\endgroup%
\begin{picture}(2700,4224)(2361,-6373)
\put(2361,-4936){\makebox(0,0)[lb]{\smash{{\SetFigFont{10}{12.0}{\rmdefault}{\mddefault}{\updefault}{\color[rgb]{0,0,0}$z$}%
}}}}
\put(4631,-4936){\makebox(0,0)[lb]{\smash{{\SetFigFont{10}{12.0}{\rmdefault}{\mddefault}{\updefault}{\color[rgb]{0,0,0}$w$}%
}}}}
\end{picture}%

%% file: hahd.pstex_t
\begin{picture}(0,0)%
\includegraphics{hahd.pstex}%
\end{picture}%
\setlength{\unitlength}{1973sp}%
\begingroup\makeatletter\ifx\SetFigFont\undefined%
\gdef\SetFigFont#1#2#3#4#5{%
  \reset@font\fontsize{#1}{#2pt}%
  \fontfamily{#3}\fontseries{#4}\fontshape{#5}%
  \selectfont}%
\fi\endgroup%
\begin{picture}(5175,4639)(76,-4433)
\put(3151, 14){\makebox(0,0)[lb]{\smash{{\SetFigFont{10}{12.0}{\rmdefault}{\mddefault}{\updefault}{\color[rgb]{0,0,0}$\ell$}%
}}}}
\put(1351,-1336){\makebox(0,0)[lb]{\smash{{\SetFigFont{8}{9.6}{\rmdefault}{\mddefault}{\updefault}{\color[rgb]{0,0,0}$X$}%
}}}}
\put(2551,-736){\makebox(0,0)[lb]{\smash{{\SetFigFont{8}{9.6}{\rmdefault}{\mddefault}{\updefault}{\color[rgb]{0,0,0}$X$}%
}}}}
\put(3151,-1936){\makebox(0,0)[lb]{\smash{{\SetFigFont{8}{9.6}{\rmdefault}{\mddefault}{\updefault}{\color[rgb]{0,0,0}$X$}%
}}}}
\put(4426,-2536){\makebox(0,0)[lb]{\smash{{\SetFigFont{8}{9.6}{\rmdefault}{\mddefault}{\updefault}{\color[rgb]{0,0,0}$X$}%
}}}}
\put(3751,-3736){\makebox(0,0)[lb]{\smash{{\SetFigFont{8}{9.6}{\rmdefault}{\mddefault}{\updefault}{\color[rgb]{0,0,0}$X$}%
}}}}
\put(1951,-1336){\makebox(0,0)[lb]{\smash{{\SetFigFont{8}{9.6}{\rmdefault}{\mddefault}{\updefault}{\color[rgb]{0,0,0}$O$}%
}}}}
\put(3151,-3136){\makebox(0,0)[lb]{\smash{{\SetFigFont{8}{9.6}{\rmdefault}{\mddefault}{\updefault}{\color[rgb]{0,0,0}$O$}%
}}}}
\put(4351,-736){\makebox(0,0)[lb]{\smash{{\SetFigFont{8}{9.6}{\rmdefault}{\mddefault}{\updefault}{\color[rgb]{0,0,0}$O$}%
}}}}
\put(1351,-1936){\makebox(0,0)[lb]{\smash{{\SetFigFont{8}{9.6}{\rmdefault}{\mddefault}{\updefault}{\color[rgb]{0,0,0}$O$}%
}}}}
\put(3751,-2536){\makebox(0,0)[lb]{\smash{{\SetFigFont{8}{9.6}{\rmdefault}{\mddefault}{\updefault}{\color[rgb]{0,0,0}$O$}%
}}}}
\put(2551,-3736){\makebox(0,0)[lb]{\smash{{\SetFigFont{8}{9.6}{\rmdefault}{\mddefault}{\updefault}{\color[rgb]{0,0,0}$O$}%
}}}}
\put(1951,-3136){\makebox(0,0)[lb]{\smash{{\SetFigFont{8}{9.6}{\rmdefault}{\mddefault}{\updefault}{\color[rgb]{0,0,0}$X$}%
}}}}
\put( 76,-2161){\makebox(0,0)[lb]{\smash{{\SetFigFont{10}{12.0}{\rmdefault}{\mddefault}{\updefault}{\color[rgb]{0,0,0}$\H^A$}%
}}}}
\put(5251,-2161){\makebox(0,0)[lb]{\smash{{\SetFigFont{10}{12.0}{\rmdefault}{\mddefault}{\updefault}{\color[rgb]{0,0,0}$\H^D$}%
}}}}
\end{picture}%

%% file: cpa.pstex_t
\begin{picture}(0,0)%
\includegraphics{cpa.pstex}%
\end{picture}%
\setlength{\unitlength}{3158sp}%
\begingroup\makeatletter\ifx\SetFigFont\undefined%
\gdef\SetFigFont#1#2#3#4#5{%
  \reset@font\fontsize{#1}{#2pt}%
  \fontfamily{#3}\fontseries{#4}\fontshape{#5}%
  \selectfont}%
\fi\endgroup%
\begin{picture}(4737,1824)(1051,-2473)
\put(1051,-1936){\makebox(0,0)[lb]{\smash{{\SetFigFont{11}{13.2}{\rmdefault}{\mddefault}{\updefault}{\color[rgb]{0,0,0}$i$}%
}}}}
\put(1051,-1336){\makebox(0,0)[lb]{\smash{{\SetFigFont{11}{13.2}{\rmdefault}{\mddefault}{\updefault}{\color[rgb]{0,0,0}$j$}%
}}}}
\put(2926,-1636){\makebox(0,0)[lb]{\smash{{\SetFigFont{11}{13.2}{\rmdefault}{\mddefault}{\updefault}{\color[rgb]{0,0,0}$*$}%
}}}}
\put(3901,-1636){\makebox(0,0)[lb]{\smash{{\SetFigFont{11}{13.2}{\rmdefault}{\mddefault}{\updefault}{\color[rgb]{0,0,0}$=$}%
}}}}
\end{picture}%

%% file: cpd.pstex_t
\begin{picture}(0,0)%
\includegraphics{cpd.pstex}%
\end{picture}%
\setlength{\unitlength}{3158sp}%
\begingroup\makeatletter\ifx\SetFigFont\undefined%
\gdef\SetFigFont#1#2#3#4#5{%
  \reset@font\fontsize{#1}{#2pt}%
  \fontfamily{#3}\fontseries{#4}\fontshape{#5}%
  \selectfont}%
\fi\endgroup%
\begin{picture}(4350,1824)(1351,-2473)
\put(1351,-1636){\makebox(0,0)[lb]{\smash{{\SetFigFont{11}{13.2}{\rmdefault}{\mddefault}{\updefault}{\color[rgb]{0,0,0}$\del$}%
}}}}
\put(3076,-1636){\makebox(0,0)[lb]{\smash{{\SetFigFont{11}{13.2}{\rmdefault}{\mddefault}{\updefault}{\color[rgb]{0,0,0}$=$}%
}}}}
\put(5176,-1936){\makebox(0,0)[lb]{\smash{{\SetFigFont{11}{13.2}{\rmdefault}{\mddefault}{\updefault}{\color[rgb]{0,0,0}$i$}%
}}}}
\put(5176,-1336){\makebox(0,0)[lb]{\smash{{\SetFigFont{11}{13.2}{\rmdefault}{\mddefault}{\updefault}{\color[rgb]{0,0,0}$j$}%
}}}}
\put(5701,-1636){\makebox(0,0)[lb]{\smash{{\SetFigFont{11}{13.2}{\rmdefault}{\mddefault}{\updefault}{\color[rgb]{0,0,0}$+ \ \ \ \dots$}%
}}}}
\end{picture}%

%% file: cpaa.pstex_t
\begin{picture}(0,0)%
\includegraphics{cpaa.pstex}%
\end{picture}%
\setlength{\unitlength}{3552sp}%
\begingroup\makeatletter\ifx\SetFigFont\undefined%
\gdef\SetFigFont#1#2#3#4#5{%
  \reset@font\fontsize{#1}{#2pt}%
  \fontfamily{#3}\fontseries{#4}\fontshape{#5}%
  \selectfont}%
\fi\endgroup%
\begin{picture}(2229,2047)(2359,-2473)
\end{picture}%

%% file: Dcpaa.pstex_t
\begin{picture}(0,0)%
\includegraphics{Dcpaa.pstex}%
\end{picture}%
\setlength{\unitlength}{3158sp}%
\begingroup\makeatletter\ifx\SetFigFont\undefined%
\gdef\SetFigFont#1#2#3#4#5{%
  \reset@font\fontsize{#1}{#2pt}%
  \fontfamily{#3}\fontseries{#4}\fontshape{#5}%
  \selectfont}%
\fi\endgroup%
\begin{picture}(8621,2744)(2026,-2848)
\put(2026,-1636){\makebox(0,0)[lb]{\smash{{\SetFigFont{11}{13.2}{\rmdefault}{\mddefault}{\updefault}{\color[rgb]{0,0,0}$\del$}%
}}}}
\put(4426,-1636){\makebox(0,0)[lb]{\smash{{\SetFigFont{11}{13.2}{\rmdefault}{\mddefault}{\updefault}{\color[rgb]{0,0,0}$=$}%
}}}}
\put(7126,-1636){\makebox(0,0)[lb]{\smash{{\SetFigFont{11}{13.2}{\rmdefault}{\mddefault}{\updefault}{\color[rgb]{0,0,0}$+$}%
}}}}
\put(9826,-1636){\makebox(0,0)[lb]{\smash{{\SetFigFont{11}{13.2}{\rmdefault}{\mddefault}{\updefault}{\color[rgb]{0,0,0}$+$}%
}}}}
\put(10201,-1636){\makebox(0,0)[lb]{\smash{{\SetFigFont{11}{13.2}{\rmdefault}{\mddefault}{\updefault}{\color[rgb]{0,0,0}$. . .$}%
}}}}
\end{picture}%

%% file: deldel1.pstex_t
\begin{picture}(0,0)%
\includegraphics{deldel1.pstex}%
\end{picture}%
\setlength{\unitlength}{2526sp}%
\begingroup\makeatletter\ifx\SetFigFont\undefined%
\gdef\SetFigFont#1#2#3#4#5{%
  \reset@font\fontsize{#1}{#2pt}%
  \fontfamily{#3}\fontseries{#4}\fontshape{#5}%
  \selectfont}%
\fi\endgroup%
\begin{picture}(7677,13899)(286,-11098)
\put(301,689){\makebox(0,0)[lb]{\smash{{\SetFigFont{9}{10.8}{\rmdefault}{\mddefault}{\updefault}{\color[rgb]{0,0,0}$(a)$}%
}}}}
\put(2401,-436){\makebox(0,0)[lb]{\smash{{\SetFigFont{8}{9.6}{\rmdefault}{\mddefault}{\updefault}{\color[rgb]{0,0,0}$\del_1$}%
}}}}
\put(2476,2114){\makebox(0,0)[lb]{\smash{{\SetFigFont{8}{9.6}{\rmdefault}{\mddefault}{\updefault}{\color[rgb]{0,0,0}$\del_3$}%
}}}}
\put(6226,-511){\makebox(0,0)[lb]{\smash{{\SetFigFont{8}{9.6}{\rmdefault}{\mddefault}{\updefault}{\color[rgb]{0,0,0}$\del_3$}%
}}}}
\put(301,-8761){\makebox(0,0)[lb]{\smash{{\SetFigFont{9}{10.8}{\rmdefault}{\mddefault}{\updefault}{\color[rgb]{0,0,0}$(c)$}%
}}}}
\put(2401,-10336){\makebox(0,0)[lb]{\smash{{\SetFigFont{8}{9.6}{\rmdefault}{\mddefault}{\updefault}{\color[rgb]{0,0,0}$\del_3$}%
}}}}
\put(2476,-7786){\makebox(0,0)[lb]{\smash{{\SetFigFont{8}{9.6}{\rmdefault}{\mddefault}{\updefault}{\color[rgb]{0,0,0}$\del_3$}%
}}}}
\put(6226,-10411){\makebox(0,0)[lb]{\smash{{\SetFigFont{8}{9.6}{\rmdefault}{\mddefault}{\updefault}{\color[rgb]{0,0,0}$\del_3$}%
}}}}
\put(301,-3811){\makebox(0,0)[lb]{\smash{{\SetFigFont{9}{10.8}{\rmdefault}{\mddefault}{\updefault}{\color[rgb]{0,0,0}$(b)$}%
}}}}
\put(2401,-5386){\makebox(0,0)[lb]{\smash{{\SetFigFont{8}{9.6}{\rmdefault}{\mddefault}{\updefault}{\color[rgb]{0,0,0}$\del_3$}%
}}}}
\put(2476,-2836){\makebox(0,0)[lb]{\smash{{\SetFigFont{8}{9.6}{\rmdefault}{\mddefault}{\updefault}{\color[rgb]{0,0,0}$\del_1$}%
}}}}
\put(6226,-5461){\makebox(0,0)[lb]{\smash{{\SetFigFont{8}{9.6}{\rmdefault}{\mddefault}{\updefault}{\color[rgb]{0,0,0}$\del_3$}%
}}}}
\put(6151,-7861){\makebox(0,0)[lb]{\smash{{\SetFigFont{8}{9.6}{\rmdefault}{\mddefault}{\updefault}{\color[rgb]{0,0,0}$\del_1$}%
}}}}
\put(6151,-2911){\makebox(0,0)[lb]{\smash{{\SetFigFont{8}{9.6}{\rmdefault}{\mddefault}{\updefault}{\color[rgb]{0,0,0}$\del_3$}%
}}}}
\put(6151,2039){\makebox(0,0)[lb]{\smash{{\SetFigFont{8}{9.6}{\rmdefault}{\mddefault}{\updefault}{\color[rgb]{0,0,0}$\del_1$}%
}}}}
\end{picture}%

%% file: deldel2.pstex_t
\begin{picture}(0,0)%
\includegraphics{deldel2.pstex}%
\end{picture}%
\setlength{\unitlength}{2565sp}%
\begingroup\makeatletter\ifx\SetFigFont\undefined%
\gdef\SetFigFont#1#2#3#4#5{%
  \reset@font\fontsize{#1}{#2pt}%
  \fontfamily{#3}\fontseries{#4}\fontshape{#5}%
  \selectfont}%
\fi\endgroup%
\begin{picture}(10899,2106)(1039,845)
\put(11176,914){\makebox(0,0)[lb]{\smash{{\SetFigFont{8}{9.6}{\rmdefault}{\mddefault}{\updefault}{\color[rgb]{0,0,0}$(e)$}%
}}}}
\put(1576,914){\makebox(0,0)[lb]{\smash{{\SetFigFont{8}{9.6}{\rmdefault}{\mddefault}{\updefault}{\color[rgb]{0,0,0}$(a)$}%
}}}}
\put(3976,914){\makebox(0,0)[lb]{\smash{{\SetFigFont{8}{9.6}{\rmdefault}{\mddefault}{\updefault}{\color[rgb]{0,0,0}$(b)$}%
}}}}
\put(6376,914){\makebox(0,0)[lb]{\smash{{\SetFigFont{8}{9.6}{\rmdefault}{\mddefault}{\updefault}{\color[rgb]{0,0,0}$(c)$}%
}}}}
\put(8776,914){\makebox(0,0)[lb]{\smash{{\SetFigFont{8}{9.6}{\rmdefault}{\mddefault}{\updefault}{\color[rgb]{0,0,0}$(d)$}%
}}}}
\end{picture}%

%% file: deldel3.pstex_t
\begin{picture}(0,0)%
\includegraphics{deldel3.pstex}%
\end{picture}%
\setlength{\unitlength}{3158sp}%
\begingroup\makeatletter\ifx\SetFigFont\undefined%
\gdef\SetFigFont#1#2#3#4#5{%
  \reset@font\fontsize{#1}{#2pt}%
  \fontfamily{#3}\fontseries{#4}\fontshape{#5}%
  \selectfont}%
\fi\endgroup%
\begin{picture}(8499,2106)(1039,845)
\put(8776,914){\makebox(0,0)[lb]{\smash{{\SetFigFont{10}{12.0}{\rmdefault}{\mddefault}{\updefault}{\color[rgb]{0,0,0}$(d)$}%
}}}}
\put(1576,914){\makebox(0,0)[lb]{\smash{{\SetFigFont{10}{12.0}{\rmdefault}{\mddefault}{\updefault}{\color[rgb]{0,0,0}$(a)$}%
}}}}
\put(3976,914){\makebox(0,0)[lb]{\smash{{\SetFigFont{10}{12.0}{\rmdefault}{\mddefault}{\updefault}{\color[rgb]{0,0,0}$(b)$}%
}}}}
\put(6376,914){\makebox(0,0)[lb]{\smash{{\SetFigFont{10}{12.0}{\rmdefault}{\mddefault}{\updefault}{\color[rgb]{0,0,0}$(c)$}%
}}}}
\end{picture}%

%% file: cpaa2.pstex_t
\begin{picture}(0,0)%
\includegraphics{cpaa2.pstex}%
\end{picture}%
\setlength{\unitlength}{3552sp}%
\begingroup\makeatletter\ifx\SetFigFont\undefined%
\gdef\SetFigFont#1#2#3#4#5{%
  \reset@font\fontsize{#1}{#2pt}%
  \fontfamily{#3}\fontseries{#4}\fontshape{#5}%
  \selectfont}%
\fi\endgroup%
\begin{picture}(6356,2287)(2357,-2473)
\end{picture}%

%% file: cpaa_t.pstex_t
\begin{picture}(0,0)%
\includegraphics{cpaa_t.pstex}%
\end{picture}%
\setlength{\unitlength}{3158sp}%
\begingroup\makeatletter\ifx\SetFigFont\undefined%
\gdef\SetFigFont#1#2#3#4#5{%
  \reset@font\fontsize{#1}{#2pt}%
  \fontfamily{#3}\fontseries{#4}\fontshape{#5}%
  \selectfont}%
\fi\endgroup%
\begin{picture}(9552,2029)(1564,-8168)
\put(4201,-7486){\makebox(0,0)[lb]{\smash{{\SetFigFont{10}{12.0}{\rmdefault}{\mddefault}{\updefault}{\color[rgb]{0,0,0}$j$}%
}}}}
\put(4201,-7111){\makebox(0,0)[lb]{\smash{{\SetFigFont{10}{12.0}{\rmdefault}{\mddefault}{\updefault}{\color[rgb]{0,0,0}$l$}%
}}}}
\put(3151,-7486){\makebox(0,0)[lb]{\smash{{\SetFigFont{10}{12.0}{\rmdefault}{\mddefault}{\updefault}{\color[rgb]{0,0,0}$x$}%
}}}}
\put(3526,-6286){\makebox(0,0)[lb]{\smash{{\SetFigFont{10}{12.0}{\rmdefault}{\mddefault}{\updefault}{\color[rgb]{0,0,0}$t$}%
}}}}
\put(3151,-7111){\makebox(0,0)[lb]{\smash{{\SetFigFont{10}{12.0}{\rmdefault}{\mddefault}{\updefault}{\color[rgb]{0,0,0}$z$}%
}}}}
\put(1951,-8086){\makebox(0,0)[lb]{\smash{{\SetFigFont{11}{13.2}{\rmdefault}{\mddefault}{\updefault}{\color[rgb]{0,0,0}$(a)$}%
}}}}
\put(1651,-7486){\makebox(0,0)[lb]{\smash{{\SetFigFont{10}{12.0}{\rmdefault}{\mddefault}{\updefault}{\color[rgb]{0,0,0}$x$}%
}}}}
\put(1651,-6961){\makebox(0,0)[lb]{\smash{{\SetFigFont{10}{12.0}{\rmdefault}{\mddefault}{\updefault}{\color[rgb]{0,0,0}$y$}%
}}}}
\put(2476,-7411){\makebox(0,0)[lb]{\smash{{\SetFigFont{10}{12.0}{\rmdefault}{\mddefault}{\updefault}{\color[rgb]{0,0,0}$i$}%
}}}}
\put(2476,-6961){\makebox(0,0)[lb]{\smash{{\SetFigFont{10}{12.0}{\rmdefault}{\mddefault}{\updefault}{\color[rgb]{0,0,0}$j$}%
}}}}
\put(1801,-6286){\makebox(0,0)[lb]{\smash{{\SetFigFont{10}{12.0}{\rmdefault}{\mddefault}{\updefault}{\color[rgb]{0,0,0}$t$}%
}}}}
\put(5401,-8086){\makebox(0,0)[lb]{\smash{{\SetFigFont{11}{13.2}{\rmdefault}{\mddefault}{\updefault}{\color[rgb]{0,0,0}$(c)$}%
}}}}
\put(5926,-7111){\makebox(0,0)[lb]{\smash{{\SetFigFont{10}{12.0}{\rmdefault}{\mddefault}{\updefault}{\color[rgb]{0,0,0}$j$}%
}}}}
\put(5926,-7486){\makebox(0,0)[lb]{\smash{{\SetFigFont{10}{12.0}{\rmdefault}{\mddefault}{\updefault}{\color[rgb]{0,0,0}$i$}%
}}}}
\put(5926,-6736){\makebox(0,0)[lb]{\smash{{\SetFigFont{10}{12.0}{\rmdefault}{\mddefault}{\updefault}{\color[rgb]{0,0,0}$l$}%
}}}}
\put(4951,-7111){\makebox(0,0)[lb]{\smash{{\SetFigFont{10}{12.0}{\rmdefault}{\mddefault}{\updefault}{\color[rgb]{0,0,0}$y$}%
}}}}
\put(4951,-6736){\makebox(0,0)[lb]{\smash{{\SetFigFont{10}{12.0}{\rmdefault}{\mddefault}{\updefault}{\color[rgb]{0,0,0}$z$}%
}}}}
\put(5326,-7561){\makebox(0,0)[lb]{\smash{{\SetFigFont{10}{12.0}{\rmdefault}{\mddefault}{\updefault}{\color[rgb]{0,0,0}$x$}%
}}}}
\put(5476,-6286){\makebox(0,0)[lb]{\smash{{\SetFigFont{10}{12.0}{\rmdefault}{\mddefault}{\updefault}{\color[rgb]{0,0,0}$t$}%
}}}}
\put(10576,-6286){\makebox(0,0)[lb]{\smash{{\SetFigFont{10}{12.0}{\rmdefault}{\mddefault}{\updefault}{\color[rgb]{0,0,0}$t$}%
}}}}
\put(10126,-6661){\makebox(0,0)[lb]{\smash{{\SetFigFont{10}{12.0}{\rmdefault}{\mddefault}{\updefault}{\color[rgb]{0,0,0}$z$}%
}}}}
\put(7126,-8086){\makebox(0,0)[lb]{\smash{{\SetFigFont{11}{13.2}{\rmdefault}{\mddefault}{\updefault}{\color[rgb]{0,0,0}$(d)$}%
}}}}
\put(6976,-6661){\makebox(0,0)[lb]{\smash{{\SetFigFont{10}{12.0}{\rmdefault}{\mddefault}{\updefault}{\color[rgb]{0,0,0}$z$}%
}}}}
\put(7651,-7486){\makebox(0,0)[lb]{\smash{{\SetFigFont{10}{12.0}{\rmdefault}{\mddefault}{\updefault}{\color[rgb]{0,0,0}$i$}%
}}}}
\put(7651,-6736){\makebox(0,0)[lb]{\smash{{\SetFigFont{10}{12.0}{\rmdefault}{\mddefault}{\updefault}{\color[rgb]{0,0,0}$l$}%
}}}}
\put(7651,-7111){\makebox(0,0)[lb]{\smash{{\SetFigFont{10}{12.0}{\rmdefault}{\mddefault}{\updefault}{\color[rgb]{0,0,0}$j$}%
}}}}
\put(6826,-6286){\makebox(0,0)[lb]{\smash{{\SetFigFont{10}{12.0}{\rmdefault}{\mddefault}{\updefault}{\color[rgb]{0,0,0}$t$}%
}}}}
\put(10576,-8086){\makebox(0,0)[lb]{\smash{{\SetFigFont{11}{13.2}{\rmdefault}{\mddefault}{\updefault}{\color[rgb]{0,0,0}$(f)$}%
}}}}
\put(11101,-7486){\makebox(0,0)[lb]{\smash{{\SetFigFont{10}{12.0}{\rmdefault}{\mddefault}{\updefault}{\color[rgb]{0,0,0}$i$}%
}}}}
\put(11101,-6736){\makebox(0,0)[lb]{\smash{{\SetFigFont{10}{12.0}{\rmdefault}{\mddefault}{\updefault}{\color[rgb]{0,0,0}$j$}%
}}}}
\put(11101,-7111){\makebox(0,0)[lb]{\smash{{\SetFigFont{10}{12.0}{\rmdefault}{\mddefault}{\updefault}{\color[rgb]{0,0,0}$l$}%
}}}}
\put(8851,-8086){\makebox(0,0)[lb]{\smash{{\SetFigFont{11}{13.2}{\rmdefault}{\mddefault}{\updefault}{\color[rgb]{0,0,0}$(e)$}%
}}}}
\put(8701,-6661){\makebox(0,0)[lb]{\smash{{\SetFigFont{10}{12.0}{\rmdefault}{\mddefault}{\updefault}{\color[rgb]{0,0,0}$z$}%
}}}}
\put(9376,-7486){\makebox(0,0)[lb]{\smash{{\SetFigFont{10}{12.0}{\rmdefault}{\mddefault}{\updefault}{\color[rgb]{0,0,0}$i$}%
}}}}
\put(9376,-6736){\makebox(0,0)[lb]{\smash{{\SetFigFont{10}{12.0}{\rmdefault}{\mddefault}{\updefault}{\color[rgb]{0,0,0}$j$}%
}}}}
\put(9376,-7111){\makebox(0,0)[lb]{\smash{{\SetFigFont{10}{12.0}{\rmdefault}{\mddefault}{\updefault}{\color[rgb]{0,0,0}$l$}%
}}}}
\put(8701,-7561){\makebox(0,0)[lb]{\smash{{\SetFigFont{10}{12.0}{\rmdefault}{\mddefault}{\updefault}{\color[rgb]{0,0,0}$x$}%
}}}}
\put(8401,-7111){\makebox(0,0)[lb]{\smash{{\SetFigFont{10}{12.0}{\rmdefault}{\mddefault}{\updefault}{\color[rgb]{0,0,0}$y$}%
}}}}
\put(8551,-6286){\makebox(0,0)[lb]{\smash{{\SetFigFont{10}{12.0}{\rmdefault}{\mddefault}{\updefault}{\color[rgb]{0,0,0}$t$}%
}}}}
\put(6676,-7561){\makebox(0,0)[lb]{\smash{{\SetFigFont{10}{12.0}{\rmdefault}{\mddefault}{\updefault}{\color[rgb]{0,0,0}$x$}%
}}}}
\put(6976,-7111){\makebox(0,0)[lb]{\smash{{\SetFigFont{10}{12.0}{\rmdefault}{\mddefault}{\updefault}{\color[rgb]{0,0,0}$y$}%
}}}}
\put(10126,-7561){\makebox(0,0)[lb]{\smash{{\SetFigFont{10}{12.0}{\rmdefault}{\mddefault}{\updefault}{\color[rgb]{0,0,0}$x$}%
}}}}
\put(10426,-7111){\makebox(0,0)[lb]{\smash{{\SetFigFont{10}{12.0}{\rmdefault}{\mddefault}{\updefault}{\color[rgb]{0,0,0}$y$}%
}}}}
\put(3676,-8086){\makebox(0,0)[lb]{\smash{{\SetFigFont{11}{13.2}{\rmdefault}{\mddefault}{\updefault}{\color[rgb]{0,0,0}$(b)$}%
}}}}
\put(3376,-6811){\makebox(0,0)[lb]{\smash{{\SetFigFont{10}{12.0}{\rmdefault}{\mddefault}{\updefault}{\color[rgb]{0,0,0}$y$}%
}}}}
\put(4201,-6811){\makebox(0,0)[lb]{\smash{{\SetFigFont{10}{12.0}{\rmdefault}{\mddefault}{\updefault}{\color[rgb]{0,0,0}$i$}%
}}}}
\end{picture}%

%% file: cpaa_leibniz.pstex_t
\begin{picture}(0,0)%
\includegraphics{cpaa_leibniz.pstex}%
\end{picture}%
\setlength{\unitlength}{3158sp}%
\begingroup\makeatletter\ifx\SetFigFont\undefined%
\gdef\SetFigFont#1#2#3#4#5{%
  \reset@font\fontsize{#1}{#2pt}%
  \fontfamily{#3}\fontseries{#4}\fontshape{#5}%
  \selectfont}%
\fi\endgroup%
\begin{picture}(8502,7209)(-611,-15523)
\put(7876,-12211){\makebox(0,0)[lb]{\smash{{\SetFigFont{10}{12.0}{\rmdefault}{\mddefault}{\updefault}{\color[rgb]{0,0,0}$i$}%
}}}}
\put(4051,-9436){\makebox(0,0)[lb]{\smash{{\SetFigFont{10}{12.0}{\rmdefault}{\mddefault}{\updefault}{\color[rgb]{0,0,0}$i$}%
}}}}
\put(6001,-8461){\makebox(0,0)[lb]{\smash{{\SetFigFont{10}{12.0}{\rmdefault}{\mddefault}{\updefault}{\color[rgb]{0,0,0}$t$}%
}}}}
\put(3451,-8461){\makebox(0,0)[lb]{\smash{{\SetFigFont{10}{12.0}{\rmdefault}{\mddefault}{\updefault}{\color[rgb]{0,0,0}$s$}%
}}}}
\put(3001,-8461){\makebox(0,0)[lb]{\smash{{\SetFigFont{10}{12.0}{\rmdefault}{\mddefault}{\updefault}{\color[rgb]{0,0,0}$t$}%
}}}}
\put(901,-8461){\makebox(0,0)[lb]{\smash{{\SetFigFont{10}{12.0}{\rmdefault}{\mddefault}{\updefault}{\color[rgb]{0,0,0}$t$}%
}}}}
\put(5811,-9661){\makebox(0,0)[lb]{\smash{{\SetFigFont{10}{12.0}{\rmdefault}{\mddefault}{\updefault}{\color[rgb]{0,0,0}$x_3$}%
}}}}
\put(5831,-9361){\makebox(0,0)[lb]{\smash{{\SetFigFont{10}{12.0}{\rmdefault}{\mddefault}{\updefault}{\color[rgb]{0,0,0}$y_1$}%
}}}}
\put(5131,-8761){\makebox(0,0)[lb]{\smash{{\SetFigFont{10}{12.0}{\rmdefault}{\mddefault}{\updefault}{\color[rgb]{0,0,0}$y_2$}%
}}}}
\put(3516,-9566){\makebox(0,0)[lb]{\smash{{\SetFigFont{10}{12.0}{\rmdefault}{\mddefault}{\updefault}{\color[rgb]{0,0,0}$y$}%
}}}}
\put(751,-9436){\makebox(0,0)[lb]{\smash{{\SetFigFont{10}{12.0}{\rmdefault}{\mddefault}{\updefault}{\color[rgb]{0,0,0}$x$}%
}}}}
\put(2926,-9511){\makebox(0,0)[lb]{\smash{{\SetFigFont{10}{12.0}{\rmdefault}{\mddefault}{\updefault}{\color[rgb]{0,0,0}$x$}%
}}}}
\put(6451,-9661){\makebox(0,0)[lb]{\smash{{\SetFigFont{10}{12.0}{\rmdefault}{\mddefault}{\updefault}{\color[rgb]{0,0,0}$i$}%
}}}}
\put(5776,-8761){\makebox(0,0)[lb]{\smash{{\SetFigFont{10}{12.0}{\rmdefault}{\mddefault}{\updefault}{\color[rgb]{0,0,0}$x_2$}%
}}}}
\put(5176,-9361){\makebox(0,0)[lb]{\smash{{\SetFigFont{10}{12.0}{\rmdefault}{\mddefault}{\updefault}{\color[rgb]{0,0,0}$x_1$}%
}}}}
\put(826,-10036){\makebox(0,0)[lb]{\smash{{\SetFigFont{11}{13.2}{\rmdefault}{\mddefault}{\updefault}{\color[rgb]{0,0,0}$(a)$}%
}}}}
\put(5701,-10036){\makebox(0,0)[lb]{\smash{{\SetFigFont{11}{13.2}{\rmdefault}{\mddefault}{\updefault}{\color[rgb]{0,0,0}$(c)$}%
}}}}
\put(1651,-9361){\makebox(0,0)[lb]{\smash{{\SetFigFont{10}{12.0}{\rmdefault}{\mddefault}{\updefault}{\color[rgb]{0,0,0}$i$}%
}}}}
\put(3226,-10036){\makebox(0,0)[lb]{\smash{{\SetFigFont{11}{13.2}{\rmdefault}{\mddefault}{\updefault}{\color[rgb]{0,0,0}$(b)$}%
}}}}
\put(451,-14911){\makebox(0,0)[lb]{\smash{{\SetFigFont{10}{12.0}{\rmdefault}{\mddefault}{\updefault}{\color[rgb]{0,0,0}$x_1$}%
}}}}
\put(1106,-14496){\makebox(0,0)[lb]{\smash{{\SetFigFont{10}{12.0}{\rmdefault}{\mddefault}{\updefault}{\color[rgb]{0,0,0}$x_2$}%
}}}}
\put(1651,-14386){\makebox(0,0)[lb]{\smash{{\SetFigFont{10}{12.0}{\rmdefault}{\mddefault}{\updefault}{\color[rgb]{0,0,0}$i$}%
}}}}
\put(1126,-14911){\makebox(0,0)[lb]{\smash{{\SetFigFont{10}{12.0}{\rmdefault}{\mddefault}{\updefault}{\color[rgb]{0,0,0}$y$}%
}}}}
\put(601,-13786){\makebox(0,0)[lb]{\smash{{\SetFigFont{10}{12.0}{\rmdefault}{\mddefault}{\updefault}{\color[rgb]{0,0,0}$t$}%
}}}}
\put(2251,-12736){\makebox(0,0)[lb]{\smash{{\SetFigFont{11}{13.2}{\rmdefault}{\mddefault}{\updefault}{\color[rgb]{0,0,0}$(e)$}%
}}}}
\put(1951,-11086){\makebox(0,0)[lb]{\smash{{\SetFigFont{10}{12.0}{\rmdefault}{\mddefault}{\updefault}{\color[rgb]{0,0,0}$t$}%
}}}}
\put(2326,-11086){\makebox(0,0)[lb]{\smash{{\SetFigFont{10}{12.0}{\rmdefault}{\mddefault}{\updefault}{\color[rgb]{0,0,0}$u$}%
}}}}
\put(2626,-11086){\makebox(0,0)[lb]{\smash{{\SetFigFont{10}{12.0}{\rmdefault}{\mddefault}{\updefault}{\color[rgb]{0,0,0}$s$}%
}}}}
\put(-149,-12736){\makebox(0,0)[lb]{\smash{{\SetFigFont{11}{13.2}{\rmdefault}{\mddefault}{\updefault}{\color[rgb]{0,0,0}$(d)$}%
}}}}
\put(-524,-11911){\makebox(0,0)[lb]{\smash{{\SetFigFont{10}{12.0}{\rmdefault}{\mddefault}{\updefault}{\color[rgb]{0,0,0}$x_1$}%
}}}}
\put( 76,-11911){\makebox(0,0)[lb]{\smash{{\SetFigFont{10}{12.0}{\rmdefault}{\mddefault}{\updefault}{\color[rgb]{0,0,0}$y$}%
}}}}
\put( 16,-12286){\makebox(0,0)[lb]{\smash{{\SetFigFont{10}{12.0}{\rmdefault}{\mddefault}{\updefault}{\color[rgb]{0,0,0}$x_2$}%
}}}}
\put(-449,-11086){\makebox(0,0)[lb]{\smash{{\SetFigFont{10}{12.0}{\rmdefault}{\mddefault}{\updefault}{\color[rgb]{0,0,0}$u$}%
}}}}
\put(-74,-11086){\makebox(0,0)[lb]{\smash{{\SetFigFont{10}{12.0}{\rmdefault}{\mddefault}{\updefault}{\color[rgb]{0,0,0}$s$}%
}}}}
\put(226,-11086){\makebox(0,0)[lb]{\smash{{\SetFigFont{10}{12.0}{\rmdefault}{\mddefault}{\updefault}{\color[rgb]{0,0,0}$t$}%
}}}}
\put(4651,-12736){\makebox(0,0)[lb]{\smash{{\SetFigFont{11}{13.2}{\rmdefault}{\mddefault}{\updefault}{\color[rgb]{0,0,0}$(f)$}%
}}}}
\put(4351,-11086){\makebox(0,0)[lb]{\smash{{\SetFigFont{10}{12.0}{\rmdefault}{\mddefault}{\updefault}{\color[rgb]{0,0,0}$u$}%
}}}}
\put(4726,-11086){\makebox(0,0)[lb]{\smash{{\SetFigFont{10}{12.0}{\rmdefault}{\mddefault}{\updefault}{\color[rgb]{0,0,0}$t$}%
}}}}
\put(5026,-11086){\makebox(0,0)[lb]{\smash{{\SetFigFont{10}{12.0}{\rmdefault}{\mddefault}{\updefault}{\color[rgb]{0,0,0}$s$}%
}}}}
\put(7051,-12736){\makebox(0,0)[lb]{\smash{{\SetFigFont{11}{13.2}{\rmdefault}{\mddefault}{\updefault}{\color[rgb]{0,0,0}$(g)$}%
}}}}
\put(6751,-11086){\makebox(0,0)[lb]{\smash{{\SetFigFont{10}{12.0}{\rmdefault}{\mddefault}{\updefault}{\color[rgb]{0,0,0}$u$}%
}}}}
\put(7126,-11086){\makebox(0,0)[lb]{\smash{{\SetFigFont{10}{12.0}{\rmdefault}{\mddefault}{\updefault}{\color[rgb]{0,0,0}$t$}%
}}}}
\put(7426,-11086){\makebox(0,0)[lb]{\smash{{\SetFigFont{10}{12.0}{\rmdefault}{\mddefault}{\updefault}{\color[rgb]{0,0,0}$s$}%
}}}}
\put(6001,-15436){\makebox(0,0)[lb]{\smash{{\SetFigFont{11}{13.2}{\rmdefault}{\mddefault}{\updefault}{\color[rgb]{0,0,0}$(j)$}%
}}}}
\put(6751,-14836){\makebox(0,0)[lb]{\smash{{\SetFigFont{10}{12.0}{\rmdefault}{\mddefault}{\updefault}{\color[rgb]{0,0,0}$i$}%
}}}}
\put(6751,-14311){\makebox(0,0)[lb]{\smash{{\SetFigFont{10}{12.0}{\rmdefault}{\mddefault}{\updefault}{\color[rgb]{0,0,0}$j$}%
}}}}
\put(6001,-14911){\makebox(0,0)[lb]{\smash{{\SetFigFont{10}{12.0}{\rmdefault}{\mddefault}{\updefault}{\color[rgb]{0,0,0}$x_1$}%
}}}}
\put(6076,-14386){\makebox(0,0)[lb]{\smash{{\SetFigFont{10}{12.0}{\rmdefault}{\mddefault}{\updefault}{\color[rgb]{0,0,0}$y$}%
}}}}
\put(5491,-14311){\makebox(0,0)[lb]{\smash{{\SetFigFont{10}{12.0}{\rmdefault}{\mddefault}{\updefault}{\color[rgb]{0,0,0}$x_2$}%
}}}}
\put(6151,-13786){\makebox(0,0)[lb]{\smash{{\SetFigFont{10}{12.0}{\rmdefault}{\mddefault}{\updefault}{\color[rgb]{0,0,0}$s$}%
}}}}
\put(5701,-13786){\makebox(0,0)[lb]{\smash{{\SetFigFont{10}{12.0}{\rmdefault}{\mddefault}{\updefault}{\color[rgb]{0,0,0}$t$}%
}}}}
\put(3376,-15436){\makebox(0,0)[lb]{\smash{{\SetFigFont{11}{13.2}{\rmdefault}{\mddefault}{\updefault}{\color[rgb]{0,0,0}$(i)$}%
}}}}
\put(4201,-14836){\makebox(0,0)[lb]{\smash{{\SetFigFont{10}{12.0}{\rmdefault}{\mddefault}{\updefault}{\color[rgb]{0,0,0}$i$}%
}}}}
\put(4201,-14386){\makebox(0,0)[lb]{\smash{{\SetFigFont{10}{12.0}{\rmdefault}{\mddefault}{\updefault}{\color[rgb]{0,0,0}$j$}%
}}}}
\put(3001,-14911){\makebox(0,0)[lb]{\smash{{\SetFigFont{10}{12.0}{\rmdefault}{\mddefault}{\updefault}{\color[rgb]{0,0,0}$x_1$}%
}}}}
\put(3656,-14496){\makebox(0,0)[lb]{\smash{{\SetFigFont{10}{12.0}{\rmdefault}{\mddefault}{\updefault}{\color[rgb]{0,0,0}$x_2$}%
}}}}
\put(3021,-14311){\makebox(0,0)[lb]{\smash{{\SetFigFont{10}{12.0}{\rmdefault}{\mddefault}{\updefault}{\color[rgb]{0,0,0}$y$}%
}}}}
\put(3526,-13786){\makebox(0,0)[lb]{\smash{{\SetFigFont{10}{12.0}{\rmdefault}{\mddefault}{\updefault}{\color[rgb]{0,0,0}$t$}%
}}}}
\put(826,-15436){\makebox(0,0)[lb]{\smash{{\SetFigFont{11}{13.2}{\rmdefault}{\mddefault}{\updefault}{\color[rgb]{0,0,0}$(h)$}%
}}}}
\put(1726,-11761){\makebox(0,0)[lb]{\smash{{\SetFigFont{10}{12.0}{\rmdefault}{\mddefault}{\updefault}{\color[rgb]{0,0,0}$x_1$}%
}}}}
\put(2776,-12361){\makebox(0,0)[lb]{\smash{{\SetFigFont{10}{12.0}{\rmdefault}{\mddefault}{\updefault}{\color[rgb]{0,0,0}$y$}%
}}}}
\put(4201,-12286){\makebox(0,0)[lb]{\smash{{\SetFigFont{10}{12.0}{\rmdefault}{\mddefault}{\updefault}{\color[rgb]{0,0,0}$x_1$}%
}}}}
\put(5176,-12286){\makebox(0,0)[lb]{\smash{{\SetFigFont{10}{12.0}{\rmdefault}{\mddefault}{\updefault}{\color[rgb]{0,0,0}$y$}%
}}}}
\put(7576,-11986){\makebox(0,0)[lb]{\smash{{\SetFigFont{10}{12.0}{\rmdefault}{\mddefault}{\updefault}{\color[rgb]{0,0,0}$y$}%
}}}}
\put(6526,-11836){\makebox(0,0)[lb]{\smash{{\SetFigFont{10}{12.0}{\rmdefault}{\mddefault}{\updefault}{\color[rgb]{0,0,0}$x_1$}%
}}}}
\put(4576,-11911){\makebox(0,0)[lb]{\smash{{\SetFigFont{10}{12.0}{\rmdefault}{\mddefault}{\updefault}{\color[rgb]{0,0,0}$x_2$}%
}}}}
\put(6976,-12286){\makebox(0,0)[lb]{\smash{{\SetFigFont{10}{12.0}{\rmdefault}{\mddefault}{\updefault}{\color[rgb]{0,0,0}$x_2$}%
}}}}
\put(2101,-12286){\makebox(0,0)[lb]{\smash{{\SetFigFont{10}{12.0}{\rmdefault}{\mddefault}{\updefault}{\color[rgb]{0,0,0}$x_2$}%
}}}}
\put(676,-11836){\makebox(0,0)[lb]{\smash{{\SetFigFont{10}{12.0}{\rmdefault}{\mddefault}{\updefault}{\color[rgb]{0,0,0}$i$}%
}}}}
\put(3076,-11836){\makebox(0,0)[lb]{\smash{{\SetFigFont{10}{12.0}{\rmdefault}{\mddefault}{\updefault}{\color[rgb]{0,0,0}$i$}%
}}}}
\put(5476,-11836){\makebox(0,0)[lb]{\smash{{\SetFigFont{10}{12.0}{\rmdefault}{\mddefault}{\updefault}{\color[rgb]{0,0,0}$i$}%
}}}}
\end{picture}%

%% file: cpaa25.pstex_t
\begin{picture}(0,0)%
\includegraphics{cpaa25.pstex}%
\end{picture}%
\setlength{\unitlength}{3158sp}%
\begingroup\makeatletter\ifx\SetFigFont\undefined%
\gdef\SetFigFont#1#2#3#4#5{%
  \reset@font\fontsize{#1}{#2pt}%
  \fontfamily{#3}\fontseries{#4}\fontshape{#5}%
  \selectfont}%
\fi\endgroup%
\begin{picture}(5799,2874)(2389,-2473)
\end{picture}%

%% file: cpaa3.pstex_t
\begin{picture}(0,0)%
\includegraphics{cpaa3.pstex}%
\end{picture}%
\setlength{\unitlength}{3158sp}%
\begingroup\makeatletter\ifx\SetFigFont\undefined%
\gdef\SetFigFont#1#2#3#4#5{%
  \reset@font\fontsize{#1}{#2pt}%
  \fontfamily{#3}\fontseries{#4}\fontshape{#5}%
  \selectfont}%
\fi\endgroup%
\begin{picture}(5681,2874)(2357,-2473)
\end{picture}%

%% file: cpaa2caps.pstex_t
\begin{picture}(0,0)%
\includegraphics{cpaa2caps.pstex}%
\end{picture}%
\setlength{\unitlength}{3158sp}%
\begingroup\makeatletter\ifx\SetFigFont\undefined%
\gdef\SetFigFont#1#2#3#4#5{%
  \reset@font\fontsize{#1}{#2pt}%
  \fontfamily{#3}\fontseries{#4}\fontshape{#5}%
  \selectfont}%
\fi\endgroup%
\begin{picture}(5799,1899)(889,-2773)
\put(5926,-1336){\makebox(0,0)[lb]{\smash{{\SetFigFont{10}{12.0}{\rmdefault}{\mddefault}{\updefault}{\color[rgb]{0,0,0}$\cdot$}%
}}}}
\put(1576,-1336){\makebox(0,0)[lb]{\smash{{\SetFigFont{10}{12.0}{\rmdefault}{\mddefault}{\updefault}{\color[rgb]{0,0,0}$\cdot$}%
}}}}
\end{picture}%

%% file: cpaa_leibniz2.pstex_t
\begin{picture}(0,0)%
\includegraphics{cpaa_leibniz2.pstex}%
\end{picture}%
\setlength{\unitlength}{2881sp}%
\begingroup\makeatletter\ifx\SetFigFont\undefined%
\gdef\SetFigFont#1#2#3#4#5{%
  \reset@font\fontsize{#1}{#2pt}%
  \fontfamily{#3}\fontseries{#4}\fontshape{#5}%
  \selectfont}%
\fi\endgroup%
\begin{picture}(7905,12774)(886,-16123)
\put(6526,-6886){\makebox(0,0)[lb]{\smash{{\SetFigFont{9}{10.8}{\rmdefault}{\mddefault}{\updefault}{\color[rgb]{0,0,0}$\del$}%
}}}}
\put(6376,-10561){\makebox(0,0)[lb]{\smash{{\SetFigFont{10}{12.0}{\rmdefault}{\mddefault}{\updefault}{\color[rgb]{0,0,0}$\cdot$}%
}}}}
\put(1801,-15586){\makebox(0,0)[lb]{\smash{{\SetFigFont{9}{10.8}{\rmdefault}{\mddefault}{\updefault}{\color[rgb]{0,0,0}$\del$}%
}}}}
\put(901,-13186){\makebox(0,0)[lb]{\smash{{\SetFigFont{9}{10.8}{\rmdefault}{\mddefault}{\updefault}{\color[rgb]{0,0,0}$(b)$}%
}}}}
\put(2926,-12961){\makebox(0,0)[lb]{\smash{{\SetFigFont{10}{12.0}{\rmdefault}{\mddefault}{\updefault}{\color[rgb]{0,0,0}$\cdot$}%
}}}}
\put(4276,-13711){\makebox(0,0)[lb]{\smash{{\SetFigFont{9}{10.8}{\rmdefault}{\mddefault}{\updefault}{\color[rgb]{0,0,0}$=$}%
}}}}
\put(4126,-15661){\makebox(0,0)[lb]{\smash{{\SetFigFont{9}{10.8}{\rmdefault}{\mddefault}{\updefault}{\color[rgb]{0,0,0}$+$}%
}}}}
\put(8776,-11236){\makebox(0,0)[lb]{\smash{{\SetFigFont{9}{10.8}{\rmdefault}{\mddefault}{\updefault}{\color[rgb]{0,0,0}$0$}%
}}}}
\put(7651,-11236){\makebox(0,0)[lb]{\smash{{\SetFigFont{9}{10.8}{\rmdefault}{\mddefault}{\updefault}{\color[rgb]{0,0,0}$= 0$}%
}}}}
\put(1801,-13261){\makebox(0,0)[lb]{\smash{{\SetFigFont{9}{10.8}{\rmdefault}{\mddefault}{\updefault}{\color[rgb]{0,0,0}$\cdot (\del \xi_2)$}%
}}}}
\put(6376,-15511){\makebox(0,0)[lb]{\smash{{\SetFigFont{9}{10.8}{\rmdefault}{\mddefault}{\updefault}{\color[rgb]{0,0,0}$\cdot \xi_2$}%
}}}}
\put(4876,-4261){\makebox(0,0)[lb]{\smash{{\SetFigFont{9}{10.8}{\rmdefault}{\mddefault}{\updefault}{\color[rgb]{0,0,0}$\cdot (\del \xi_2)$}%
}}}}
\put(7876,-4411){\makebox(0,0)[lb]{\smash{{\SetFigFont{9}{10.8}{\rmdefault}{\mddefault}{\updefault}{\color[rgb]{0,0,0}$= 0$}%
}}}}
\put(6376,-3811){\makebox(0,0)[lb]{\smash{{\SetFigFont{10}{12.0}{\rmdefault}{\mddefault}{\updefault}{\color[rgb]{0,0,0}$\cdot$}%
}}}}
\put(1801,-8836){\makebox(0,0)[lb]{\smash{{\SetFigFont{9}{10.8}{\rmdefault}{\mddefault}{\updefault}{\color[rgb]{0,0,0}$\del$}%
}}}}
\put(901,-6436){\makebox(0,0)[lb]{\smash{{\SetFigFont{9}{10.8}{\rmdefault}{\mddefault}{\updefault}{\color[rgb]{0,0,0}$(a)$}%
}}}}
\put(2926,-6211){\makebox(0,0)[lb]{\smash{{\SetFigFont{10}{12.0}{\rmdefault}{\mddefault}{\updefault}{\color[rgb]{0,0,0}$\cdot$}%
}}}}
\put(4276,-6961){\makebox(0,0)[lb]{\smash{{\SetFigFont{9}{10.8}{\rmdefault}{\mddefault}{\updefault}{\color[rgb]{0,0,0}$=$}%
}}}}
\put(4126,-8911){\makebox(0,0)[lb]{\smash{{\SetFigFont{9}{10.8}{\rmdefault}{\mddefault}{\updefault}{\color[rgb]{0,0,0}$+$}%
}}}}
\put(6376,-8761){\makebox(0,0)[lb]{\smash{{\SetFigFont{9}{10.8}{\rmdefault}{\mddefault}{\updefault}{\color[rgb]{0,0,0}$\cdot \xi_2$}%
}}}}
\put(1801,-6511){\makebox(0,0)[lb]{\smash{{\SetFigFont{9}{10.8}{\rmdefault}{\mddefault}{\updefault}{\color[rgb]{0,0,0}$\cdot \xi_2$}%
}}}}
\put(5026,-11011){\makebox(0,0)[lb]{\smash{{\SetFigFont{9}{10.8}{\rmdefault}{\mddefault}{\updefault}{\color[rgb]{0,0,0}$\cdot \xi_2$}%
}}}}
\put(8326,-11086){\makebox(0,0)[lb]{\smash{{\SetFigFont{9}{10.8}{\rmdefault}{\mddefault}{\updefault}{\color[rgb]{0,0,0}$\del$}%
}}}}
\end{picture}%

%% file: cpaa_lc1.pstex_t
\begin{picture}(0,0)%
\includegraphics{cpaa_lc1.pstex}%
\end{picture}%
\setlength{\unitlength}{3158sp}%
\begingroup\makeatletter\ifx\SetFigFont\undefined%
\gdef\SetFigFont#1#2#3#4#5{%
  \reset@font\fontsize{#1}{#2pt}%
  \fontfamily{#3}\fontseries{#4}\fontshape{#5}%
  \selectfont}%
\fi\endgroup%
\begin{picture}(7899,5049)(1414,-14248)
\put(3376,-10936){\makebox(0,0)[lb]{\smash{{\SetFigFont{10}{12.0}{\rmdefault}{\mddefault}{\updefault}{\color[rgb]{0,0,0}$*$}%
}}}}
\put(3001,-9886){\makebox(0,0)[lb]{\smash{{\SetFigFont{12}{14.4}{\rmdefault}{\mddefault}{\updefault}{\color[rgb]{0,0,0}$\cdot$}%
}}}}
\put(4426,-10411){\makebox(0,0)[lb]{\smash{{\SetFigFont{10}{12.0}{\rmdefault}{\mddefault}{\updefault}{\color[rgb]{0,0,0}$=$}%
}}}}
\put(6751,-10411){\makebox(0,0)[lb]{\smash{{\SetFigFont{10}{12.0}{\rmdefault}{\mddefault}{\updefault}{\color[rgb]{0,0,0}$=$}%
}}}}
\put(3601,-9436){\makebox(0,0)[lb]{\smash{{\SetFigFont{10}{12.0}{\rmdefault}{\mddefault}{\updefault}{\color[rgb]{0,0,0}$*$}%
}}}}
\put(5626,-10036){\makebox(0,0)[lb]{\smash{{\SetFigFont{12}{14.4}{\rmdefault}{\mddefault}{\updefault}{\color[rgb]{0,0,0}$\cdot$}%
}}}}
\put(1951,-12886){\makebox(0,0)[lb]{\smash{{\SetFigFont{12}{14.4}{\rmdefault}{\mddefault}{\updefault}{\color[rgb]{0,0,0}$\cdot$}%
}}}}
\put(3151,-13261){\makebox(0,0)[lb]{\smash{{\SetFigFont{10}{12.0}{\rmdefault}{\mddefault}{\updefault}{\color[rgb]{0,0,0}$*$}%
}}}}
\put(4276,-13261){\makebox(0,0)[lb]{\smash{{\SetFigFont{10}{12.0}{\rmdefault}{\mddefault}{\updefault}{\color[rgb]{0,0,0}$=$}%
}}}}
\put(6301,-13336){\makebox(0,0)[lb]{\smash{{\SetFigFont{10}{12.0}{\rmdefault}{\mddefault}{\updefault}{\color[rgb]{0,0,0}$*$}%
}}}}
\put(7351,-13261){\makebox(0,0)[lb]{\smash{{\SetFigFont{10}{12.0}{\rmdefault}{\mddefault}{\updefault}{\color[rgb]{0,0,0}$=$}%
}}}}
\put(3697,-12721){\makebox(0,0)[lb]{\smash{{\SetFigFont{12}{14.4}{\rmdefault}{\mddefault}{\updefault}{\color[rgb]{0,0,0}$\cdot$}%
}}}}
\end{picture}%

%% file: cpaa_lc2.pstex_t
\begin{picture}(0,0)%
\includegraphics{cpaa_lc2.pstex}%
\end{picture}%
\setlength{\unitlength}{3158sp}%
\begingroup\makeatletter\ifx\SetFigFont\undefined%
\gdef\SetFigFont#1#2#3#4#5{%
  \reset@font\fontsize{#1}{#2pt}%
  \fontfamily{#3}\fontseries{#4}\fontshape{#5}%
  \selectfont}%
\fi\endgroup%
\begin{picture}(7899,5049)(1414,-14248)
\put(3376,-10936){\makebox(0,0)[lb]{\smash{{\SetFigFont{10}{12.0}{\rmdefault}{\mddefault}{\updefault}{\color[rgb]{0,0,0}$*$}%
}}}}
\put(3001,-9886){\makebox(0,0)[lb]{\smash{{\SetFigFont{12}{14.4}{\rmdefault}{\mddefault}{\updefault}{\color[rgb]{0,0,0}$\cdot$}%
}}}}
\put(4426,-10411){\makebox(0,0)[lb]{\smash{{\SetFigFont{10}{12.0}{\rmdefault}{\mddefault}{\updefault}{\color[rgb]{0,0,0}$=$}%
}}}}
\put(6751,-10411){\makebox(0,0)[lb]{\smash{{\SetFigFont{10}{12.0}{\rmdefault}{\mddefault}{\updefault}{\color[rgb]{0,0,0}$=$}%
}}}}
\put(3601,-9436){\makebox(0,0)[lb]{\smash{{\SetFigFont{10}{12.0}{\rmdefault}{\mddefault}{\updefault}{\color[rgb]{0,0,0}$*$}%
}}}}
\put(5626,-10036){\makebox(0,0)[lb]{\smash{{\SetFigFont{12}{14.4}{\rmdefault}{\mddefault}{\updefault}{\color[rgb]{0,0,0}$\cdot$}%
}}}}
\put(1951,-12886){\makebox(0,0)[lb]{\smash{{\SetFigFont{12}{14.4}{\rmdefault}{\mddefault}{\updefault}{\color[rgb]{0,0,0}$\cdot$}%
}}}}
\put(3151,-13261){\makebox(0,0)[lb]{\smash{{\SetFigFont{10}{12.0}{\rmdefault}{\mddefault}{\updefault}{\color[rgb]{0,0,0}$*$}%
}}}}
\put(4276,-13261){\makebox(0,0)[lb]{\smash{{\SetFigFont{10}{12.0}{\rmdefault}{\mddefault}{\updefault}{\color[rgb]{0,0,0}$=$}%
}}}}
\put(6301,-13336){\makebox(0,0)[lb]{\smash{{\SetFigFont{10}{12.0}{\rmdefault}{\mddefault}{\updefault}{\color[rgb]{0,0,0}$*$}%
}}}}
\put(7351,-13261){\makebox(0,0)[lb]{\smash{{\SetFigFont{10}{12.0}{\rmdefault}{\mddefault}{\updefault}{\color[rgb]{0,0,0}$=$}%
}}}}
\put(3697,-12721){\makebox(0,0)[lb]{\smash{{\SetFigFont{12}{14.4}{\rmdefault}{\mddefault}{\updefault}{\color[rgb]{0,0,0}$\cdot$}%
}}}}
\end{picture}%

%% file: cpaa_lc2b.pstex_t
\begin{picture}(0,0)%
\includegraphics{cpaa_lc2b.pstex}%
\end{picture}%
\setlength{\unitlength}{3158sp}%
\begingroup\makeatletter\ifx\SetFigFont\undefined%
\gdef\SetFigFont#1#2#3#4#5{%
  \reset@font\fontsize{#1}{#2pt}%
  \fontfamily{#3}\fontseries{#4}\fontshape{#5}%
  \selectfont}%
\fi\endgroup%
\begin{picture}(6927,4749)(2014,-13948)
\put(5776,-12961){\makebox(0,0)[lb]{\smash{{\SetFigFont{10}{12.0}{\rmdefault}{\mddefault}{\updefault}{\color[rgb]{0,0,0}$*$}%
}}}}
\put(6901,-12961){\makebox(0,0)[lb]{\smash{{\SetFigFont{10}{12.0}{\rmdefault}{\mddefault}{\updefault}{\color[rgb]{0,0,0}$= 0$}%
}}}}
\put(6322,-12421){\makebox(0,0)[lb]{\smash{{\SetFigFont{12}{14.4}{\rmdefault}{\mddefault}{\updefault}{\color[rgb]{0,0,0}$\cdot$}%
}}}}
\put(4801,-12511){\makebox(0,0)[lb]{\smash{{\SetFigFont{12}{14.4}{\rmdefault}{\mddefault}{\updefault}{\color[rgb]{0,0,0}$\cdot$}%
}}}}
\put(3376,-10936){\makebox(0,0)[lb]{\smash{{\SetFigFont{10}{12.0}{\rmdefault}{\mddefault}{\updefault}{\color[rgb]{0,0,0}$*$}%
}}}}
\put(3001,-9886){\makebox(0,0)[lb]{\smash{{\SetFigFont{12}{14.4}{\rmdefault}{\mddefault}{\updefault}{\color[rgb]{0,0,0}$\cdot$}%
}}}}
\put(4426,-10411){\makebox(0,0)[lb]{\smash{{\SetFigFont{10}{12.0}{\rmdefault}{\mddefault}{\updefault}{\color[rgb]{0,0,0}$=$}%
}}}}
\put(6751,-10411){\makebox(0,0)[lb]{\smash{{\SetFigFont{10}{12.0}{\rmdefault}{\mddefault}{\updefault}{\color[rgb]{0,0,0}$=$}%
}}}}
\put(3601,-9436){\makebox(0,0)[lb]{\smash{{\SetFigFont{10}{12.0}{\rmdefault}{\mddefault}{\updefault}{\color[rgb]{0,0,0}$*$}%
}}}}
\put(5776,-10036){\makebox(0,0)[lb]{\smash{{\SetFigFont{12}{14.4}{\rmdefault}{\mddefault}{\updefault}{\color[rgb]{0,0,0}$\cdot$}%
}}}}
\put(8926,-10411){\makebox(0,0)[lb]{\smash{{\SetFigFont{10}{12.0}{\rmdefault}{\mddefault}{\updefault}{\color[rgb]{0,0,0}$= 0$}%
}}}}
\end{picture}%

%% file: cpaa_lc2c.pstex_t
\begin{picture}(0,0)%
\includegraphics{cpaa_lc2c.pstex}%
\end{picture}%
\setlength{\unitlength}{3158sp}%
\begingroup\makeatletter\ifx\SetFigFont\undefined%
\gdef\SetFigFont#1#2#3#4#5{%
  \reset@font\fontsize{#1}{#2pt}%
  \fontfamily{#3}\fontseries{#4}\fontshape{#5}%
  \selectfont}%
\fi\endgroup%
\begin{picture}(7899,5049)(1414,-14248)
\put(3376,-10936){\makebox(0,0)[lb]{\smash{{\SetFigFont{10}{12.0}{\rmdefault}{\mddefault}{\updefault}{\color[rgb]{0,0,0}$*$}%
}}}}
\put(3001,-9886){\makebox(0,0)[lb]{\smash{{\SetFigFont{12}{14.4}{\rmdefault}{\mddefault}{\updefault}{\color[rgb]{0,0,0}$\cdot$}%
}}}}
\put(4426,-10411){\makebox(0,0)[lb]{\smash{{\SetFigFont{10}{12.0}{\rmdefault}{\mddefault}{\updefault}{\color[rgb]{0,0,0}$=$}%
}}}}
\put(6751,-10411){\makebox(0,0)[lb]{\smash{{\SetFigFont{10}{12.0}{\rmdefault}{\mddefault}{\updefault}{\color[rgb]{0,0,0}$=$}%
}}}}
\put(3601,-9436){\makebox(0,0)[lb]{\smash{{\SetFigFont{10}{12.0}{\rmdefault}{\mddefault}{\updefault}{\color[rgb]{0,0,0}$*$}%
}}}}
\put(1951,-12886){\makebox(0,0)[lb]{\smash{{\SetFigFont{12}{14.4}{\rmdefault}{\mddefault}{\updefault}{\color[rgb]{0,0,0}$\cdot$}%
}}}}
\put(3151,-13261){\makebox(0,0)[lb]{\smash{{\SetFigFont{10}{12.0}{\rmdefault}{\mddefault}{\updefault}{\color[rgb]{0,0,0}$*$}%
}}}}
\put(4276,-13261){\makebox(0,0)[lb]{\smash{{\SetFigFont{10}{12.0}{\rmdefault}{\mddefault}{\updefault}{\color[rgb]{0,0,0}$=$}%
}}}}
\put(6301,-13336){\makebox(0,0)[lb]{\smash{{\SetFigFont{10}{12.0}{\rmdefault}{\mddefault}{\updefault}{\color[rgb]{0,0,0}$*$}%
}}}}
\put(7351,-13261){\makebox(0,0)[lb]{\smash{{\SetFigFont{10}{12.0}{\rmdefault}{\mddefault}{\updefault}{\color[rgb]{0,0,0}$=$}%
}}}}
\put(3697,-12721){\makebox(0,0)[lb]{\smash{{\SetFigFont{12}{14.4}{\rmdefault}{\mddefault}{\updefault}{\color[rgb]{0,0,0}$\cdot$}%
}}}}
\put(5776,-9961){\makebox(0,0)[lb]{\smash{{\SetFigFont{12}{14.4}{\rmdefault}{\mddefault}{\updefault}{\color[rgb]{0,0,0}$\cdot$}%
}}}}
\end{picture}%

%% file: cpaa_lc3.pstex_t
\begin{picture}(0,0)%
\includegraphics{cpaa_lc3.pstex}%
\end{picture}%
\setlength{\unitlength}{3158sp}%
\begingroup\makeatletter\ifx\SetFigFont\undefined%
\gdef\SetFigFont#1#2#3#4#5{%
  \reset@font\fontsize{#1}{#2pt}%
  \fontfamily{#3}\fontseries{#4}\fontshape{#5}%
  \selectfont}%
\fi\endgroup%
\begin{picture}(7899,5049)(1414,-14248)
\put(3376,-10936){\makebox(0,0)[lb]{\smash{{\SetFigFont{10}{12.0}{\rmdefault}{\mddefault}{\updefault}{\color[rgb]{0,0,0}$*$}%
}}}}
\put(3001,-9886){\makebox(0,0)[lb]{\smash{{\SetFigFont{12}{14.4}{\rmdefault}{\mddefault}{\updefault}{\color[rgb]{0,0,0}$\cdot$}%
}}}}
\put(4426,-10411){\makebox(0,0)[lb]{\smash{{\SetFigFont{10}{12.0}{\rmdefault}{\mddefault}{\updefault}{\color[rgb]{0,0,0}$=$}%
}}}}
\put(6751,-10411){\makebox(0,0)[lb]{\smash{{\SetFigFont{10}{12.0}{\rmdefault}{\mddefault}{\updefault}{\color[rgb]{0,0,0}$=$}%
}}}}
\put(3601,-9436){\makebox(0,0)[lb]{\smash{{\SetFigFont{10}{12.0}{\rmdefault}{\mddefault}{\updefault}{\color[rgb]{0,0,0}$*$}%
}}}}
\put(3151,-13261){\makebox(0,0)[lb]{\smash{{\SetFigFont{10}{12.0}{\rmdefault}{\mddefault}{\updefault}{\color[rgb]{0,0,0}$*$}%
}}}}
\put(4276,-13261){\makebox(0,0)[lb]{\smash{{\SetFigFont{10}{12.0}{\rmdefault}{\mddefault}{\updefault}{\color[rgb]{0,0,0}$=$}%
}}}}
\put(6301,-13336){\makebox(0,0)[lb]{\smash{{\SetFigFont{10}{12.0}{\rmdefault}{\mddefault}{\updefault}{\color[rgb]{0,0,0}$*$}%
}}}}
\put(7351,-13261){\makebox(0,0)[lb]{\smash{{\SetFigFont{10}{12.0}{\rmdefault}{\mddefault}{\updefault}{\color[rgb]{0,0,0}$=$}%
}}}}
\put(3697,-12721){\makebox(0,0)[lb]{\smash{{\SetFigFont{12}{14.4}{\rmdefault}{\mddefault}{\updefault}{\color[rgb]{0,0,0}$\cdot$}%
}}}}
\put(5851,-9961){\makebox(0,0)[lb]{\smash{{\SetFigFont{12}{14.4}{\rmdefault}{\mddefault}{\updefault}{\color[rgb]{0,0,0}$\cdot$}%
}}}}
\put(2251,-12811){\makebox(0,0)[lb]{\smash{{\SetFigFont{12}{14.4}{\rmdefault}{\mddefault}{\updefault}{\color[rgb]{0,0,0}$\cdot$}%
}}}}
\end{picture}%

%% file: cpaa_lc3b.pstex_t
\begin{picture}(0,0)%
\includegraphics{cpaa_lc3b.pstex}%
\end{picture}%
\setlength{\unitlength}{3158sp}%
\begingroup\makeatletter\ifx\SetFigFont\undefined%
\gdef\SetFigFont#1#2#3#4#5{%
  \reset@font\fontsize{#1}{#2pt}%
  \fontfamily{#3}\fontseries{#4}\fontshape{#5}%
  \selectfont}%
\fi\endgroup%
\begin{picture}(7974,2349)(2014,-11548)
\put(3376,-10936){\makebox(0,0)[lb]{\smash{{\SetFigFont{10}{12.0}{\rmdefault}{\mddefault}{\updefault}{\color[rgb]{0,0,0}$*$}%
}}}}
\put(3001,-9886){\makebox(0,0)[lb]{\smash{{\SetFigFont{12}{14.4}{\rmdefault}{\mddefault}{\updefault}{\color[rgb]{0,0,0}$\cdot$}%
}}}}
\put(4426,-10411){\makebox(0,0)[lb]{\smash{{\SetFigFont{10}{12.0}{\rmdefault}{\mddefault}{\updefault}{\color[rgb]{0,0,0}$=$}%
}}}}
\put(6751,-10411){\makebox(0,0)[lb]{\smash{{\SetFigFont{10}{12.0}{\rmdefault}{\mddefault}{\updefault}{\color[rgb]{0,0,0}$= 0 =$}%
}}}}
\put(3601,-9436){\makebox(0,0)[lb]{\smash{{\SetFigFont{10}{12.0}{\rmdefault}{\mddefault}{\updefault}{\color[rgb]{0,0,0}$*$}%
}}}}
\put(5851,-9961){\makebox(0,0)[lb]{\smash{{\SetFigFont{12}{14.4}{\rmdefault}{\mddefault}{\updefault}{\color[rgb]{0,0,0}$\cdot$}%
}}}}
\put(9226,-10486){\makebox(0,0)[lb]{\smash{{\SetFigFont{10}{12.0}{\rmdefault}{\mddefault}{\updefault}{\color[rgb]{0,0,0}$*$}%
}}}}
\put(9772,-9946){\makebox(0,0)[lb]{\smash{{\SetFigFont{12}{14.4}{\rmdefault}{\mddefault}{\updefault}{\color[rgb]{0,0,0}$\cdot$}%
}}}}
\put(8326,-10036){\makebox(0,0)[lb]{\smash{{\SetFigFont{12}{14.4}{\rmdefault}{\mddefault}{\updefault}{\color[rgb]{0,0,0}$\cdot$}%
}}}}
\end{picture}%

%% file: cpad0.pstex_t
\begin{picture}(0,0)%
\includegraphics{cpad0.pstex}%
\end{picture}%
\setlength{\unitlength}{3158sp}%
\begingroup\makeatletter\ifx\SetFigFont\undefined%
\gdef\SetFigFont#1#2#3#4#5{%
  \reset@font\fontsize{#1}{#2pt}%
  \fontfamily{#3}\fontseries{#4}\fontshape{#5}%
  \selectfont}%
\fi\endgroup%
\begin{picture}(5355,1974)(586,-4273)
\put(601,-3136){\makebox(0,0)[lb]{\smash{{\SetFigFont{11}{13.2}{\rmdefault}{\mddefault}{\updefault}{\color[rgb]{0,0,0}$\del$}%
}}}}
\put(3226,-3136){\makebox(0,0)[lb]{\smash{{\SetFigFont{11}{13.2}{\rmdefault}{\mddefault}{\updefault}{\color[rgb]{0,0,0}$=$}%
}}}}
\put(4351,-4186){\makebox(0,0)[lb]{\smash{{\SetFigFont{11}{13.2}{\rmdefault}{\mddefault}{\updefault}{\color[rgb]{0,0,0}$j$}%
}}}}
\put(4951,-4186){\makebox(0,0)[lb]{\smash{{\SetFigFont{11}{13.2}{\rmdefault}{\mddefault}{\updefault}{\color[rgb]{0,0,0}$l$}%
}}}}
\put(5926,-3136){\makebox(0,0)[lb]{\smash{{\SetFigFont{11}{13.2}{\rmdefault}{\mddefault}{\updefault}{\color[rgb]{0,0,0}$+ \ \ \ \dots$}%
}}}}
\end{picture}%

%% file: cpad1.pstex_t
\begin{picture}(0,0)%
\includegraphics{cpad1.pstex}%
\end{picture}%
\setlength{\unitlength}{3552sp}%
\begingroup\makeatletter\ifx\SetFigFont\undefined%
\gdef\SetFigFont#1#2#3#4#5{%
  \reset@font\fontsize{#1}{#2pt}%
  \fontfamily{#3}\fontseries{#4}\fontshape{#5}%
  \selectfont}%
\fi\endgroup%
\begin{picture}(6999,1599)(1939,-4048)
\end{picture}%

%% file: pairA.pstex_t
\begin{picture}(0,0)%
\includegraphics{pairA.pstex}%
\end{picture}%
\setlength{\unitlength}{3552sp}%
\begingroup\makeatletter\ifx\SetFigFont\undefined%
\gdef\SetFigFont#1#2#3#4#5{%
  \reset@font\fontsize{#1}{#2pt}%
  \fontfamily{#3}\fontseries{#4}\fontshape{#5}%
  \selectfont}%
\fi\endgroup%
\begin{picture}(6237,4974)(2101,-9298)
\put(2851,-7936){\makebox(0,0)[lb]{\smash{{\SetFigFont{12}{14.4}{\rmdefault}{\mddefault}{\updefault}{\color[rgb]{0,0,0}$\odot$}%
}}}}
\put(7576,-5311){\makebox(0,0)[lb]{\smash{{\SetFigFont{12}{14.4}{\rmdefault}{\mddefault}{\updefault}{\color[rgb]{0,0,0}$\odot$}%
}}}}
\put(4651,-7936){\makebox(0,0)[lb]{\smash{{\SetFigFont{12}{14.4}{\rmdefault}{\mddefault}{\updefault}{\color[rgb]{0,0,0}$\odot$}%
}}}}
\put(4876,-5297){\makebox(0,0)[lb]{\smash{{\SetFigFont{12}{14.4}{\rmdefault}{\mddefault}{\updefault}{\color[rgb]{0,0,0}$\odot$}%
}}}}
\put(5920,-5298){\makebox(0,0)[lb]{\smash{{\SetFigFont{12}{14.4}{\rmdefault}{\mddefault}{\updefault}{\color[rgb]{0,0,0}$\odot$}%
}}}}
\put(6301,-7916){\makebox(0,0)[lb]{\smash{{\SetFigFont{12}{14.4}{\rmdefault}{\mddefault}{\updefault}{\color[rgb]{0,0,0}$\odot$}%
}}}}
\end{picture}%

%% file: cpda1.pstex_t
\begin{picture}(0,0)%
\includegraphics{cpda1.pstex}%
\end{picture}%
\setlength{\unitlength}{3552sp}%
\begingroup\makeatletter\ifx\SetFigFont\undefined%
\gdef\SetFigFont#1#2#3#4#5{%
  \reset@font\fontsize{#1}{#2pt}%
  \fontfamily{#3}\fontseries{#4}\fontshape{#5}%
  \selectfont}%
\fi\endgroup%
\begin{picture}(4120,1749)(1714,-5473)
\end{picture}%

%% file: cpdd.pstex_t
\begin{picture}(0,0)%
\includegraphics{cpdd.pstex}%
\end{picture}%
\setlength{\unitlength}{3552sp}%
\begingroup\makeatletter\ifx\SetFigFont\undefined%
\gdef\SetFigFont#1#2#3#4#5{%
  \reset@font\fontsize{#1}{#2pt}%
  \fontfamily{#3}\fontseries{#4}\fontshape{#5}%
  \selectfont}%
\fi\endgroup%
\begin{picture}(1824,1824)(2389,-3373)
\end{picture}%

%% file: cpdd1.pstex_t
\begin{picture}(0,0)%
\includegraphics{cpdd1.pstex}%
\end{picture}%
\setlength{\unitlength}{3552sp}%
\begingroup\makeatletter\ifx\SetFigFont\undefined%
\gdef\SetFigFont#1#2#3#4#5{%
  \reset@font\fontsize{#1}{#2pt}%
  \fontfamily{#3}\fontseries{#4}\fontshape{#5}%
  \selectfont}%
\fi\endgroup%
\begin{picture}(6462,3099)(2251,-4948)
\put(2251,-2386){\makebox(0,0)[lb]{\smash{{\SetFigFont{12}{14.4}{\rmdefault}{\mddefault}{\updefault}{\color[rgb]{0,0,0}$\del$}%
}}}}
\end{picture}%

%% file: pairD.pstex_t
\begin{picture}(0,0)%
\includegraphics{pairD.pstex}%
\end{picture}%
\setlength{\unitlength}{3552sp}%
\begingroup\makeatletter\ifx\SetFigFont\undefined%
\gdef\SetFigFont#1#2#3#4#5{%
  \reset@font\fontsize{#1}{#2pt}%
  \fontfamily{#3}\fontseries{#4}\fontshape{#5}%
  \selectfont}%
\fi\endgroup%
\begin{picture}(5637,4449)(1951,-8773)
\put(4201,-4936){\makebox(0,0)[lb]{\smash{{\SetFigFont{12}{14.4}{\rmdefault}{\mddefault}{\updefault}{\color[rgb]{0,0,0}$\del$}%
}}}}
\put(2251,-5311){\makebox(0,0)[lb]{\smash{{\SetFigFont{12}{14.4}{\rmdefault}{\mddefault}{\updefault}{\color[rgb]{0,0,0}$\del$}%
}}}}
\put(4697,-5291){\makebox(0,0)[lb]{\smash{{\SetFigFont{12}{14.4}{\rmdefault}{\mddefault}{\updefault}{\color[rgb]{0,0,0}$\odot$}%
}}}}
\put(6621,-5441){\makebox(0,0)[lb]{\smash{{\SetFigFont{12}{14.4}{\rmdefault}{\mddefault}{\updefault}{\color[rgb]{0,0,0}$\odot$}%
}}}}
\put(2810,-7316){\makebox(0,0)[lb]{\smash{{\SetFigFont{12}{14.4}{\rmdefault}{\mddefault}{\updefault}{\color[rgb]{0,0,0}$\odot$}%
}}}}
\put(4903,-7541){\makebox(0,0)[lb]{\smash{{\SetFigFont{12}{14.4}{\rmdefault}{\mddefault}{\updefault}{\color[rgb]{0,0,0}$\odot$}%
}}}}
\end{picture}%